\documentclass[a4paper]{article}
\usepackage[a4paper,top=3cm,bottom=3cm,left=2.1cm,right=2.1cm,marginparwidth=1.75cm]{geometry}
\usepackage[utf8]{inputenc}
\usepackage{bm}
\usepackage{amsmath, amsfonts, amssymb, amsthm}
\usepackage{bbm}
\usepackage[justification=centering]{caption}
\usepackage{mathrsfs, dsfont}
\usepackage[dvipsnames]{xcolor}
\usepackage{graphicx}
\usepackage{float}
\usepackage{subfig}
\usepackage{bigints}

\usepackage[colorlinks=true,allcolors=blue]{hyperref}
\usepackage{xcolor}
\usepackage{xparse} 
\usepackage{hyperref}
\usepackage{multirow}
\usepackage{parskip}

\usepackage{tasks}
\settasks{style=itemize}

\usepackage[round]{natbib}
\usepackage{pifont}
\newtheorem{theorem}{Theorem}[section]
\newtheorem{lemma}[theorem]{Lemma}
\newtheorem{definition}[theorem]{Definition}

\newtheorem{proposition}[theorem]{Proposition}

\newtheorem{corollary}[theorem]{Corollary}

\theoremstyle{remark}
\newtheorem{remark}{Remark}[theorem]
\newtheorem{example}{Example}[theorem]

\numberwithin{equation}{section}

\newenvironment{sqremark}{\begin{remark}}{\hfill \tiny $\blacksquare$ \end{remark}}
\newenvironment{sqexample}{\begin{example}}{\hfill \tiny $\blacksquare$ \end{example}}

\newcommand{\R}{\mathbb{R}}
\newcommand{\N}{\mathbb{N}}

\newcommand{\E}{\mathbb{E}}
\newcommand{\X}{\mathbb{X}}
\def\P{\mathbb{P}}
\newcommand\Q{\mathbb{Q}}
\newcommand\F{\mathcal{F}}

\newcommand{\alphabet}[1][d]{{A}_{#1}}
\newcommand{\TA}[1][d]{T(\R^{#1})}
\newcommand{\eTA}[1][d]{T((\R^{#1}))}
\newcommand{\tTA}[2][d]{T^{#2}(\R^{#1})}

\newcommand{\emptyword}{{\color{NavyBlue}{\mathbf{\varnothing}}}}
\newcommand{\word}[1]{{\color{NavyBlue}{\mathbf{#1}}}}

\newcommand{\proj}[1]{|_{\word{#1}}}

\newcommand{\conpow}[1]{^{\otimes #1}}

\newcommand{\shuprod}{\mathrel{\sqcup \mkern -3.2mu \sqcup}}
\newcommand{\shupow}[1]{^{\shuprod #1}}

\newcommand{\G}{{\bf \Lambda}}

\NewDocumentCommand{\sig}{O{t} O{W}}{\widehat{\mathbb{#2}}_{#1}}
\NewDocumentCommand{\sigE}{O{t} O{W}}{\E[\sig[#1][#2]]}
\NewDocumentCommand{\sigX}{O{t} O{X}}{\mathbb{#2}_{#1}}

\NewDocumentCommand{\bracketsigX}{O{t} O{X} m}{\left \langle #3, \sigX[#1][#2] \right \rangle} 
\NewDocumentCommand{\bracketsig}{O{t} O{W} m}{\left \langle #3, \sig[#1][#2] \right \rangle}   
\NewDocumentCommand{\bracketsigE}{O{t} O{W} m}{\left \langle #3, \sigE[#1][#2] \right \rangle} 

\NewDocumentCommand{\ssig}{O{t} O{W} O{\blambda}}{\widehat{\mathbb{#2}}^{#3}_{#1}}
\NewDocumentCommand{\ssigE}{O{t} O{W} O{\blambda}}{\E[\ssig[#1][#2][#3]]}
\NewDocumentCommand{\Essig}{O{t} O{\blambda}}{\widehat{\mathcal{E}}_{#1}^{#2}}
\NewDocumentCommand{\ssigX}{O{t} O{X} O{\blambda}}{\mathbb{#2}^{#3}_{#1}}
\NewDocumentCommand{\D}{O{h} O{\blambda}}{{\bf D}_{#1}^{#2}}

\NewDocumentCommand{\bracketssigX}{O{t} O{X} O{\blambda} m}{\left \langle #4, \ssigX[#1][#2][#3] \right \rangle} 
\NewDocumentCommand{\bracketssig}{O{t} O{W} O{\blambda} m}{\left \langle #4, \ssig[#1][#2][#3] \right \rangle}   
\NewDocumentCommand{\bracketssigE}{O{t} O{W} O{\blambda} m}{\left \langle #4, \ssigE[#1][#2][#3] \right \rangle} 

\newcommand{\bsigma}{\bm{\sigma}}
\newcommand{\bpsi}{\bm{\psi}}
\newcommand{\brho}{\bm{\rho}}

\newcommand{\blambda}{\bm{\lambda}}
\newcommand{\bell}{\bm{\ell}}
\newcommand{\bp}{\bm{p}}
\newcommand{\bq}{\bm{q}}

\usepackage{mathtools}
\usepackage{authblk}
\mathtoolsset{showonlyrefs}

\usepackage{enumitem}
\usepackage[normalem]{ulem}
\usepackage[textsize=small]{todonotes}

\usepackage[most]{tcolorbox}

\newtcolorbox{crossout}{
    blanker, 
    overlay={
        \draw[red, thick] (frame.north west) -- (frame.south east);
        \draw[red, thick] (frame.north east) -- (frame.south west);
    },
    breakable 
}

\title{Exponentially Fading Memory Signature}

\author[a]{Eduardo Abi Jaber\thanks{eduardo.abi-jaber@polytechnique.edu.  EAJ is  grateful for the financial support from the Chaires FiME-FDD, Financial Risks and Deep Finance \& Statistics at Ecole Polytechnique.}}
\author[a,b]{Dimitri Sotnikov\thanks{{dmitrii.sotnikov@polytechnique.edu. Corresponding author. DS is grateful for the financial support provided by Engie Global Markets.  \\ Both authors thank Paul Hager, Fabian Harang, Luca Pelizzari, Josef Teichmann, and William Turner for fruitful discussions and insightful comments. We also thank the anonymous referee for their careful reading and valuable suggestions, which significantly improved the paper.}}}
\affil[a]{École Polytechnique, CMAP, Route de Saclay
91128 Palaiseau, France}
\affil[b]{Engie Global Markets,  1 Pl. Samuel de Champlain, 92400 Courbevoie, France}

\begin{document}

\maketitle
    \begin{abstract}
    We introduce the exponentially fading memory (EFM) signature, a time-invariant transformation of an infinite (possibly rough) path that serves as a mean-reverting analogue of the classical path signature. We construct the EFM signature via rough path theory, carefully adapted to accommodate improper integration from minus infinity. The EFM signature retains many of the key algebraic and analytical properties of classical signatures, including a suitably modified Chen identity, the linearization property, path-determinacy, and the universal approximation property. From the probabilistic perspective,  the EFM signature provides a “stationarized” representation, making it particularly well-suited for time-series analysis and signal processing overcoming the shortcomings of the standard signature. In particular, the EFM signature of time-augmented Brownian motion evolves as a group-valued Ornstein–Uhlenbeck process. We establish its stationarity, Markov property, and exponential ergodicity in the Wasserstein distance, and we derive an explicit  formula à la Fawcett for its expected value in terms of Magnus expansions. We also study linear combinations of EFM signature elements and  the computation of associated characteristic functions in terms of a mean-reverting infinite dimensional Riccati equation.
\end{abstract}


\section{Introduction}
Fading memory and time invariance are important properties when modeling ``memory'' effects of dynamical systems.  
\textit{Time invariance} means that the output signal $Y_t$ at time $t$ depends only on the continuous input signal $(X_{s})_{-\infty < s \leq t}$ up to time $t$, but not on the absolute time $t$. In other words, a parallel shift in time of the input signal results in an identical shift of the output signal. More precisely, we suppose that there exists a functional $F \colon C(\R_+, \R^d) \to \R$ such that, for all $t \in \R$,
\begin{equation}\label{eq:time_invariance}
    Y_t = F((X_{t - s})_{s \geq 0}).
\end{equation}

Time invariance postulates, in some sense, a stationarity in the relationship between input and output. This is a much weaker condition than the stationarity of the output itself, as it allows for a nonstationary input. Nevertheless, the assumption of stationary dependence is crucial for the application of machine learning and econometric methods in time-series analysis. It is worth noting that, in practice, often only a single realization of a time series or a process trajectory---referred to more generally here as a ``signal''---is available, for instance, in financial data. Within this framework, it is natural to assume that, at any given time $t$, the input signal $X$ is observable from the infinitely distant past $-\infty$ up to $t$. Without this assumption, the notion of time invariance would not make sense.

 \textit{Fading memory} is an additional assumption on the functional $F$, introduced to handle the infinite length of the past in a tractable way. Informally, it means that the influence of the distant past diminishes over time, i.e., that $Y_t$ gradually ``forgets'' the remote history of $X$. This idea dates back to the works of \citet{volterra2005theory} and \citet{wiener1958nonlinear}. \citet{Boyd1985FadingMA} formalize this concept by requiring the functional $F$ to be continuous not with respect to the uniform topology, but with respect to a \textit{weighted} uniform topology, thereby strengthening the notion of continuity. That is, $F$ has fading memory if its output remains close for input paths that are close in the recent past, even if they differ in the distant past.

Moreover, \citet[Theorem 5]{Boyd1985FadingMA} prove a folk theorem stating that any linear time-invariant operator with fading memory admits a convolution representation. In discrete time, this implies that the output can be written as a weighted moving average:
\begin{equation}
    y_n = F((x_{n - k})_{k \geq 0}) = \sum_{k \geq 0} w_k x_{n - k},
\end{equation}
for some sequence of weights $w = (w_k)_{k \geq 0} \in \ell^1$. A particularly popular example of such a weighted moving average is the exponential moving average (EWMA), defined as
\begin{equation}\label{eq:ewma}
    y_n = \sum_{k \geq 0} \rho^k x_{n - k}, \quad \rho \in (0, 1).
\end{equation}
Unlike the general case, the EWMA yields a \emph{Markovian} sequence, as it satisfies $y_{n + 1} = \rho y_n + x_{n + 1}$, which significantly simplifies both analysis and computation. 

However, linear functionals are quite limited in their ability to capture complex memory effects. \citet[Theorem 1]{Boyd1985FadingMA} address the nonlinear case and provide an approximation result for time-invariant fading memory operators using Volterra series. This result is a consequence of the Stone--Weierstrass theorem and thus relies on the algebraic structure of the Volterra series. Nevertheless, as the authors point out, identifying an explicit approximating Volterra series remains a challenging task:
\begin{quote}
    ``While the approximations are certainly strong enough to be useful in applications like macro-modeling of complicated systems or in universal nonlinear system identifiers, we know of no general procedure, based only on input/output measurements, by which an approximation can be found.''
\end{quote}
{For a review of the Volterra series method and the different approaches to estimating Volterra kernels, we refer to \citet*{cheng2017volterra}.}

{Path signatures offer an alternative to Volterra series. Originally introduced by \citet{chen1958integration}, signatures have played an important role in stochastic Taylor expansions (see \citet{kloeden1992stochastic, benarous1989flots}) and have attracted significant attention through the work of Terry Lyons and the development of rough path theory (see \citet*{lyons2007differential}), which has led to a wide range of applications, particularly in machine learning.} We refer the interested reader to the comprehensive surveys by \citet{Chevyrev2016} and \citet{lyons2022signature}. Similar to Volterra series, path signatures possess a rich algebraic structure and enjoy a linearization property, making them powerful tools for functional approximation. A key advantage of path signatures, compared to Volterra series, lies in their computational tractability, made possible by the celebrated \cite{chen1958integration} identity.

When applied to time series or stochastic processes, the standard approach treats path signatures as features that efficiently summarize the information contained in the signal; see, e.g., \citet*{levin2016learningpastpredictingstatistics, morrill2020generalised}. Recent works by \citet*{cuchiero2022theocalib} and \citet{abijaber2024signature} propose an alternative perspective, where linear functionals on path signatures are studied as stochastic processes themselves. This viewpoint allows for explicit representations of path-dependent processes as linear functionals of the time-augmented Brownian motion signature, as shown in \citet*{jaber2024pathdependentprocessessignatures}.

However, this framework is not naturally suited for modeling time-invariant dependencies as defined in \eqref{eq:time_invariance}. In particular, one cannot directly compute the signature of an infinite path. To incorporate information from the entire past, a natural model would be
\begin{equation}\label{eq:sig_model}
    Y_t = \bracketsigX[0, t][\widehat{X}]{\bell}, \quad t \geq 0,
\end{equation}
where $\sigX[0, t][\widehat{X}]$ denotes the signature of the time-augmented path $\widehat{X}_t = (t, X_t)$ over $[0, t]$.

This model is clearly not time-invariant and can easily lead to overfitting, as illustrated in the examples below. Furthermore, since signatures play a role analogous to polynomials in path space, \eqref{eq:sig_model} can be interpreted as a kind of Taylor expansion of the output process $Y$; see, e.g., \citet*{litterer2014chen, kloeden1992stochastic}. While such expansions offer excellent local approximation properties, they may fail globally: a model trained on an interval $[0, T_1]$ may diverge when applied on a longer interval $[0, T_2]$ with $T_2 > T_1$, as reported in \citet{abijaber2024signature}. This behavior is undesirable in the context of modeling time-invariant systems \eqref{eq:time_invariance}.

One could ``stationarize'' the signature model \eqref{eq:sig_model} to make it time-invariant by computing the path signature over a rolling window of fixed length $\Delta > 0$, that is, by considering
\begin{equation}
    Y_t = \bracketsigX[t - \Delta, t][\widehat{X}]{\bell}, \quad t \in \R.
\end{equation}
This approach serves as a continuous-time analogue of the expected signature (ES) model introduced by \citet*{levin2016learningpastpredictingstatistics}. However, this modification introduces several limitations. First, it assumes that the future depends only on a fixed portion of the past, abruptly excluding all information that falls outside the rolling window.
Second, although this model enforces time invariance and can be handled efficiently from a numerical perspective, its dynamics remain analytically complicated: the moving lower limit $t - \Delta$ in the signature computation leads to the delayed equation
\[
    d\sigX[t - \Delta, t]
    =
    \sigX[t - \Delta, t] \otimes \circ dX_t
    -
    \circ dX_{t - \Delta} \otimes \sigX[t - \Delta, t],
\]
and prevents one from obtaining a simple semimartingale decomposition for linear functionals applied to it.

In this paper, we aim to construct an object that combines the best features from three different approaches:
\begin{enumerate}
    \item Time-invariance of Volterra series,
    \item Algebraic properties, approximation power, and numerical tractability of signatures,
    \item Markovian structure of EWMA.
\end{enumerate}

To achieve this, we introduce the \textit{exponentially fading memory signature} (EFM signature for short), defined as follows: for an $\R^d$-valued continuous  path  $X = (X_t)_{t \in \R}$ and a vector $\blambda \in \R^d$ with positive entries, the fading-memory signature is defined as {a collection of iterated integrals of different orders $n$
\begin{equation}\label{eq:statsig_intro}
        \ssigX[t][X][\blambda, \word{i_1 \cdots i_n}] := \int_{-\infty < u_1 < \cdots < u_n < t} e^{-\blambda^{i_1}(t - u_1)} dX_{u_1}^{i_1}  \ldots  e^{-\blambda^{i_n}(t - u_n)} dX_{u_n}^{i_n}, \quad t \in\R,
    \end{equation}
where $i_1, \ldots, i_n \in \{1, \ldots, d\}$.}

In particular, we will show that when the input signal is a Brownian motion, its EFM signature becomes a stationary $\eTA$-valued Markov process, which can be treated with standard tools of stochastic calculus.

The well-posedness of the iterated integrals is ensured by the exponential weights. These weights maintain the same order of magnitude for each component of the EFM signature as $t$ evolves, which is crucial for many applications.
Moreover, the exponential form is crucial for preserving the algebraic structure of signatures, in particular, the Chen's identity and the shuffle property. The same reason explains the choice of the argument $(t - u_j)$ in the terms $e^{-\blambda(t - u_j)} \odot dX_{u_j}$, rather than the choice $(u_{j+1} - u_j)$, which would correspond to a particular case of the Volterra signature introduced by \citet{Harang2021}. Indeed, $(t - u_j)$ guarantees that the mean-reversion rate increases with the signature level, which is necessary for the multiplicative structure. In contrast, in the Volterra case with $(u_{j+1} - u_j)$, neither Chen's identity nor the shuffle property hold. For this reason, \cite{Harang2021} propose replacing the tensor product with a convolution product to preserve multiplicativity, which makes the computation of the Volterra signature more intricate.

The importance of ``forgetting'' --- that is, relying more heavily on recent information --- was highlighted and leveraged by \citet*{toth2024learningforgetbayesiantime}, who proposed a memory mechanism similar to ours to compute random Fourier signature features, and developed a Bayesian framework for sequence-to-sequence regression problems.
We also refer to the work of \citet{corcuera2024fractionalsignaturegeneralisationsignature}, which introduces fractional signatures and demonstrates improved performance when memory is incorporated into the signature.

A discretized version of the EFM signature can be seen as a particular case of the general weighted iterated sum introduced by \citet{diehl2024fruits} for an appropriately chosen weight function. {This work examines the application of the signature approach to time series, where it has been demonstrated to perform well in classification tasks.} However, the authors do not introduce this particular form and do not leverage its algebraic and analytical properties.

{We develop a framework for the EFM signature, establishing its principal algebraic, analytic, and probabilistic properties. We construct the EFM signature via rough path theory, carefully adapted to accommodate improper integration from minus infinity (Theorem~\ref{thm:fm_sig_existence}) and  we prove the fading memory property (Theorem~\ref{thm:rough_fm}). We show that the EFM signature preserves many of the essential features of classical signatures—such as a modified Chen identity (Proposition~\ref{prop:Chen_FM}), the linearization property (Proposition~\ref{prop:shuffle}), uniqueness (Proposition~\ref{prop:uniqueness}), the universal approximation theorem (Theorem~\ref{thm:UAT}), and moment determinacy (Proposition~\ref{prop:momentss})—while naturally adapting them to a mean-reverting, time-invariant setting. The EFM signature of piece-wise linear path can be computed explicitly (Theorem~\ref{thm:linear_path_magnus}), which allows for an efficient numerical computation and simulation. We pay particular attention to the case of time-augmented Brownian motion, for which the EFM signature evolves as a group-valued Ornstein–Uhlenbeck process. In this context, we prove its Markov property (Theorem~\ref{thm:Markov}), stationarity (Theorem~\ref{thm:stationarity}), and exponential ergodicity under the Wasserstein distance (Theorem~\ref{thm:ergodicity}), and we derive explicit formulas for its expected EFM signature in the spirit of Fawcett’s formula via Magnus expansions (Theorem~\ref{thm:expected_sig}). Finally, we investigate linear functionals of the EFM signature and establish an Itô-type decomposition (Proposition~\ref{prop:ito_general}) along with a representation of the associated characteristic function via an infinite-dimensional, mean-reverting Riccati equation (Theorem~\ref{thm:char}).
}

\paragraph{Outline.} Section~\ref{sect:examples} illustrates the motivation for introducing the EFM signature with two simple examples. In Section~\ref{sect:definition}, we introduce the notation, provide a definition of the fading-memory signature, and discuss its first properties. Section~\ref{sect:algebraic_prop} establishes the algebraic and analytical properties of the EFM signature for the general semimartingale path. In Section~\ref{sect:SSig_BM}, we treat the case of the Brownian motion path and study its probabilistic properties. Section~\ref{sect:linear_forms} is dedicated to the linear functionals of the EFM signature of Brownian motion. Finally, in Section~\ref{sect:rough_paths}, we develop the rough path tools necessary to treat the EFM signature and we establish the integration result and the fading memory property.

\section{Motivating examples}\label{sect:examples}

To illustrate shortcomings of the standard signature model \eqref{eq:sig_model} and the relevance of our approach, we consider two toy examples.

\subsection{Signature representation of an Ornstein--Uhlenbeck process}\label{ex:OU_repr} 

 Consider an Ornstein--Uhlenbeck signal $Y$ given by  
\begin{equation}
    d Y_t = -\mu Y_t \, d t + dW_t,  \quad t \in \R, \quad Y_0 \in \mathbb R.
\end{equation}  
It can be written as an infinite linear combination of signature elements of the time extended Brownian motion $(t,W_t)$:
\begin{align}\label{eq:OUsig}
Y_t = Y_0 + Y_0\sum_{k \geq 1}(-\mu)^k\int_{0 < s_1 < \ldots <s_k <t}ds_1\ldots ds_k + \sum_{k \geq 1}(-\mu)^{k-1}\int_{0 < s_1 < \ldots <s_k <t}dW_{s_1}\,ds_2\ldots ds_k,
    \end{align}
 as shown by \citet{abijaber2024signature}. In practice, this series should be truncated at some order $N$. For the numerical example, we chose $N = 10$. Similarly, we will show in Example~\ref{ex:OU_repr_revisited} that $Y$ can be approximated by a linear functional \eqref{eq:OU_approx} of the EFM signature, also truncated at order $N = 10$.  

The trajectories of these linear combinations are shown in Figure~\ref{fig:sig_repr}, the representation \eqref{eq:OUsig} using the standard signature on the left, and \eqref{eq:OU_approx} using the EFM signature on the right. We clearly observe the explosive behavior of the standard signature when $T$ is far from zero on the left graph, exactly as in the case of the Taylor expansion of real-valued functions, which is accurate in a neighborhood of the expansion point. On the contrary, the EFM signature on the right-hand side provides a uniform global approximation independent of the time interval length overcoming the explosive problem of the standard signature.

\begin{figure}[H]
    \begin{center}
    \includegraphics[width=1\linewidth]{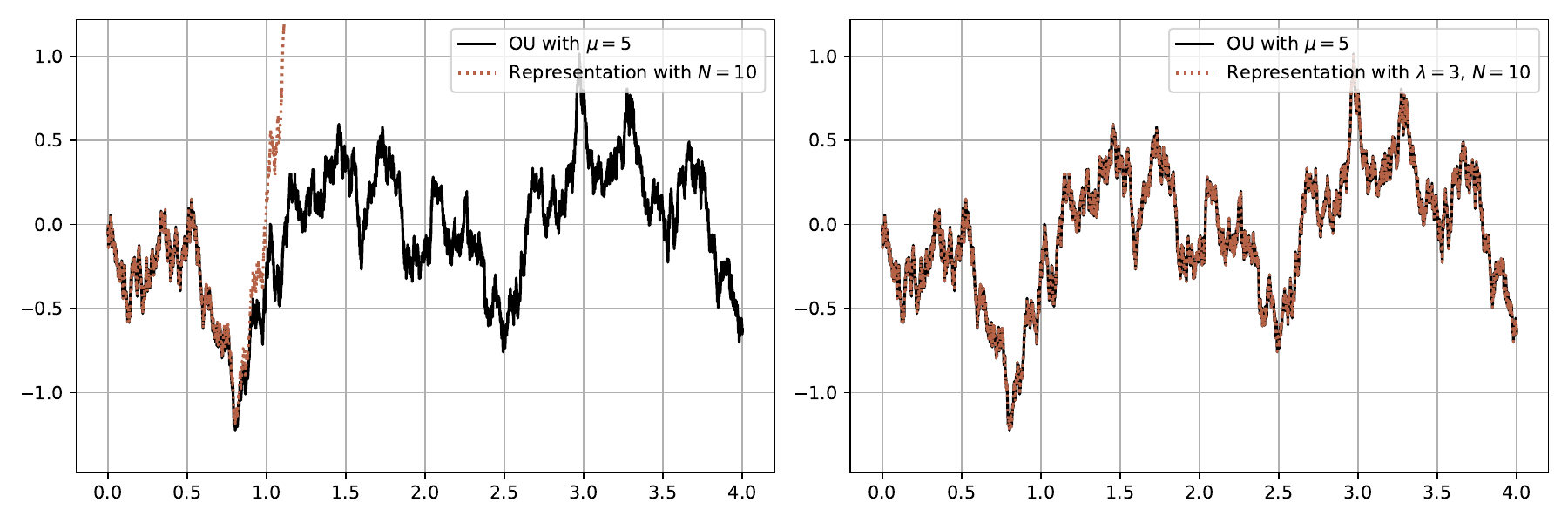}
    \caption{Representation of an OU process as a linear functional of the signature (on the left) and the EFM signature (on the right) of the time-augmented Brownian motion. The parameter of the EFM signature is $\lambda = 3$, and the mean-reversion parameter of the OU process is $\mu = 5$. Both signatures were truncated at order $N = 10$.  
}
    \label{fig:sig_repr}
    \end{center}
\end{figure}

\subsection{Learning stochastic dynamics with signatures}
In the second example, we compare regression models based on the signature and the EFM signature to fit the dynamics of a stationary SDE solution.  

Consider a signal $Y$ following the dynamics:
\begin{equation}\label{eq:langevin}
    dY_t = -\mu Y_t^p\, dt + dW_t, \quad t \in \R,
\end{equation}
for some $\mu > 0$ and an odd $p \in \N$. Our goal is to learn the dynamics \eqref{eq:langevin} by observing the signal $Y$ and the Brownian motion $W$, using only a single trajectory of the process. Given a time grid $0 = t_0 < t_1 < \ldots < t_N = T$, we formulate the following regression problem:
\begin{equation}
    \min_{\bell \in \tTA[2]{N}}\left\{ \dfrac{1}{2(N + 1)}\sum_{0 \leq t_i \leq T} \left| Y_{t_i} - \bracketsigX[t_i]{\bell} \right|^2 
    + \alpha \omega \|\bell\|_1 
    + \dfrac{1}{2}\alpha (1 - \omega) \|\bell\|_2^2\right\},
\end{equation}
where $\alpha > 0$ is a regularization parameter, $\omega \in [0, 1]$ is a mixing parameter, and the $\ell^q$-norm for $q \geq 1$ is defined by
\begin{equation}
    \|\bell\|_q := \left(\sum_{n=0}^N \sum_{|\word{v}| = n} |\bell^{\word{v}}|^q\right)^{1/q}, \quad \bell \in \tTA[2]{N},
\end{equation}
{we refer to Section~\ref{subsection:prelimeray} for the precise definitions and notations.}
Three options for the signature $\sigX$ are considered:

\begin{enumerate}
    \item \textbf{Signature of time-augmented Brownian motion $\widehat W_t=(t,W_t)$}: $\sigX = \sig$. An approach implemented by \citet*{cuchiero2022theocalib} and \citet{abijaber2024signature} seems to be the most natural from a theoretical point of view and consists of regressing the signal against the (extended) signature of the underlying Brownian motion. Although this method is able to fit nonstationary dynamics given many trajectory samples, as indicated in the mentioned works, the model quickly overfits when the dynamics is stationary and can hardly reproduce a reasonable behavior on out-of-sample intervals, as shown in Figure~\ref{fig:regression_comparison} (brown trajectory).
    
    \item \textbf{Signature of time-augmented stationary OU process $\widehat U_t=(t,U_t)$}: $\sigX = \sigX[t][\widehat{U}]$, where $U$ is an Ornstein--Uhlenbeck process satisfying  
    $$
    dU_t = -\lambda U_t dt + dW_t, \quad t \in \R,
    $$
    for some $\lambda > 0$. At first glance, this approach seems to be better suited for dealing with stationary processes. For instance, \citet*{cuchiero2023spvix} use the signatures of correlated OU processes for stochastic volatility modeling. However, due to the presence of nonstationary terms corresponding to the $t$-component of the path and the nonstationary nature of the model itself, this approach also demonstrates poor out-of-sample performance (green trajectory in Figure~\ref{fig:regression_comparison}).
    
    \item \textbf{EFM signature of time-augmented Brownian motion}: $\sigX = \ssig$. The third model relies on the EFM signature of the Brownian motion  observed on $(-\infty, T]$ introduced in this paper. As shown in the illustration, this model demonstrates the best out-of-sample performance and remains  remarkably robust and close to the signal process, thanks to the time invariance embedded in the construction of the EFM signature.
\end{enumerate}

For the numerical illustration, we took $N = 6$ and $T = 2$ and simulated a trajectory with a sampling frequency of $3650$ points per unit time. The parameters of $Y$ were chosen as $p = 5$ and $\mu = 10$. For each model, the interval $[0,1]$ was used as the training sample, while the entire interval $[0,T]$ was used to calibrate the hyperparameters $\alpha$, $\omega$, and $\lambda$. The remaining part, $[2,4]$, was not involved in the optimization problem and is provided to evaluate the out-of-sample model performance.

\begin{figure}[H]
    \begin{center}
    \includegraphics[width=0.6\linewidth]{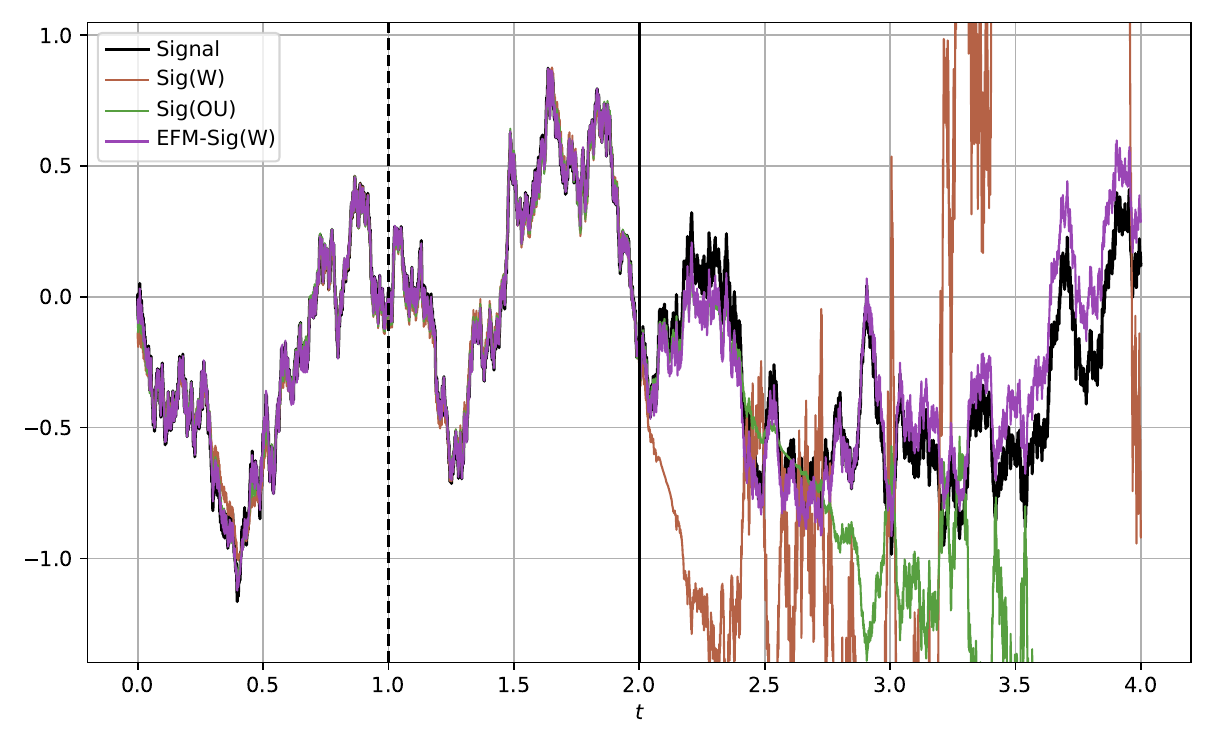}
    \caption{Trajectories predicted by the three regression models fitted to the stationary signal (black) given by \eqref{eq:langevin}: regression on $\sig$ (brown), regression on $\sigX[t][\widehat{U}]$ (green), and regression on $\ssig$ (violet). Black vertical lines indicate the training interval (dotted) and the in-sample interval (solid).
    }
    \label{fig:regression_comparison}
    \end{center}
\end{figure}

\section{{Exponentially fading memory signature and its first properties}}\label{sect:definition}
\paragraph{Notations.} Throughout this paper, $(\Omega, \F, (\F_t)_{t \in \R}, \P)$ will denote a filtered probability space. We will use the notion of increment semimartingales, i.e., processes $X$ such that, for all $s \in \R$, the process $(X_{s + t} - X_s)_{t \geq 0}$ is a semimartingale in the usual sense with respect to $(\F_{s+t})_{t \geq 0}$. In this case, the notion of a stochastic integral with respect to increment semimartingales can be defined, and integration over the whole real line becomes possible. More precisely, if $X = A + M$, where $A$ is a process of bounded variation and $M$ is an increment local martingale, then the stochastic integral
\begin{equation}\label{eq:improper_stochastic_int}
    I_t = \int_{-\infty}^t h_s\, dX_s = \int_{-\infty}^t h_s\, dA_s + \int_{-\infty}^t h_s\, dM_s,
\end{equation}
is well-defined for predictable processes $h$ such that
\[
\int_{-\infty}^t |h_s| \, d|A_s| + \int_{-\infty}^t h_s^2 \, d\langle M \rangle_s < \infty, \quad \text{a.s.}
\]
The first integral in \eqref{eq:improper_stochastic_int} is defined pathwise as a Lebesgue--Stieltjes integral, and the second one is defined as a stochastic increment integral. Moreover, we have $\lim_{t \to -\infty} I_t = 0$ a.s. For details, we refer to the works of \citet*{basseoconnor2010martingaletypeprocessesindexedreal,OConnor2014}. This theory allows us to manipulate stochastic integrals on the real line in exactly the same way as standard stochastic integrals on $\R_+$. With a slight abuse of terminology, we will henceforth use the terms semimartingales and increment semimartingales interchangeably.

An important example of an increment semimartingale we will often use is a two-sided Brownian motion: a centered Gaussian process $W = (W_t)_{t \in \R}$ such that $W_0 = 0$, and for any $s \in \R$ and $t > 0$, the increment $W_{s + t} - W_s$ is distributed as $\mathcal{N}(0, t)$ and is independent of $\F_s$.

For $t_0 \in \R \cup \{-\infty\}$, we will denote the Itô integral by $\int_{t_0}^\cdot Y_t dX_t$ and the Stratonovich integral by $\int_{t_0}^\cdot Y_t \circ dX_t$. If both $X$ and $Y$ are continuous semimartingales, then we have the relation 
$$
\int_{t_0}^\cdot Y_t \circ dX_t = \int_{t_0}^\cdot Y_t dX_t  + \frac{1}{2} \int_{t_0}^\cdot d\langle X,Y\rangle_t.
$$
We denote by $\odot$ the Hadamard product: for two vectors $\bm{x} = (\bm{x}^1, \ldots, \bm{x}^d) \in \R^d$ and $\bm{y} = (\bm{y}^1, \ldots, \bm{y}^d) \in \R^d$, we define 
$$
\bm{x} \odot \bm{y} := (\bm{x}^i \bm{y}^i)_{i = 1, \ldots, d} \in\R^d.
$$
We write $\bm x \prec \bm y$ if $\bm{x}^i < \bm{y}^i$ for $i = 1, \ldots, d$. We use $\lesssim$ to denote an inequality up to a positive constant, and $\lesssim_a$ to denote an inequality up to a positive constant depending on $a$. We will also use $\bm{x}^{\min}$ and $\bm{x}^{\max}$ to denote $\min\limits_{i = 1, \ldots, d} \bm{x}^i$ and $\max\limits_{i = 1, \ldots, d} \bm{x}^i$ respectively.

\subsection{Preliminaries on signatures}\label{subsection:prelimeray}
In this subsection, we recall the necessary notations related to the signatures. For a more detailed introduction of signatures and its properties, we refer to the work of \citet{Chevyrev2016}.

\paragraph{Definition of signature.}
Consider an $\R^d$-valued continuous semimartingale $X = (X_t)_{t \in [0, T]}$ defined on a finite interval $[0, T]$. Its signature over $[s, t] \subset [0, T]$ is nothing else but a collection of \textit{properly ordered} iterated Stratonovich integrals of different orders $n \in \N_0$,
\begin{equation}\label{eq:signature_word}
    \sigX[s, t]^\word{i_1 \cdots i_n} := \int_{s < u_1 < \cdots < u_n < t} \circ d X_{u_1}^{i_1} \cdots \circ d X_{u_n}^{i_n} = \int_s^t \sigX[s, u_n]^\word{i_1 \cdots i_{n-1}} \circ d X_{u_n}^{\word{i_n}},
\end{equation}
where $i_1, \ldots, i_n \in \{1, \ldots, d\}$. The iterated integrals of order $n \in \N_0$ are then the elements of $(\R^d)^{\otimes n}$. Let $\{e_1, \ldots, e_d\}$ denote the canonical basis of $\R^d$. Then, the basis of $(\R^d)^{\otimes n}$ is given by the vectors  
$$
e_{i_1} \otimes \ldots \otimes e_{i_n},\quad i_1, \ldots, i_n \in \{1, \ldots, d\},
$$ 
identified with words $\word{v} = \word{i_1 \ldots i_n}$ of length $|\word{v}| = n$ over the alphabet $\alphabet = \{\word{1}, \ldots, \word{d}\}$. This explains the indexation in \eqref{eq:signature_word}. For the particular case $n = 0$, we set $(\R^d)^{\otimes 0} \cong \R$, and denote the empty word corresponding to its basis vector by $\emptyword$. By convention, we set $ \sigX^\word{\emptyword} \equiv 1$.

This allows us to write the \textit{$n$-th level of signature} $\sigX[s, t]^n = (\sigX[s, t]^{\word{v}})_{|\word{v}| = n} \in (\R^d)^{\otimes n}$ shortly as 
\begin{equation}\label{eq:sig_definition}
    \sigX[s, t]^n := \int_{s < u_1 < \cdots < u_n < t} \circ d X_{u_1} \otimes \cdots \otimes \circ d X_{u_n}.
\end{equation}

The \textit{path signature} of $X$ is then given by a sequence of all signature levels, also called a {tensor sequence}:
\begin{equation}
    \sigX[s, t] := (1, \sigX[s, t]^1, \dots, \sigX[s, t]^n, \dots),
\end{equation}
which is an element of the \textit{extended tensor algebra}:
\begin{equation}
    \eTA := \left\{\bell = (\bell^n)_{n \geq 0} \colon \ \bell^n \in (\R^d)^{\otimes n} \right\}.
\end{equation}
We will also denote $\sigX := \sigX[0, t]$.

The signature plays a similar role to polynomials on path-space. Indeed, in dimension $d=1$, the signature $\sigX[s, t]$ of $X$ is the sequence of monomials $\left( \frac{1}{n!} (X_t - X_s)^n \right)_{n \in \N}$.

\begin{sqexample}
    In the case $d=2$ and $X_t = \widehat{W}_t := (t, W_t)$ where $W$ is a 1-dimensional Brownian motion, the first orders of signature are given by
\begin{equation}\label{eq:sig_example_bm}
    \sig^0 = 1,
    \quad
    \sig^1 =
    \begin{pmatrix}
        t \\
         W_t
    \end{pmatrix},
    \quad
    \sig^2 =
    \begin{pmatrix}
        \frac{t^2}{2!} & \int_0^t s d W_s \\
        \int_0^t W_s ds & \frac{W_t^2}{2!}
    \end{pmatrix}
\end{equation}
\end{sqexample}
\paragraph{Tensor algebra.}
The basis of $\eTA$ is given by all the words $\word{v}$ over the alphabet $\alphabet$, so that each element $\bell \in \eTA$ can be represented as 
\begin{equation}
    \bell = \sum_{n \geq 0} \sum_{|\word{v}| = n} \bell^{\word{v}} \word{v}.
\end{equation}

We also define the \textit{truncated tensor algebra} of order $N \in \N$ by
\begin{equation}
    \tTA{N} := \left\{ \bell \in \eTA \colon \ \bell^n = 0, \ n > N \right\},
\end{equation}
and the \textit{tensor algebra} $\TA := \bigcup_{N \geq 0} \tTA{N}$ as the algebra containing only finite tensor sequences. Similarly, the \textit{truncated signature} of order $N \in \N$ is defined by
\begin{align} \label{def:sig-trunc}
    \sigX[s, t]^{\leq N} := \left( 1, \sigX[s, t]^1, \dots, \sigX[s, t]^N, 0, \dots, 0, \dots \right).
\end{align}

Along with the natural sum and multiplication by constants for the elements of $\eTA$, we define the tensor product of $\bp, \bq \in \eTA$ as
\begin{equation}
    \bp \otimes \bq = \left( \sum_{k = 0}^n \bp^k \otimes\bq^{n - k} \right)_{n \geq 0},
\end{equation}
or, in coordinate form,
\begin{equation}
    (\bp \otimes \bq)^{\word{v}} = \sum_{\substack{\word{u}, \word{w} \\ \word{uw} = \word{v}}} \bp^{\word{u}} \bq^{\word{w}},
\end{equation}
where $\word{uw}$ is the concatenation of the words $\word{u}$ and $\word{w}$, identified with the tensor product of vectors corresponding to $\word{u}$ and $\word{w}$.

We define the bracket $\langle \cdot, \cdot \rangle: \TA \times \eTA \to \R$ by
\begin{equation}
    \langle \bp, \bq \rangle = \sum_{n \geq 0} \sum_{|\word{v}| = n} \bp^{\word{v}} \bq^{\word{v}}.
\end{equation}

\subsection{Exponentially fading memory signature}

We now consider an $\R^d$-valued increment semimartingale $X = (X_t)_{t \in \R}$ and a vector $\blambda = (\blambda^{1}, \ldots, \blambda^{d}) \in \R^d$ with positive entries, i.e., $\blambda \succ 0$. 
We will denote $\Delta := \{(s, t) \in \R^2 \colon\ s \leq t\}$.

To define the EFM signature over the infinite time interval, we introduce exponential damping of the integrators $\circ dX_u^i$ by weights $e^{-\blambda^i(t - u)}$, for some positive constants $\blambda^i > 0$.
\begin{definition} \label{def:sig}
    The {$\blambda$-exponential fading memory} signature of $X$ is defined by  
    \begin{align*}
        \ssigX[s, t]\colon \Omega \times \Delta &\to \eTA \\
        (\omega, s, t) &\mapsto \ssigX[s, t](\omega) := (1, \ssigX[s, t][X][\blambda, 1](\omega), \ldots, \ssigX[s, t][X][\blambda, n](\omega), \ldots),
    \end{align*}
    where
    \begin{equation}\label{eq:fm_sig_def}
        \ssigX[s, t][X][\blambda, n] := \int_{s < u_1 < \cdots < u_n < t} e^{-\blambda(t - u_1)} \odot dX_{u_1} \circ \otimes \cdots \circ \otimes e^{-\blambda(t - u_n)} \odot dX_{u_n}
    \end{equation}
    takes values in $(\R^d)^{\otimes n}$, for $n \geq 0$. We also denote $\ssigX := \ssigX[-\infty, t]$.
\end{definition}

\begin{sqremark} \label{rmk:sig_iteration_def}
    Explicitly, the components of $\X_{s,t}^{\blambda, n}$ are given by
    \begin{equation}
        \ssigX[s, t][X][\blambda, \word{i_1 \cdots i_n}] := \int_{s < u_1 < \cdots < u_n < t} e^{-\blambda^{i_1}(t - u_1)} dX_{u_1}^{i_1} \circ \cdots \circ  e^{-\blambda^{i_n}(t - u_n)} dX_{u_n}^{i_n}.
    \end{equation}
    They can also be expressed in an iterated form:
    \begin{align} \label{eq:def_stat_sig_coef} 
        \sigX[s, t]^{\blambda, \word{i_1 \cdots i_n}} = \int_{s}^t e^{-\blambda^{\word{i_1 \cdots i_n}} (t - u_n)} \sigX[s, u_n]^{\blambda, \word{i_1 \cdots i_{n-1}}} \circ dX_{u_n}^{i_n},
    \end{align}
    where we set
    \begin{equation} \label{eq:lam_v}
        \blambda^{\word{i_1 \cdots i_n}} := \sum_{k=1}^n \blambda^{i_k}.
    \end{equation}
    In particular, we set $\blambda^{\emptyword} = 0$. It is clear that $\blambda^{\word{uv}} = \blambda^{\word{u}} + \blambda^{\word{v}}$ for all words $\word{u}, \word{v}$.
\end{sqremark}

While the EFM signature over a finite interval $[s, t]$ is well-defined as an iterated Stratonovich integral \eqref{eq:def_stat_sig_coef}, integration over the infinite interval $(-\infty, t]$ requires additional assumptions on the process $X$.
We first provide several examples of the EFM signature, assuming that all the improper integrals are well-defined, and then discuss the rigorous construction of the EFM signature in Subsection~\ref{sect:markovian_lift}. 

\begin{sqexample}
    For a one-dimensional process $X_t$ with $d = 1$, its EFM signature is given by
    \begin{equation}
        \ssigX[t][X][\blambda] = \left(1,\, U_t,\, \frac{U_t^2}{2!},\, \ldots,\, \frac{U_t^n}{n!},\, \ldots\right),
    \end{equation}
    where 
    \begin{equation}\label{eq:OU_def}
        U_t := \int_{-\infty}^t e^{-\blambda^1(t - s)}\, dX_s.
    \end{equation}
    In particular, if $X = W$ is a Brownian motion, its EFM signature is nothing other than the powers of the corresponding stationary Ornstein--Uhlenbeck process with mean-reversion coefficient $\blambda^1$. In the general case, the process $U$ can be interpreted as a continuous-time version of the exponential moving average of the past increments of $X$.
\end{sqexample}

\begin{sqexample}
    In the case $d = 2$ and $X_t = \widehat{W}_t := (t, W_t)$, the first levels of the signature are given by
    \begin{equation}\label{eq:fm_sig_examle}
        \sig^{\blambda, 0} = 1,
        \quad
        \sig^{\blambda, 1} =
        \begin{pmatrix}
            (\blambda^1)^{-1} \\
            U_t
        \end{pmatrix},
        \quad
        \sig^{\blambda, 2} =
        \begin{pmatrix}
            \frac{(\blambda^1)^{-2}}{2!} & (\blambda^1)^{-1} \int_{-\infty}^t e^{-(\blambda^1 + \blambda^2)(t-s)}\, dW_s \\
            \int_{-\infty}^t e^{-(\blambda^1 + \blambda^2)(t-s)} U_s\, ds & \frac{U_t^2}{2!}
        \end{pmatrix},
    \end{equation}
    where $U = (U_t)_{t \in \R}$ is a stationary Ornstein–Uhlenbeck process defined by
    \begin{equation}
        U_t = \int_{-\infty}^t e^{-\blambda^2(t-s)}\, dW_s.
    \end{equation}
    This is to be compared with the (standard) signature of $\widehat{W}$ given in \eqref{eq:sig_example_bm}.
\end{sqexample}

\begin{sqexample}
The entries $\blambda^i$ of the vector $\blambda$ are hyperparameters that control how quickly the influence of the past of the $i$-th signal component $X^i$ fades, or equivalently, how rapidly the signature elements involving the integrators $dX^i$ revert to their mean. More precisely, the mean-reversion rate of $\sigX^{\blambda, \word{v}}$ is given by $\blambda^{\word{v}}$. Multiple time scales can be explored either by computing higher-order terms of the EFM signature, as shown in the previous example, or by extending the path dimension {duplicating path components and} assigning them different values of $\blambda^i$. For instance, taking the path $\bar X_t = (X_t, X_t)$ and $\blambda = (\blambda^1, \blambda^2)$, we obtain
\begin{equation}
    \sigX^{\blambda, 0} = 1,
    \quad
    \sigX^{\blambda, 1} =
    \begin{pmatrix}
        U_t^{\blambda^1}\\
        U_t^{\blambda^2}
    \end{pmatrix},
    \quad
    \sigX^{\blambda, 2} =
    \begin{pmatrix}
        \frac{(U_t^{\blambda^1})^2}{2!} & \int_{-\infty}^t e^{-(\blambda^1 + \blambda^2)(t-s)}U_s^{\blambda^1}\circ dX_s \\
        \int_{-\infty}^t e^{-(\blambda^1 + \blambda^2)(t-s)} U_s^{\blambda^2}\circ dX_s & \frac{(U_t^{\blambda^2})^2}{2!}
    \end{pmatrix},
\end{equation}
where $U = (U_t^{\blambda^i})_{t \in \R}$ is given by
\begin{equation}
    U_t^{\blambda^i} = \int_{-\infty}^t e^{-\blambda^i(t-s)}\, dX_s.
\end{equation}
This construction allows for explicit representation of Volterra processes $\int_{-\infty}^t K(t-s)\,dX_s$ with 
\begin{equation}\label{eq:sum_exp}
    K(t) = \sum_{k=1}^n c_k e^{-\blambda^k t}, \quad t \geq 0
\end{equation}
and the polynomials of these processes. Due to the remarkable approximation power of kernels of the form \eqref{eq:sum_exp}, this also opens the door to a tractable approximation of nonlinear functionals of Volterra processes with more general kernels, see, e.g., \citet{CarmonaCoutin} and \citet{abi2019multifactor}.
\end{sqexample}

\subsection{Link between the EFM signature and the standard signature}\label{sect:markovian_lift}

In this section, we show that the EFM signature, considered pointwise (i.e., $T$ by $T$ for $T \in \R$), coincides with the standard signature of a modified path $X^T$. Indeed, fix $T \in \R$ and define the path $X^T = (X_t^T)_{t \leq T}$ by
\begin{equation}\label{eq:volterra_lift}
    dX_t^T = e^{-\blambda(T-t)} \odot dX_t, \quad t \leq T.
\end{equation}
We denote by $(\sigX[s, t]^T)_{s \leq t \leq T}$ the classical signature of $X^T$.

\begin{proposition}\label{Prop:Markovian_lift}
For all $s \leq T$, we have
\[
    \ssigX[s, T] = \sigX[s, T]^T.
\]
Moreover, if $X$ admits a semimartingale decomposition $X = A + M$, where $A$ is a process of bounded variation and $M$ is a local martingale such that, for all $i = 1, \ldots, d$,
\begin{equation}\label{eq:lift_assumption}
    \int_{-\infty}^T e^{-\blambda^i(T-u)} \, d|A_u^i|
    + \int_{-\infty}^T e^{-2\blambda^i(T-u)} \, d\langle M^i \rangle_u
    < \infty,
\end{equation}
then there exist a filtration $(\F_s^{Z^T})_{s \in [0,1]}$ and a continuous semimartingale $Z^T = (Z_s^T)_{s \in [0,1]}$ with respect to this filtration such that
\begin{equation*}
    \ssigX[T] = \mathbb{Z}_{0,1}^T.
\end{equation*}
\end{proposition}
\begin{proof}
    The first part of the proposition follows immediately from the definition of the EFM signature.

    To define the EFM signature over the infinite interval $(-\infty, T]$, we first observe that the process
    \[
        X_t^{T} = \int_{-\infty}^t e^{-\blambda(T-u)} \odot dX_u,
    \]
    is well-defined under assumption \eqref{eq:lift_assumption}, satisfies \eqref{eq:volterra_lift}, and verifies $X_{-\infty}^T = \lim\limits_{t \to -\infty} X_t^{T} = 0$ a.s.
    
    We now define the time-changed process
    \[
        Z_s^T := X_{\log(s) + T}^T,
    \]
    which is a continuous semimartingale on $(\varepsilon, 1]$ with respect to the filtration $\F_s^{Z^T} = \F_{\log s + T}$ for all $\varepsilon \in (0,1)$; see \cite[Chapter V, Proposition 1.4 and Proposition 1.5]{revuz1999ContinuousMartingalesBrownian}. Moreover, $Z_0^T := \lim_{s \to 0} Z_s^T = 0$ a.s., so that $Z^T$ is a well-defined continuous semimartingale on $[0,1]$ satisfying
    \begin{equation}\label{eq:lamperti}
        Z_s^T = \int_0^s r^{\blambda} \odot dX_{\log r + T}.
    \end{equation}
    
    It remains to observe that
    \[
        \ssigX[T] = \sigX[-\infty, T]^T = \mathbb{Z}_{0,1}^T.
    \]
\end{proof}

From now on, we assume that the path $X$ satisfies condition \eqref{eq:lift_assumption} for all $T \in \R$, which guarantees the well-definedness of the EFM signature.

Proposition~\ref{Prop:Markovian_lift} shows that the EFM signature over the infinite interval $(-\infty, T]$ can be viewed as a classical signature over the finite interval $[0,1]$. This observation considerably simplifies the proofs below, since many properties of the standard signature are directly inherited by the EFM signature. For instance, it follows that the EFM signature $\sigX[t]^{\blambda, \leq N}$, truncated at order $N$, belongs to the standard Lie group $G^N$ corresponding to the Lie algebra $\mathfrak{g}^N$ generated by $\word{1}, \ldots, \word{d}$.

\begin{sqremark}
    The time change \eqref{eq:lamperti}, which allows us to transform the infinite signal $X$ into a finite one, was first studied by \citet{Lamperti62}. Indeed, we observe that the increments of $Z^T$ are precisely the Lamperti transform of the increments of $X$.
\end{sqremark}

It is important to notice here that, although the EFM signature can be represented pointwise as the classical signature of $X^T$ for each $T \in \R$, it is impossible to find a path $Y = (Y_t)_{t \in \R}$ such that
$
    \ssigX = \mathbb{Y}_t,
$
for all $t \in \R$. This becomes clear when one views the EFM signature $\ssigX$ dynamically as a function of $t$.

\subsection{EFM signature dynamics}

In this subsection, we show that the EFM signature satisfies a linear convolution equation analogous to the linear SDE satisfied by the standard signature. Consequently, the EFM signature process $\sigX[t]^{\blambda, \leq N}$, as a Lie group-valued process, may be viewed as a mean-reverting analogue of the signature transform.
In particular, since the signature of $d$-dimensional Brownian motion $\sigX[][W]$ is the Brownian motion on $G^N$ in the sense of \cite[Definition 1.2]{Baudoin}, the EFM signature $\sigX[][W]^{\blambda, \leq N}$ plays the role of the Ornstein--Uhlenbeck process on $G^N$. We return to this analogy in Section~\ref{sect:SSig_BM}.

We start by introducing two operators that are important for understanding the properties of the EFM signature and for the analysis of its dynamics.

\begin{definition}\label{def:G_oper}
    The operator $\G\colon \eTA \to \eTA$ is defined by
    $$
    \G\colon \bell = \sum_{n \geq 0}\sum_{|\word{v}| = n} \bell^{\word{v}}\word{v} \mapsto \G\bell = \sum_{n \geq 0}\sum_{|\word{v}| = n} \blambda^{\word{v}}\bell^{\word{v}}\word{v},
    $$
    where $\blambda^{\word{v}}$ is given by \eqref{eq:lam_v}.
\end{definition}
In other words, $\G$ is a diagonal operator multiplying each entry of the tensor algebra element by the corresponding mean-reversion coefficient. We note that this operator is not continuous with respect to any norm on $\eTA$.

\begin{definition}\label{def:D_operator}
    For $\blambda \in\R^d$ and $h \in\R$, we define the operator $\D\colon \eTA \to \eTA$ by
    \begin{align}\label{eq:D_def}
        \D\colon
        \bell = \sum_{n \geq 0}\sum_{|\word{v}| = n} \bell^{\word{v}}\word{v} \mapsto \D \bell =  \sum_{n \geq 0}\sum_{|\word{v}| = n} e^{-\blambda^{\word{v}} h}  \bell^{\word{v}}\word{v}.
    \end{align}
\end{definition}
\begin{proposition}\label{prop_D_ppties}
    The following properties hold
    \begin{enumerate}
        \item[(i)] $\D$ is a dilation operator multiplying the $i$-th basis vector by $e^{-\blambda^ih}$.
        \item[(ii)] For all $h_1, h_2 \in \R$, $\D[h_1]\D[h_2] = \D[h_1 + h_2]$, so that $(\D[h])_{h \geq 0}$ form a semigroup.
        \item[(iii)] For all $\blambda_1, \blambda_2 \in \R^d$, $\D[h][\blambda_1]\D[h][\blambda_2] = \D[h][\blambda_1 + \blambda_2]$.
        \item[(iv)] For all $c \in \R$, $\D[ch] = \D[h][c\blambda]$.
        \item[(v)] $\D$ is an operator exponential of $-h\G$, i.e.,
        $\D = e^{-h\G} = \sum_{n\geq0}\dfrac{(-h\G)^n}{n!}$.
        \item[(vi)] For $h \in\R$, $\, \dfrac{d}{dh}\D = -\G \D = -\D \G$. 
        \item[(vii)] For $\bp, \bq \in \eTA$, we have $\D(\bp\otimes\bq) = (\D\bp)\otimes(\D\bq)$.
        \item[(viii)] For $\bp, \bq \in \eTA$, we have $\G(\bp\otimes\bq) = (\G\bp)\otimes\bq + \bp\otimes(\G\bq)$. 
    \end{enumerate}
\end{proposition}
\begin{proof}
    The properties {(i)}--{(vi)} follow immediately from the definition. For the last two properties, recall that
    $$
    (\bp \otimes \bq)^{\word{v}} = \sum_{\substack{\word{u}, \word{w} \\ \word{uw} = \word{v}}} \bp^{\word{u}} \bq^{\word{w}},
    $$
    so that 
    $$
    (\D(\bp \otimes \bq))^{\word{v}} = e^{-\blambda^{\word{v}}h}\sum_{\substack{\word{u}, \word{w} \\ \word{uw} = \word{v}}} \bp^{\word{u}} \bq^{\word{w}} = \sum_{\substack{\word{u}, \word{w} \\ \word{uw} = \word{v}}} (e^{-\blambda^{\word{u}}h}\bp^{\word{u}})(e^{-\blambda^{\word{w}}h}\bq^{\word{w}}) = ((\D\bp) \otimes (\D\bq))^{\word{v}},
    $$
    since $\blambda^{\word{v}} = \blambda^{\word{u}} + \blambda^{\word{w}}$ for all $\word{u}, \word{w}$ such that $\word{uw} = \word{v}$. Similarly,
    $$
    (\G(\bp \otimes \bq))^{\word{v}} = \blambda^{\word{v}}\sum_{\substack{\word{u}, \word{w} \\ \word{uw} = \word{v}}} \bp^{\word{u}} \bq^{\word{w}} = \sum_{\substack{\word{u}, \word{w} \\ \word{uw} = \word{v}}} (\blambda^{\word{u}} \bp^{\word{u}}\bq^{\word{w}} +  \blambda^{\word{w}}\bp^{\word{u}}\bq^{\word{w}}) = ((\G\bp) \otimes \bq)^{\word{v}} + (\bp \otimes (\G\bq))^{\word{v}}.
    $$
\end{proof}
\begin{sqremark}
    By Proposition~\ref{prop_D_ppties}, $(\D)_{h \in \R}$ is a semigroup with generator $-\G$. Furthermore, Proposition~\ref{prop_D_ppties}\,(vii) implies that $\D$ is a tensor algebra homomorphism; thus, it is uniquely determined by its action on $\ell \in \R^d$, where $\D \ell := e^{-\blambda h}\odot \ell$, and subsequently extended to $(\R^d)^{\otimes n}$ and $\eTA$ via the homomorphism property. In the sequel, we consistently apply $\D$ to elements of $(\R^d)^{\otimes n}$ with this extension in mind. For instance, we write $\D[t-u](\circ dX_u) = e^{-\blambda(t-u)}\odot\circ dX_u$.
\end{sqremark}

Proposition~\ref{prop_D_ppties} allows us to rewrite the iterated integrals in the EFM signature more concisely. Indeed, it follows from \eqref{eq:def_stat_sig_coef} that $\ssigX[]$ can be written as
\begin{equation}\label{eq:sig_as_D}
    \ssigX[s, t] = \emptyword + \int_{s}^t \D[t - u](\ssigX[s, u]\otimes\circ dX_u), \quad -\infty \leq s \leq t.
\end{equation}

The following proposition is a compact formulation of the variation of constants formula written in terms of the operators $\G$ and $\D$.
{
\begin{proposition}\label{prop:variation_of_c}
    {Fix $s \in \R\cup\{-\infty\}$.} Suppose that $\bsigma = (\bsigma_t)_{t \geq s}$ is a continuous $\eTA$-valued semimartingale with decomposition $\bsigma = \bm A^{\bsigma} + \bm M^{\bsigma}$, where $\bm M^{\bsigma}$ is a local martingale and $\bm A^{\bsigma}$ is a process of bounded variation, such that
    \begin{equation}\label{eq:var_of_const_integr_cond}
        \int_{s}^t e^{-\blambda^{\word{v}}(t - u)} |d\bm A^{\bsigma, \word{v}}|_u < \infty, \qquad \int_{s}^t e^{-2\blambda^{\word{v}}(t - u)} d\langle \bm M^{\bsigma, \word{v}} \rangle_u < \infty.
    \end{equation}
    for all words $\word{v}$ and $t\geq s$.
    Then, for all $\bell_s \in\eTA$, the process
    \begin{equation}
        \bell_t = \D[t - s]  \bell_{s} + \int_{s}^t \D[t - u] d\bsigma_u.
    \end{equation}
    solves the equation
    \begin{equation}
        d\bell_t = -\G \bell_t dt + d\bsigma_t, \quad t \geq s.
    \end{equation}
\end{proposition}
\begin{proof}
    The condition \eqref{eq:var_of_const_integr_cond} ensures that the process $\bell = (\bell_t)_{t \geq s}$ is well-defined as a $\eTA$-valued semimartingale. Fix $s_0 \in \R$ and observe that 
    \[
    \bell_t = \D[t - s_0]\left(\D[s_0 -s] \bell_{s} + \int_{s}^t \D[s_0 - u] d\bsigma_u\right).
    \]
    Applying Itô's formula element-wise and using Proposition~\ref{prop_D_ppties}\,(vi) yields the desired equation.

\end{proof}
}
Proposition~\ref{prop:variation_of_c} combined with \eqref{eq:sig_as_D} yield that the dynamics of the EFM signature is given by
\begin{equation}\label{eq:sig_sde}
    d\ssigX[s, t] = -\G \ssigX[s, t]\,dt + \ssigX[s, t]\otimes\circ dX_t, \quad t \geq s.
\end{equation}
The second term in \eqref{eq:sig_sde} corresponds to the standard signature dynamics equation, while the first one is the mean-reversion term appearing due to the exponential weights in \eqref{eq:fm_sig_def}.
\begin{sqremark}
    This and all the following $\eTA$-valued equations are understood element-wise. This means that for any truncation order $N$, the truncated equation reduces to a linear multidimensional Stratonovich SDE with coefficients depending only on the elements of $\ssigX[s, t]$ up to level $N$. {Hence, the SDE} admits a unique strong solution.
\end{sqremark}
The EFM signature allows us to construct a fundamental solution of \eqref{eq:sig_sde}.

\begin{proposition}\label{prop:fundamental_sol}
    Fix $\mathbbm{x}\in \eTA$. The solution to the equation 
    \begin{equation}\label{eq:sig_sde_fund}
    \begin{aligned}
        d\mathbf{\Phi}^{s, \mathbbm{x}}_t &= -\G \mathbf{\Phi}^{s, \mathbbm{x}}_tdt + \mathbf{\Phi}^{s, \mathbbm{x}}_t\otimes\circ dX_t,\quad t \geq s, \\
        \mathbf{\Phi}^{s, \mathbbm{x}}_s &= \mathbbm{x}.
    \end{aligned}
    \end{equation}
    is given by $\mathbf{\Phi}^{s, \mathbbm{x}}_t = (\D[t - s]\mathbbm{x})\otimes\ssigX[s, t]$.
\end{proposition}
\begin{proof}
    Using Proposition~\ref{prop_D_ppties}\,{(vi)} and \eqref{eq:sig_sde}, we obtain
    \begin{align}
        d((\D[t - s]\mathbbm{x})\otimes\ssigX[s, t]) &= d(\D[t - s]\mathbbm{x})\otimes\ssigX[s, t] + (\D[t - s]\mathbbm{x})\otimes d\ssigX[s, t] \\
        &= -(\G\D[t - s]\mathbbm{x})\otimes\ssigX[s, t]\,dt - (\D[t - s]\mathbbm{x})\otimes \G\ssigX[s, t]\,dt + (\D[t - s]\mathbbm{x})\otimes\ssigX[s, t]\otimes\circ dX_t \\
        &= -\G((\D[t - s]\mathbbm{x})\otimes\ssigX[s, t])dt + ((\D[t - s]\mathbbm{x})\otimes\ssigX[s, t])\otimes\circ dX_t,
    \end{align}
    where we used Proposition~\ref{prop_D_ppties}\,{(viii)} for the last equality. This shows that $(\D[t - s]\mathbbm{x})\otimes\ssigX[s, t]$ satisfies \eqref{eq:sig_sde_fund}. Observing that the system \eqref{eq:sig_sde_fund}, truncated at any order $N \in \mathbb{N}$, admits a unique strong solution as a linear SDE concludes the proof.
\end{proof}

\subsection{Fading-memory property of the EFM signature}\label{sect:well_posed}

The fading memory property, as introduced in \cite*{Boyd1985FadingMA}, refers to the continuity of the transformation operator as a function of the past path $(X_t)_{t \leq T}$ with respect to a weighted uniform metric. In the case of the signature transform, working with the increments of $X$ rather than with $X$ itself, one should identify the appropriate spaces in which $X$ lives, as well as a natural way to introduce a time-weighted metric on these spaces. For continuous semimartingales, these questions can be naturally addressed using tools from rough path theory. Although the EFM signature over finite intervals can be readily constructed using existing results in rough integration, improper integration requires additional work. In this subsection, we introduce all the definitions and notations from rough path theory that are necessary to state the main results. Their proofs will rely on improper rough {integration} and are postponed to Section~\ref{sect:rough_paths}.

To define the fading memory property, one first needs to specify the path space $\mathcal{X}$ together with a weighted metric~$\varrho$, in the sense that the distance between two infinite paths is small if they are close in the recent past. Of course, the space $\mathcal{X}$ depends on the regularity of paths $X$ under consideration. In this section, we consider the case of absolutely continuous paths $X$ and the case of continuous semimartingales.

\begin{definition}\label{def:fm_ppty}
    Following \citet{Boyd1985FadingMA}, we say that the EFM signature has the fading memory property with respect to a weighted metric~$\varrho$ on the path space $(\mathcal{X}, \varrho)$ if, for all $s \le t$ and $n \in \mathbb{N}_0$, the mapping
    $$
    X \in \mathcal{X} \mapsto \sigX[s, t]^{\blambda, n},
    $$
    is continuous with respect to~$\varrho$.
\end{definition}

In words, this means that if two paths $X, Y \in \mathcal{X}$ are close in the weighted metric $\varrho$, then the EFM signatures $\sigX[s, t]^{\blambda, n}$ and $\sigX[s, t][Y]^{\blambda, n}$ will also be close. 

\begin{sqremark}
    Definition~\ref{def:fm_ppty} is slightly stronger than that introduced by \citet{Boyd1985FadingMA}: for an infinite path $X = (X_t)_{t \leq T}$, the latter only requires continuity of $\sigX[T]^{\blambda, n}$, whereas the former requires it for all pairs $(s, t)$ such that $s \leq t \leq T$.
\end{sqremark}

We first consider absolutely continuous paths $X = (X_t)_{t \leq T}$ taking values in $\R^d$ and define the weighted variation norm by
\begin{equation}\label{eq:noem_BV_lambda}
    \|X\|_{\mathrm{BV},\blambda} := \int_{-\infty}^T|\D[T-u] \dot X_u|\,du
\end{equation}
The path space $\mathcal{X}$ in this case consists of all absolutely continuous paths $X = (X_t)_{t \leq T}$ such that $\|X\|_{\mathrm{BV},\blambda} < \infty$, and the metric $\varrho$ is induced by the norm $\|\cdot\|_{\mathrm{BV},\blambda}$.

The proposition below establishes the fading memory property of the EFM signature. {Note that
for continuously differentiable $X$, the fading memory property implies that} each level of the EFM signature is a fading memory operator for the signal $u_t = \dot X_t$ in the sense of \citet{Boyd1985FadingMA}. Of course, the signature approach admits much rougher input signals, since we pass from $u$ to its integral $X$. This means that $u = \dot X$ may be not defined in terms of \cite{Boyd1985FadingMA}. One can think about a Brownian motion $X = W$, which is almost surely nowhere differentiable; in this case, $u$ corresponds to the white noise, which is, of course, not a continuous process.

\begin{proposition}\label{prop:fm_bv_case}
    For the absolutely continuous path $X = (X_t)_{t\leq T}$ such that $\|X\|_{\mathrm{BV},\blambda} < \infty$, and for all $s \leq t \leq T$, the following bound holds:
    \begin{equation}\label{eq:BV_bound}
        \left|\D[T-t]\sigX[s, t]^{\blambda, n}\right| \leq \dfrac{\|X\|_{\mathrm{BV},\blambda}^n}{n!}.
    \end{equation}
    Moreover, the EFM signature has the fading memory property with respect to the weighted variation distance induced by $\|\cdot\|_{\mathrm{BV},\blambda}$.
\end{proposition}
\begin{proof}
    Repeating the proof of \cite[Lemma 5.1]{Lyons2014RoughPS}, for $n \geq 1$, we obtain 
    \begin{align}
        \left|\D[T-t]\sigX[s, t]^{\blambda, n}\right| &= \left|\int_{s < u_1 < \cdots < u_n < t} \D[T-u_1] \dot X_{u_1} \otimes \ldots \otimes \D[T-u_n]\dot X_{u_n}\, du_1 \ldots d{u_n}\right| \\
        &\leq \int_{s < u_1 < \cdots < u_n < t}  \left|\D[T-u_1] \dot X_{u_1}\otimes \ldots \otimes \D[T-u_n] \dot X_{u_n}\right|\, du_1 \ldots d{u_n} \\
        &= \int_{s < u_1 < \cdots < u_n < t}  \prod_{k=1}^n \left|\D[T-u_k] \dot X_{u_k}\right|\, du_1 \ldots d{u_n} \\
        &= \dfrac{1}{n!}\left(\int_{s}^{t} \left|\D[T-u]  \dot X_{u}\right|\, du\right)^n \leq \dfrac{\|X\|_{\mathrm{BV},\blambda}^n}{n!}.
    \end{align}
    To prove the continuity, we proceed by induction in $n$ and we claim that 
    \begin{equation}\label{eq:bv_recursive_bound}
         \sup_{u \in [s, T]}|\D[T-u](\sigX[s, u][X]^{\blambda, n} - \sigX[s, u][Y]^{\blambda, n})| \to 0 \quad \text{as} \quad \|X - Y\|_{\mathrm{BV}, \blambda} \to 0,
    \end{equation}
    once the absolutely continuous paths $X, Y$ satisfy
    \begin{equation}\label{eq:BV_M_assumption}
        \|X\|_{\mathrm{BV}, \blambda} < M, \quad \|Y\|_{\mathrm{BV}, \blambda} < M,
    \end{equation}
    for some constant $M > 0$. In particular, \eqref{eq:bv_recursive_bound} implies that $|\sigX[s, t][X]^{\blambda, n} - \sigX[s, t][Y]^{\blambda, n}| \to 0$ as $\|X - Y\|_{\mathrm{BV}, \blambda} \to 0$ for all $s \leq t \leq T$.
    For $n = 0$, $\sigX[s, t]^{\blambda, n} \equiv 1$, and the statement follows trivially. Suppose that we proved {\eqref{eq:bv_recursive_bound}} for all $k \leq n$, and recall that
    $$
    \sigX[s, t]^{\blambda, n+1} = \int_s^t \D[t - u](\sigX[s, u]^{\blambda, n}\otimes dX_u).
    $$
    Then, the difference between two signatures reads 
    \begin{align}
        \sigX[s, t][X]^{\blambda, n+1} - \sigX[s, t][Y]^{\blambda, n+1} = \int_s^t \D[t - u]((\sigX[s, u][X]^{\blambda, n} - \sigX[s, u][Y]^{\blambda, n})\otimes dX_u) + \int_s^t \D[t - u](\sigX[s, u][Y]^{\blambda, n}\otimes d(X_u - Y_u)).
    \end{align}
    Applying $\D[T-t]$, using $\D[T-t]\D[t-u]=\D[T-u]$ and the homomorphism property of $\D[T-u]$, we obtain
    \begin{align}
        |\D[T-t](\sigX[s, t][X]^{\blambda, n+1} - \sigX[s, t][Y]^{\blambda, n+1})| 
        \leq \|X\|_{\mathrm{BV}, \blambda} \sup_{u \in [s, T]}|\D[T-u](\sigX[s, u][X]^{\blambda, n} - \sigX[s, u][Y]^{\blambda, n})| + \|X - Y\|_{\mathrm{BV}, \blambda}\sup_{u \in [s, T]}|\D[T-u]\sigX[s, u][Y]^{\blambda, n}|.
    \end{align}
    
    We note that $\|X\|_{\mathrm{BV}, \blambda}$ and $\sup\limits_{u \in [s, T]}|\D[T-u]\sigX[s, u][Y]^{\blambda, n}|$ are bounded thanks to \eqref{eq:BV_M_assumption} and \eqref{eq:BV_bound}. Taking the supremum in \( t \in [s, T] \) and using the induction hypothesis, we obtain \eqref{eq:bv_recursive_bound} for $n+1$. This finishes the proof.
\end{proof}

We will show that the idea of fading memory can be generalized to encompass irregular drivers $X$ thanks to rough path theory, namely, through the notion of the {exponentially weighted rough path} that we introduce below. To proceed, we need to introduce some notations related to Hölder spaces and rough path spaces.

\paragraph{Hölder spaces.} In this subsection, as well as in Section~\ref{sect:rough_paths}, all Hölder functions $f$ take values in $(\R^d)^{\otimes n}$ for some $n \in \N_0$, where $d$ denotes the dimension of the driving path. This guarantees that $\D[]f$ is well-defined as a restriction of the operator $\D[]$ in Definition~\ref{def:D_operator}.  
For continuous functions $A = (A_{u, v})_{s \leq u \leq v \leq t}$ {and \( \alpha > 0 \)}, we define the Hölder norm
\begin{equation}\label{eq:holder_seminorm_def}
    \|A\|_{\alpha, [s, t]} := \sup_{\substack{u, v \in [s, t] \\ 0 < |u - v| < 1}} \dfrac{|A_{u, v}|}{|u - v|^{\alpha}}.
\end{equation}
We adopt the standard notations for increments
\begin{equation}
    \delta f_{s,t} := f_t - f_s, \quad s \leq t.
\end{equation}

The constraint $0 < |u - v| < 1$ in \eqref{eq:holder_seminorm_def} is added to simplify the further analysis of Hölder functions on infinite time intervals. Notice that this additional constraint does not affect any rough path result since Hölder regularity is always considered as a local property.

Let $\mathcal{C}^\alpha([s, t])$ denote the space of $\alpha$-Hölder paths on $[s, t]$ with finite seminorm
\[
    \|f\|_{\alpha, [s, t]} := \|\delta f\|_{\alpha, [s, t]}.
\]

For functions $A= (A_{u, v})_{u \leq v \leq T}$ defined on infinite time horizons $(-\infty, T]$, we propose to use the exponentially weighted Hölder seminorms with coefficients {$\bm\kappa \in \R^d$}:
\begin{equation}
     \|A\|_{\alpha, \bm\kappa} := \sup_{\substack{u, v \in (-\infty, T] \\ 0 < |u - v| < 1}} \dfrac{|\D[T-u][\bm\kappa]A_{u, v}|}{|u - v|^{\alpha}},
\end{equation}
and we denote by ${\mathcal{C}}^{\alpha, \bm\kappa}((-\infty, T])$ the space of functions $f$ such that $\|f\|_{\alpha, \bm\kappa} := \|\delta f\|_{\alpha, \bm\kappa} < \infty$. 

\begin{sqremark}\label{rmk:holder_spaces}
    The finiteness of $\|f\|_{\alpha, \bm\kappa}$ 
    implies the finiteness of $\|f\|_{\alpha, [s, t]}$
    for all $s < t \leq T$.
\end{sqremark}

However, the finiteness of $\|f\|_{\alpha, \bm\kappa}$ appears to be insufficient to properly define the EFM signature. One should also require boundedness with exponential weight, i.e., the finiteness of the norm
\begin{equation}
    \|f\|_{\infty, \bm\kappa} := \sup_{t \in (-\infty, T]} \left|\D[T - t][\bm\kappa] f_t\right|.
\end{equation}
We denote the space of weighted bounded Hölder functions by
\begin{equation}\label{eq:hat_C_b_alpha_rho}
    {\mathcal{C}}^{\alpha, \bm\kappa}_b((-\infty, T]) := \left\{ f \in {\mathcal{C}}^{\alpha, \bm\kappa}((-\infty, T]) \colon\ \|f\|_{\alpha, \bm\kappa} + \|f\|_{\infty, \bm\kappa} < \infty \right\}.
\end{equation}

\paragraph{Rough path spaces.}  
We recall the definition of the standard $\alpha$-Hölder rough path for $\alpha \in \left(\frac{1}{3}, \frac{1}{2}\right]$ over a finite interval. A more comprehensive introduction to rough path theory can be found in \citet{friz2020course}.

\begin{definition}
    For a finite time interval $[R, T]$ and $\alpha \in \left(\frac{1}{3}, \frac{1}{2}\right]$, the space $\mathscr{C}^\alpha([R, T], \R^d)$ of $\alpha$-Hölder rough paths over $\R^d$ consists of pairs $\mathbf{X} = (\mathbb{X}^1, \mathbb{X}^2)$ satisfying 
    \begin{equation}
        \|\mathbb{X}^1\|_{\alpha, [R, T]} < \infty, \quad
        \|\mathbb{X}^2\|_{2\alpha, [R, T]} < \infty,
    \end{equation}
    as well as the Chen's identity
    \begin{equation}\label{eq:chen_standard}
        \mathbb{X}^2_{s, t} = \mathbb{X}^2_{s, u} + \mathbb{X}^2_{u, t} + \mathbb{X}^1_{s, u} \otimes \mathbb{X}^1_{u, t}, \quad R \leq s \leq u \leq t \leq T
    \end{equation}
    We denote
    \begin{equation}
         \|\mathbf{X}\|_{\alpha, [R, T]} := \|\mathbb{X}^1\|_{\alpha, [R, T]} + \|\mathbb{X}^2\|_{2\alpha, [R, T]}.
    \end{equation}
\end{definition}

The natural extension of this definition to paths defined on infinite time intervals $(-\infty, T]$ consists in replacing the Hölder space and norm conditions by their weighted counterparts.

\begin{definition}[Exponentially weighted rough path]
    For $\alpha \in \left(\frac{1}{3}, \frac{1}{2}\right]$ and {$\bm\kappa \in \R^d$}, the space $\mathscr{C}^{\alpha, \bm\kappa}((-\infty, T], \R^d)$ of $(\alpha, \bm\kappa)$-Hölder rough paths over $\R^d$ consists of pairs $\mathbf{X} = (\mathbb{X}^1, \mathbb{X}^2)$ satisfying
    \begin{equation}
        \|\mathbb{X}^1\|_{\alpha, \bm\kappa} < \infty, \quad
        \|\mathbb{X}^2\|_{2\alpha, \bm\kappa} < \infty,
    \end{equation}
    and Chen's identity \eqref{eq:chen_standard} for all $s \leq u \leq t \leq T$. We also define the norm
    \begin{equation}
         \|\mathbf{X}\|_{\alpha, \bm\kappa} := \|\mathbb{X}^1\|_{\alpha, \bm\kappa} + \|\mathbb{X}^2\|_{2\alpha, \bm\kappa}.
    \end{equation}
\end{definition}

It follows from Remark~\ref{rmk:holder_spaces} that the restriction of a rough path $\mathbf{X} \in \mathscr{C}^{\alpha, \bm\kappa}((-\infty, T], \R^d)$ to $[R, T]$ lies in $\mathscr{C}^\alpha([R, T], \R^d)$ for all $R < T$. 

For $\mathbf{X}, \mathbf{Y} \in \mathscr{C}^{\alpha, \bm\kappa}((-\infty, T], \R^d)$, we define the exponentially weighted rough path distance by
\begin{equation}\label{eq:fm_rp_distance}
    \varrho_{\alpha, \bm\kappa}(\mathbf{X}, \mathbf{Y}) := \|\mathbb{X}^1 - \mathbb{Y}^1\|_{\alpha, \bm\kappa} + \|\mathbb{X}^2 - \mathbb{Y}^2\|_{2\alpha, \bm\kappa},
\end{equation}
which induces the exponentially weighted rough path topology on $\mathscr{C}^{\alpha, \bm\kappa}((-\infty, T], \R^d)$.

To define the EFM signature as a rough integral, we follow the strategy presented in Section~\ref{sect:markovian_lift}. Namely, we use the auxiliary process $X^T$ defined in \eqref{eq:volterra_lift} and consider the rough path lift $\mathbf{X}^T$ associated with it. Using this rough path lift, we define the signature $\mathbb{X}_{s, t}^T$ recursively as an iterated rough integral
\begin{equation}\label{eq:recursive_fm_sig_rough}
    \mathbb{X}^{T, n + 1}_{s, t} = \int_s^t \mathbb{X}^{T, n}_{s, u} \, d\mathbf{X}_u^{T}, \quad s \leq t \leq T, \quad n \geq 0,
\end{equation}
and then set
\[
    \mathbb{X}^{\lambda}_{s, t} := \D[t-T][\blambda]\mathbb{X}^{T}_{s, t}, \quad s \leq t \leq T.
\]
This definition is motivated by the fact that \( \D[t-T][\blambda]\mathbf{X}^{T} = \mathbf{X}^{t} \) and $\D[t-T][\blambda]\mathbb{X}^{T}_{s, t} = \mathbb{X}^{t}_{s, t}$, where the right-hand side coincides with the EFM signature by Proposition~\ref{Prop:Markovian_lift}.

Care must be taken when defining the signature \eqref{eq:recursive_fm_sig_rough} as an improper rough integral when $s = -\infty$. In Section~\ref{sect:rough_paths}, we show that this can be done under the assumption that the rough path lift $\mathbf{X}^T$ belongs to $\mathscr{C}^{\alpha, \brho-\blambda}((-\infty, T], \R^d)$ for some $\brho$ such that $0 \prec \brho \prec \blambda$.
Our main result concerning the existence of the EFM signature as a rough integral shows that this assumption guarantees the well-posedness of the EFM signature with parameter $\blambda$ as an element of ${\mathcal{C}}^{\alpha, \brho}_b((-\infty, T])$.

\begin{theorem}\label{thm:fm_sig_existence}
        If $\mathbf{X}^T \in \mathscr{C}^{\alpha, \brho-\blambda}((-\infty, T])$ for some $\alpha\in (\frac13, \frac12]$ and $0 \prec \brho \prec \blambda$, then the EFM signature $\ssigX[s, \cdot]$ is well-defined for all \( s \in [-\infty, T] \). More precisely, for all $n \geq 0$, we have $\mathbb{X}^{\blambda, n}_{s, \cdot} \in {\mathcal{C}}^{\alpha}([s, T])$ when \( s > -\infty \) and \( \mathbb{X}^{\blambda, n}_{\cdot} \in {\mathcal{C}}^{\alpha, \brho}_b((-\infty, T]) \).
\end{theorem}
\begin{proof}
    For the proof, we refer to Subsection~\ref{corollary:fm_sig_s_inf}.
\end{proof}

\begin{sqremark}
    For standard rough path theory on finite intervals, the rough integral against a semimartingale $X^T$ recovers either the Itô or Stratonovich integral, provided $X^T$ is enhanced with the corresponding lift:
\[
\begin{aligned}
\mathbb{X}^{T, 2, \mathrm{It\hat{o}}}_{s,t} = \int_s^t X_{s, u}^T \otimes dX_u^T,\quad 
\mathbb{X}^{T, 2, \mathrm{Strato}}_{s,t} = \int_s^t X_{s, u}^T \otimes \circ dX_u^T,
\end{aligned}
\]
see \citep[Proposition~5.1 and Corollary~5.2]{friz2020course}.
This result can be readily extended to our setting on infinite intervals by taking the limit as in \eqref{eq:improper_conv_int}. In particular, we can define the EFM signature of a semimartingale in terms of either the Itô or Stratonovich integral. Throughout this paper, we will adopt the Stratonovich lift for continuous semimartingales.
\end{sqremark}
    
The fading memory property can now be formulated for the path space {$\mathcal{X} = \mathscr{C}^{\alpha, \brho-\blambda}((-\infty, T])$} together with the weighted rough path distance {$\varrho = \varrho_{\alpha, \brho-\blambda}$} defined in \eqref{eq:fm_rp_distance}, providing a rough path generalization of the results of \citet{Boyd1985FadingMA} and Proposition~\ref{prop:fm_bv_case}.

\begin{theorem}[Fading memory property]\label{thm:rough_fm}
    Let {$\mathbf{X}^T, \mathbf{Y}^T \in \mathscr{C}^{\alpha, \brho-\blambda}((-\infty, T])$} for  some $\alpha \in \left(\frac{1}{3}, \frac{1}{2}\right]$ and $0 \prec \brho \prec \blambda$. 
    Then, the EFM signature $\ssigX[]$ has the fading memory property w.r.t. the exponentially weighted rough path distance {$\varrho_{\alpha, \brho-\blambda}$} given by \eqref{eq:fm_rp_distance}, that is, for all $-\infty \leq s \leq t \leq T$ and $n \geq 0$,
    \[
        |\mathbb{X}_{s, t}^{\blambda, n} - \mathbb{Y}_{s, t}^{\blambda, n}| \to 0 \quad \text{as} \quad {\varrho_{\alpha, \brho-\blambda}(\mathbf{X}^T, \mathbf{Y}^T) \to 0}.
    \]
\end{theorem}
\begin{proof}
    The proof is given in Subsection~\ref{sect:rough_fm_proof}.
\end{proof}

\begin{sqremark}\label{rmk:lift-exp-decay}
    The sign of $\brho-\blambda \prec 0$ may seem misleading at first glance, since all the weighted norms under consideration multiply the increments by $\D[T - u][\brho-\blambda]$, which tends to $+\infty$ as $u \to -\infty$. However, it is important to remember that the functions to which these weights are applied have their increments discounted by $\D[T-u]$:
    $$
    dX_t^T = \D[T-t]dX_t,
    $$
    so that the quantities appearing in the norms 
    $$
    \D[T - u][\brho-\blambda]\delta X_{u, v}^T \approx \D[T - u][\brho-\blambda] \D[T-u]\delta X_{u, v} =  \D[T-u][\brho]\delta X_{u, v}
    $$
    weight the increments of $X$ with the positive exponent $\brho \succ 0$, ensuring that the memory indeed \textit{fades} and that $\mathbb{X}_{s, \cdot}^{\blambda}$ has the $\brho$-weighted Hölder regularity given by Theorem~\ref{thm:fm_sig_existence}.
\end{sqremark}

\subsection{EFM signature as a stationary process}
By definition, the EFM signature is time-invariant in the sense of \eqref{eq:time_invariance}. This result is formalized in the following proposition.

\begin{proposition}
    If $X = (X_t)_{t \in \R}$ is a continuous semimartingale and $Y_\cdot = X_{\cdot + h}$ for some $h \in \R$, then
    \begin{equation}\label{eq:fm_sig_TI}
        \ssigX[t][Y] = \ssigX[t + h][X].
    \end{equation}
\end{proposition}
\begin{proof}
We prove the result by induction on $|\word{v}|$. For $\word{v} = \emptyword$, the result holds trivially. Now, let $\word{v} = \word{v'i}$ for some $\word{i} \in A_d$, and assume that the result holds for $\word{v'}$. Using \eqref{eq:def_stat_sig_coef}, we obtain 
$$
\mathbb{Y}_t^{\blambda, \word{v}} = \int_{-\infty}^te^{-\blambda^{\word{v}}(t-u)}\mathbb{Y}_u^{\blambda, \word{v'}}\circ dX_{u+h}^i = \int_{-\infty}^te^{-\blambda^{\word{v}}(t-u)}\mathbb{X}_{u+h}^{\blambda, \word{v'}}\circ dX_{u+h}^i,
$$ 
by the induction hypothesis. The change of variables $r = u + h$ yields:
$$
\mathbb{Y}_t^{\blambda, \word{v}} = \int_{-\infty}^{t+h}e^{-\blambda^{\word{v}}((t+h)-r)}\mathbb{X}_{r}^{\blambda, \word{v'}}\circ dX_{r}^i = \mathbb{X}_{t + h}^{\blambda, \word{v}}.
$$
This completes the proof.

\end{proof}

As a time-invariant transformation, EFM signature is expected to produce a stationary output in response to the stationary input.
Indeed, the EFM signature $\mathbb{X}^{\blambda} = (\mathbb{X}^{\blambda}_t)_{t\in\mathbb{R}}$ can be rigorously formulated as a $T((\mathbb{R}^d))$-valued stochastic process. For any finite sequence of times $t_1 < \dots < t_n$ and any truncation level $N$, the corresponding finite-dimensional projections of the signature are well-defined as solutions to the linear SDE \eqref{eq:sig_sde}. Because this holds for all $N$, it guarantees the consistency of the finite-dimensional distributions across the extended tensor algebra. Consequently, by Kolmogorov's extension theorem, there exists a unique probability measure on the path space $(T((\mathbb{R}^d)))^{\mathbb{R}}$ that admits these exact projections. This unique measure defines the law of the full process, denoted by $\mathcal{L}(\mathbb{X}^{\blambda})$, which naturally aligns with the finite-dimensional laws $\mathcal{L}(\mathbb{X}^{\blambda}_{t_1}, \dots, \mathbb{X}^{\blambda}_{t_n})$. From this point forward, whenever we refer to random variables or stochastic processes taking values in $T((\mathbb{R}^d))$, we implicitly rely on this construction. In this subsection, we will show that if $X$ is a {continuous semimartingale} with stationary increments, then its EFM signature is a stationary process, i.e., for all $t_1 < \ldots < t_n$ and for all $h \in \R$, we have
\begin{equation}
    \mathcal{L}(\ssigX[t_1], \ldots, \ssigX[t_n]) = \mathcal{L}(\ssigX[t_1 + h], \ldots, \ssigX[t_n + h]).
\end{equation}
The proof relies on the following technical lemma.

\begin{lemma}\label{lem:stationarity}
    Let $\eta_t$ and $\sigma_t$ be stationary {continuous semimartingales} (possibly multivariate), such that
    $$Y_t = \int_{-\infty}^t e^{-\mu(t-s)}\sigma_s \circ dX_s,$$
    {is well-defined} and for all $h \in\R$,
    \begin{equation}
        \mathcal{L}((\eta_{t}, \sigma_{t}, X_t)_{t\in\R}) = \mathcal{L}((\eta_{t+h}, \sigma_{t+h}, X^h_t)_{t\in\R}),
    \end{equation}
    where $X_t^h := X_{t + h} - X_h$.
    Then, for all $\mu > 0$ and $i \in \{1, \ldots, d\}$, $Y$ is stationary and
    \begin{equation}\label{eq:int_law_eq}
        \mathcal{L}((\eta_{t}, Y_{t}, X_t)_{t\in\R}) = \mathcal{L}((\eta_{t+h}, Y_{t+h}, X^h_t)_{t\in\R}).
    \end{equation}
\end{lemma}
\begin{proof}
    The proof of this statement follows the lines of the proof of \cite[Proposition~1]{sauri2017brownian}, with the following straightforward modifications: considering a semimartingale with stationary increments instead of Brownian motion, using Stratonovich integration instead of Itô integration, and adding the equality \eqref{eq:int_law_eq}, which also follows immediately from the referenced proof.
\end{proof}

\begin{theorem}\label{thm:stationarity}
    Suppose that $X = (X_t)$ is an $\R^d$-valued continuous semimartingale with stationary increments. Then, $\ssigX[] = (\ssigX[t])_{t \in \R}$ is a stationary $\eTA[d]$-valued process. In particular, for all $t \in \R$,
    \begin{equation}
        \mathcal{L}(\ssigX[t]) = \mathcal{L}(\ssigX[0]).
    \end{equation}
\end{theorem}

\begin{proof}
We show the stationarity of the truncated EFM signature $\sigX[]^{\blambda, \leq N}$ for all $N$ by induction on words $\word{v}$, ordered first by length and then lexicographically (this order is denoted by $\prec$). For $\word{v} = \emptyword$, the EFM signature is a constant process equal to $1$. 

For the induction step, consider a word $\word{v} = \word{v'i}$ and suppose that the process $(\sigX[]^{\blambda, \word{w}})_{\word{w} \prec \word{v}}$ is stationary, and that
\begin{equation}
    \mathcal{L}(((\sigX[t]^{\blambda, \word{w}})_{\word{w} \prec \word{v}}, X_t)_{t \in \R}) = \mathcal{L}(((\sigX[t+h]^{\blambda, \word{w}})_{\word{w} \prec \word{v}}, X_t^h)_{t \in \R}).
\end{equation}
Recall that $\sigX[]^{\blambda, \word{v}}$ is given by \eqref{eq:def_stat_sig_coef}. It follows from Lemma~\ref{lem:stationarity}, with $\eta_t = (\sigX^{\blambda, \word{w}})_{\word{w} \prec \word{v}}$, $\sigma_t = \sigX[t]^{\blambda, \word{v'}}$, and $\mu = \blambda^{\word{v}}$, that $\sigX[t]^{\blambda, \word{v}}$ is stationary, and for all $h \in \R$, we have
\begin{equation}
    \mathcal{L}(((\sigX[t]^{\blambda, \word{w}})_{\word{w} \preceq \word{v}}, X_t)_{t \in \R}) = \mathcal{L}(((\sigX[t+h]^{\blambda, \word{w}})_{\word{w} \preceq \word{v}}, X_t^h)_{t \in \R}).
\end{equation}

This proves the stationarity of the finite-dimensional processes $\sigX[]^{\blambda, \leq N}$ for all $N \in \N$, and hence implies the stationarity of $\ssigX[]$.
\end{proof}

Another advantage of increments stationarity is that it allows one to verify the conditions of Theorem~\ref{thm:fm_sig_existence}, showing that $\mathbf{X}^T \in \mathscr{C}^{\alpha, \brho-\blambda}((-\infty, T])$.

\begin{proposition}\label{prop:stat_holder_const}
    Let $X = (X_t)_{t \leq T}$ be a continuous $\R^d$-valued semimartingale with stationary increments and let $X^T$ be given by \eqref{eq:volterra_lift}. Suppose that for some $h \geq 2$ and $ \alpha \in (\frac13, \frac12]$, the Hölder constants satisfy
    \begin{equation}\label{eq:Holder_const}
         \E\left[\|\mathbb{X}^{T, 1}\|_{\alpha, [T - h, T]}\right]  + \E\left[\|\mathbb{X}^{T, 2}\|_{2\alpha, [T-h, T]}\right] < \infty.
    \end{equation}
    Then, for all \( \brho \in \mathbb{R}^d \) such that \( \brho \succ 0 \), we have $\mathbf{X}^T \in \mathscr{C}^{\alpha, \brho-\blambda}((-\infty, T])$ almost surely.
\end{proposition}

\begin{proof}
    First, by the stationarity of increments, the laws of $\|\D[T - t_0][-\blambda]\mathbb{X}^{T, 1}\|_{\alpha, [t_0, t_0 + h]}$ and $\|\D[T - t_0][-\blambda]\mathbb{X}^{T, 2}\|_{2\alpha, [t_0, t_0 + h]}$ do not depend on $t_0 \in (-\infty, T-h]$ and, by \eqref{eq:Holder_const},
    $$
        \E\left[\|\D[T - t_0][-\blambda]\mathbb{X}^{T, 1}\|_{\alpha, [t_0, t_0 + h]}\right]  + \E\left[\|\D[T - t_0][-\blambda]\mathbb{X}^{T, 2}\|_{2\alpha, [t_0, t_0 + h]}\right] < \infty.
    $$
    Since $h \geq 2$, for all $u, v \in (-\infty, T]$ such that $0 < v - u < 1$, there exists $n \in \mathbb{N}_0$ such that $[u,v] \subset [T - (n + 2), T - n]$, so that 
    \begin{align*}
        |\D[T-u][\brho-\blambda]\mathbb{X}^{T, 1}_{u, v}| &=  |\D[T-u][\brho]\D[T-u][-\blambda]\mathbb{X}^{T, 1}_{u, v}| \leq |\D[n][\brho]\D[n+2][-\blambda]\mathbb{X}^{T, 1}_{u, v}| \\
        &\leq e^{-\brho^{\min}n}\|\D[n+2][-\blambda]\mathbb{X}^{T, 1}\|_{\alpha, [T - (n + 2), T - n]}|v - u|^\alpha \\
        &\leq \sum_{n \geq 0 }e^{-\brho^{\min}n}\|\D[n+2][-\blambda]\mathbb{X}^{T, 1}\|_{\alpha, [T - (n + 2), T - n]}|v - u|^\alpha.
    \end{align*}    
    This implies that
    $$
    \| \mathbb{X}^{T, 1}\|_{\alpha,\brho-\blambda} \leq \sum_{n \geq 0 }e^{-\brho^{\min}n}\|\D[n+2][-\blambda]\mathbb{X}^{T, 1}\|_{\alpha, [T - (n + 2), T - n]}.
    $$
     Since $\E[\|\D[n+2][-\blambda]\mathbb{X}^{T, 1}\|_{\alpha, [T - (n + 2), T - n]}]$ is finite and independent of $n$, taking expectations, we obtain
    $$
    \E[\| \mathbb{X}^{T, 1}\|_{\alpha, \brho-\blambda}] \leq \dfrac{1}{1 - e^{-\brho^{\min}}}\E[\|\D[2][-\blambda]\mathbb{X}^{T, 1}\|_{\alpha, [T - 2, T]}] < \infty.
    $$
    Hence, $\|\mathbb{X}^{T, 1}\|_{\alpha, \brho-\blambda}$ is almost surely finite. A similar argument shows that $\|\mathbb{X}^{T, 2}\|_{2\alpha, \brho-\blambda} < \infty$ a.s. as well.
\end{proof}

\begin{sqremark}\label{rem:BM_EFM_sig_well_def}
        The condition \eqref{eq:Holder_const} can be verified using the Kolmogorov continuity theorem for rough paths (see, e.g., \cite[Theorem 3.1]{friz2020course}). Proposition~\ref{prop:stat_holder_const}, in particular, establishes the well-definedness of the EFM signature of $d$-dimensional (possibly time-augmented) Brownian motion $\ssigX[][W]$ ($\ssigX[][\widehat W]$ respectively).
\end{sqremark}

\section{Properties of the EFM signature}\label{sect:algebraic_prop}

In this section, we establish results on the algebraic and analytic properties of EFM signatures, which closely resemble those of standard signatures. 

\subsection{Chen's identity}
Chen's identity for signatures (see \cite{chen1958integration}) states that for any $0 \leq s \leq u \leq t \leq T$,
\begin{equation}
    \sigX[s, t] = \sigX[s, u] \otimes \sigX[u, t].
\end{equation}
In other words, to compute the signature of a concatenated path $X * Y$, one can compute independently the signatures of $X$ and $Y$, and then take their tensor product.

When dealing with EFM signatures, the algebraic structure of iterated integrals remains unchanged, except that each component $\sigX[s, u]^{\blambda, \word{v}}$ of the EFM signature over $[s, u]$ must be multiplied by $e^{-\blambda^{\word{v}}(t - u)}$ to compensate for the difference in exponential weights in \eqref{eq:fm_sig_def}. This means that the EFM signature $\ssigX[s, u]$ must be appropriately \emph{discounted} using the operator $\D[t - u]$ before applying the tensor product. In other words, the contribution of past segments to the EFM signature decays exponentially at rate $\blambda^i$ for the $i$-th component of the path. More precisely, the following result holds.

\begin{proposition}\label{prop:Chen_FM}
    {Fix $\blambda\in\R^d, \blambda\succ 0$ and let $X = (X_t)_{t\in\R}$ be continuous $\R^d$-valued semimartingale such that $\ssigX[]$ is well-defined. } Then,
    \begin{equation}\label{eq:stat_chen}
        \ssigX[s, t] = (\D[t - u]\ssigX[s, u]) \otimes \ssigX[u, t]
    \end{equation}
    for all $s \leq u \leq t$.
\end{proposition}
\begin{proof}
    The proof follows from Proposition~\ref{prop:fundamental_sol} together with the EFM signature dynamics \eqref{eq:sig_sde}.
\end{proof}

\begin{sqremark}
    Chen's identity for the EFM signature can also be proved using the link to the standard signature presented in Subsection~\ref{sect:markovian_lift}. Indeed, by Proposition~\ref{Prop:Markovian_lift} and the classical Chen identity,
    \[
        \ssigX[s, t] = \mathbb{X}_{s,t}^t = \mathbb{X}_{s,u}^t \otimes \mathbb{X}_{u,t}^t.
    \]
    It remains to observe that $X^t = \D[t-u]X^u$, so that
    \(
        \mathbb{X}_{s,u}^t = \D[t-u]\mathbb{X}_{s,u}^u = \D[t-u]\ssigX[s,u].
    \)
\end{sqremark}

\subsection{EFM signature of linear path}\label{sect:linear_path}
Chen's identity is a crucial property for the computation of signatures. Indeed, a standard approach to computing the signature of a discrete path is to approximate it by a piecewise linear path, compute the signature of each linear segment—which is known explicitly—and then recover the signature of the entire path by performing the tensor multiplication of the segment signatures.

We begin by noticing that when {all the components of $\blambda$ are equal}, the EFM signature of a linear path can be computed in closed form as an exponential, similarly to the standard signature. For 
$\bell\in\eTA$, we define the tensor product exponential as the series
$$
\exp^\otimes(\bell) := \sum_{k \geq 0}\dfrac{\bell\conpow{k}}{k!}.
$$
\begin{lemma}
    Suppose that $\blambda = (\lambda, \ldots, \lambda)$ for some $\lambda > 0$. For a linear path $X_t = t \bm{x}$ with $\bm{x} \in \R^d$, its EFM signature $\mathbb{X}^{\blambda}_{s, t} = (\mathbb{X}^{\blambda, n}_{s, t})_{n \geq 0}$ is given explicitly by
    \begin{equation}\label{eq:linear_path_stat_sig}
        \ssigX[s, t] = \exp^\otimes\left(\frac{1 - e^{-\lambda(t - s)}}{\lambda(t-s)}  (X_{t} - X_s)\right).
    \end{equation}

In particular, the EFM signature of a linear path over $(-\infty, t]$ is constant and equals
\begin{equation}\label{eq:linear_path_stat_sig_inf}
    \mathbb{X}^{\blambda}_{t} = \exp^\otimes\left({\frac{\bm{x}}{\lambda}}\right).
\end{equation}
Moreover, the EFM signature of a piecewise-linear path with breakpoints $t_0 < t_1 < \ldots < t_n$ is given recursively by
\begin{equation}\label{eq:piecewise_linear_path_stat_sig_inf}
   \ssigX[t_0, t_n] = (\D[t_n - t_{n-1}]\ssigX[t_0, t_{n-1}]) \otimes \exp^\otimes\left(\frac{1 - e^{-\lambda(t_{n} - t_{n-1})}}{\lambda(t_{n}-t_{n-1})} (X_{t_{n}} - X_{t_{n-1}})\right). 
\end{equation}
\end{lemma}
\begin{proof}
A straightforward computation shows that
\begin{equation}
    \mathbb{X}^{\blambda, n}_{s, t} = \dfrac{1}{n!}\left(\dfrac{1 - e^{-\lambda(t - s)}}{\lambda(t-s)} (X_{t} - X_s)\right)^{\otimes n}, \quad n \geq 0,
\end{equation}
so that \eqref{eq:linear_path_stat_sig} holds. \eqref{eq:linear_path_stat_sig_inf} follows as a limit as $s \to -\infty$.
The formula \eqref{eq:piecewise_linear_path_stat_sig_inf} follows from Chen's identity \eqref{eq:stat_chen}.
\end{proof}

The case where $\blambda$ has different components is more delicate, as the computation involves integrals of the form
\begin{equation}\label{eq:interated_exp_int}
    \int_{s < u_1 < \cdots < u_n < t} e^{-\blambda^{i_1}(t - u_1)} \cdots e^{-\blambda^{i_n}(t - u_n)} \, d u_1 \ldots d u_n,
\end{equation}
which are no longer symmetric, making it impossible to use formulas like \eqref{eq:linear_path_stat_sig}. Fortunately, these integrals can still be computed explicitly via the following lemma.
\begin{lemma}\label{lem:linear_path}
    Fix $\blambda = (\blambda^1, \ldots, \blambda^d)\in\R^d, \, \blambda\succ 0$.
    For each word $\word{v} = \word{i_1\ldots i_n}$, the components of the EFM signature $\ssigX[]$ of the linear path $X_t = tx$ are given by
    \begin{equation}\label{eq:linear_path_coefs_explicit}
        \mathbb{X}^{\blambda, \word{v}}_{s, t} = (x^{i_1}\cdot\ldots\cdot x^{i_n})\sum_{k=0}^{n} c_k e^{-\mu_k (t - s)}, \qquad \mathbb{X}^{\blambda, \word{v}}_{t} = (x^{i_1}\cdot\ldots\cdot x^{i_n})\prod\limits_{k=1}^n\dfrac{1}{\mu_k}
    \end{equation}
    where $c_k = \prod\limits_{l \neq k}\dfrac{1}{\mu_l - \mu_k}$ and $\mu_k = \sum\limits_{j = 1}^k\blambda^{i_j}, \ k = 0, \ldots, n,$ with $\mu_{0} = 0$.
\end{lemma}
\begin{proof}
    The proof is given in Appendix~\ref{sect:linear_path_proof}.
\end{proof}
\begin{sqremark}
    In the case $\blambda = (\lambda, \ldots, \lambda)$, the computational complexity of the EFM signature for a linear path truncated at order $N$ coincides with that of the standard signature, equaling $\mathcal{O}(d^N)$ in terms of both memory and time. This follows immediately from \eqref{eq:linear_path_stat_sig}. In the general case $\blambda = (\blambda^1, \ldots, \blambda^d) \in \mathbb{R}^d$, the complexity of computation using formula \eqref{eq:linear_path_coefs_explicit} is slightly higher, at $\mathcal{O}(N^2d^N)$. We suggest that this could be improved. However, the development of efficient numerical methods for computing the EFM signature is beyond the scope of this paper.
\end{sqremark}

We can now state the result for general $\blambda$. The EFM signature of a linear path can be also written in exponential form, as in \eqref{eq:linear_path_stat_sig}, using the notion of Magnus exponential. 

We rely on the classical result due to \citet{magnus}, adapted to the right multiplication case. It states that the solution to the linear system on $\tTA{N}$
\begin{align}
    \dot{\mathbbm{x}}_t^N &= \mathbbm{x}_t^N\otimes \bm A(t), \quad t \geq s,\\
    \mathbbm{x}_s^N &= \emptyword,
\end{align}
is given by
\begin{equation}\label{eq:magnus_exp}
    \mathbbm{x}_t^N = \exp^\otimes\left(\sum_{k \geq 1}\bm{\Omega}_k(s, t)\right),
\end{equation}
where 
\begin{equation}\label{eq:magnus_omega}
\begin{aligned}
    \bm{\Omega}_1(s, t) &= \int_{s}^t \bm A(u_1)du_1, \\
    \bm{\Omega}_2(s, t) &= \dfrac{1}{2}\int_s^tdu_1\int_s^{u_1}du_2 [\bm A(u_1), \bm A(u_2)], \\
    \bm{\Omega}_3(s, t) &= \dfrac{1}{6}\int_s^tdu_1\int_s^{u_1}du_2\int_s^{u_2}du_3 \left([\bm A(u_1), [\bm A(u_2), \bm A(u_3)]] + [[\bm A(u_1), \bm A(u_2)], \bm A(u_3)]\right), \\
    \ldots \\
    \bm{\Omega}_n(s, t) &= \sum_{j=1}^{n-1}\dfrac{B_j}{j!}\sum_{\substack{k_1 + \ldots + k_j = n-1 \\ k_1 \geq 1, \ldots, k_j \geq 1}}\int_s^t\mathrm{ad}_{\bm{\Omega}_{k_1}(s, u)}\ldots\mathrm{ad}_{\bm{\Omega}_{k_j}(s, u)}\bm{A}(u)\,du,
\end{aligned}
\end{equation}
{the bracket $[\cdot, \cdot]$ is the inverse commutator $[\bp, \bq] := \bq \otimes\bp - \bp\otimes\bq$, $\mathrm{ad}_{\bm{B}}\bm{A} := [\bm B, \bm A]$, and $B_j$ stands for the Bernoulli numbers defined through the generating function
$$
\dfrac{x}{e^x-1} = \sum_{n \geq 0}\dfrac{B_n}{n!}x^n.
$$}
The following result requires the inversion of the operator $\G$ introduced in Definition~\ref{def:G_oper}. Although $\G$ is diagonal, it is not injective, since $\G\emptyword = 0$. For this reason, we will use its generalized inverse $\G^\dag$, defined by
\begin{align}\label{eq:G_inv}
    \G^\dag \colon \bell = \sum_{n \geq 0} \sum_{|\word{v}| = n} \bell^{\word{v}}\word{v} &\mapsto \G^\dag \bell = \sum_{n \geq 1} \sum_{|\word{v}| = n} \frac{1}{\blambda^{\word{v}}} \bell^{\word{v}}\word{v}.
\end{align}

\begin{theorem}\label{thm:linear_path_magnus}
    Fix $\blambda = (\blambda^1, \ldots, \blambda^d)\in\R^d, \, \blambda\succ 0$. For a linear path $X_t = t \bm{x}$ with $\bm{x} \in \R^d$, its EFM signature $\mathbb{X}^{\blambda}_{s, t}$ is given by
    \begin{equation}\label{eq:linear_path_stat_sig_magnus}
        \ssigX[s, t] = \exp^\otimes \left(\sum_{k \geq 1}\bm{\Omega}_k^{\blambda, {\bm x}}(s, t)\right), \quad
        \ssigX[t] = \exp^\otimes\left(\sum_{k \geq 1}\bm{\Omega}_k^{\blambda, {\bm x}}(-\infty, t)\right) = \sum_{k \geq 0}{\bf H}_{\bm x}^k\emptyword,
    \end{equation}
    where ${\bm{\Omega}}_k^{\blambda, {\bm x}}(s, t)$ is given by \eqref{eq:magnus_omega} with $\bm A(u) = \D[t - u]{\bm x}$ and ${\bf H}_{\bm x}\colon \bell \mapsto \G^\dag(\bell\otimes {\bm x})$. The coefficients of the EFM signature are given explicitly by \eqref{eq:linear_path_coefs_explicit}.
    Moreover, if $\blambda^1 = \ldots = \blambda^d =: \lambda$, then 
    \begin{equation}\label{eq:linear_path_stat_sig_magnus_equal}
        \ssigX[s, t] = \exp^\otimes\left(\frac{1 - e^{-\lambda(t - s)}}{\lambda} {\bm x}\right), \quad \mathbb{X}^{\blambda}_{t} = \exp^\otimes\left({\frac{\bm{x}}{\lambda}}\right).
    \end{equation}
\end{theorem}
\begin{proof}
    In the case of the linear path, the EFM signature equation \eqref{eq:sig_as_D} reads
    \begin{equation}\label{eq:linear_path_D_sig}
        \ssigX[s, t] =\emptyword +  \int_{s}^t \D[t - u](\ssigX[s, u]\otimes x) du, \quad t \geq s.
    \end{equation}
    This implies that $\mathbbm{y}_t := \D[s - t]\ssigX[s, t]$ satisfies the linear equation 
    \begin{align}
        d{\mathbbm{y}}_t &= \mathbbm{y}_t\otimes (\D[s - t]\bm x)\, dt, \quad  t \geq s,\\
        \mathbbm{y}_0 &= \emptyword,
    \end{align}
    so that 
    $$
    \ssigX[s, t] = \D[t - s]{\mathbbm{y}}_t = \D[t - s]\exp^\otimes\left(\sum_{k \geq 1}\tilde{\bm{\Omega}}_k^{\blambda, \bm x}(s, t)\right),
    $$
    where $\tilde{\bm{\Omega}}_k^{\blambda, \bm x}(s, t)$ are given by \eqref{eq:magnus_omega} with $\bm A(u) = \D[s - u]\bm x$. Thanks to Proposition~\ref{prop_D_ppties}\,{(vii)}, one has 
    $$
    \D \exp^{\otimes}(\bell) = \exp^{\otimes}(\D\bell), \quad h \in \R,\quad \bell \in\eTA.
    $$
    Applying again Proposition~\ref{prop_D_ppties}\,{(vii)}, we obtain that
    $$
    \D[t-s]\tilde{\bm{\Omega}}_k^{\blambda, \bm x}(s, t) = {\bm{\Omega}}_k^{\blambda, \bm x}(s, t), \quad k \geq 1,
    $$
    so that the exponential representation in \eqref{eq:linear_path_stat_sig_magnus} holds. 

    The second part of \eqref{eq:linear_path_stat_sig_magnus} follows from the application of the Picard iterations to \eqref{eq:linear_path_D_sig} and noticing that for all $\bell$ such that $\bell^{\emptyword} = 0$, we have
    \begin{equation}\label{eq:D_int}
        \int_{s}^t \D[t-u]\bell\, du = \G^\dag(\mathrm{Id} - \D[t-s])\bell, \quad\int_{-\infty}^t \D[t-u]\bell\, du = \G^\dag\bell.
    \end{equation}
    
    In the case $\blambda^1 = \ldots = \blambda^d = \lambda$, the generator $\bm A(u)$ becomes commutative, so that $[\bm A(u_1), \bm A(u_2)] \equiv 0$ and $\bm \Omega_k^{\blambda, \bm x}(s, t) \equiv 0$ for $k \geq 2$. The formula \eqref{eq:linear_path_stat_sig_magnus_equal} follows immediately from \eqref{eq:D_int} with $\bell = \bm x$.
\end{proof}
\begin{corollary}
    The EFM signature of a piecewise-linear path with breakpoints $t_0 < t_1 < \ldots < t_n$ is given recursively by
\begin{equation}\label{eq:piecewise_linear_path_stat_sig_piece_magnus}
       \ssigX[t_0, t_n] = (\D[t_n - t_{n-1}]\ssigX[t_0, t_{n-1}]) \otimes \exp^\otimes\left(\sum_{k \geq 1}\bm{\Omega}_k^{\blambda, \frac{X_{t_n} - X_{t_{n-1}}}{t_n - t_{n-1}}}(t_{n-1}, t_n)\right). 
    \end{equation}
\end{corollary}
%

\subsection{Time reparametrization}
It is well-known that the signature of path is invariant under reparametrization. Namely, for any smooth reparametrization $\psi\colon\ [0, T] \to [0, T]$, denote $\tilde X_t := X_{\psi_t}$. It follows that
$$
\int_{0 < u_1 < \cdots < u_n < T} \circ d X_{u_1} \otimes \cdots \otimes \circ d X_{u_n}  = \int_{0 < u_1 < \cdots < u_n < T} \circ d\tilde X_{u_1} \otimes \cdots \otimes \circ d\tilde X_{u_n}.
$$
Similarly, for the EFM signatures,
\begin{align*}
    &\int_{u_1 < \cdots < u_n < T}  e^{-\blambda(T - u_1)}\odot d X_{u_1} \otimes\circ  \cdots  \otimes \circ e^{-\blambda(T - u_n)} \odot d X_{u_n} \\ = &\int_{u_1 < \cdots < u_n < T}  e^{-\blambda(\psi_T - \psi_{u_1})}\odot d\tilde X_{u_1} \otimes\circ \cdots \otimes \circ e^{-\blambda(\psi_T - \psi_{u_n})}\odot d\tilde X_{u_n}.
\end{align*}
Thus, the EFM signature is not invariant under time reparametrization of the path $X$, since time explicitly appears in its definition. However, if one views the EFM signature as a transformation of both time $\tau_\alpha$ and the path $X_\alpha$, parametrized by $\alpha \in (-\infty, T]$, then:
\begin{equation}
    \mathrm{EFM}^{\blambda, n}_T(\tau, X) := \int_{\alpha_1 < \cdots < \alpha_n < T}  e^{-\blambda(\tau_T - \tau_{\alpha_1})} \odot d X_{\alpha_1} \circ \otimes \cdots \otimes \circ e^{-\blambda(\tau_T - \tau_{\alpha_n})} \odot  d X_{\alpha_n},
\end{equation}
it is invariant under reparametrizations of the concatenated path $(\tau, X)$; that is, for $\tilde{\tau}_\alpha = \tau_{\psi(\alpha)}$ and $\tilde{X}_\alpha = X_{\psi(\alpha)}$, where $\psi \colon (-\infty, T] \to (-\infty, T]$ is any smooth reparametrization, the following reparametrization invariance property holds:
$$
    \mathrm{EFM}^{\blambda, n}_T(\tau, X) = \mathrm{EFM}^{\blambda, n}_T(\tilde\tau, \tilde X)
$$

To simplify notation, we will always assume that the parametrization of the path is chosen so that $\tau_t = \alpha$ for all $\alpha \in \R$.

\subsection{Uniqueness of the EFM signature}

The uniqueness of the signature means that it determines the path up to tree-like equivalence, as shown by \cite{hambly2010uniqueness} for paths of bounded variation and by \citet{boedihardjo2016signatureroughpath} in the rough path setting. In particular, the signature of the time-augmented path $\widehat{X}_t = (t, X_t),\ t \in [0, T]$ determines the path uniquely. In this subsection, we state the uniqueness result for the EFM signature.
For notational simplicity, when working with time-augmented paths $\widehat{X} = (t, X_t)$, we denote the time component by the letter $\word{0}$, and the spatial components of $X$ by the letters $\word{1}, \ldots, \word{d}$.


\begin{proposition}\label{prop:uniqueness}
    Let $X = (X_t)_{-\infty < t \leq T}$ be a continuous $\R^d$-valued semimartingale satisfying \eqref{eq:lift_assumption}, and let $\blambda \in \R^d$ with $\blambda \succ 0$. Then, the EFM signature $\ssigX[T]$ determines the modified path
    \(
        X^T = (X_t^T)_{t \leq T},
    \)
    where \(
    X_t^T = \int_{-\infty}^t \D[T-u]\, dX_u,
    \)
    uniquely up to tree-like equivalence. Moreover, if $X^T$ has at least one strictly monotone component, then $\ssigX[T]$ determines $X = (X_t)_{t \leq T}$ up to a constant shift.
\end{proposition}

\begin{proof}
    The proof of the first statement follows from \cite[Theorem~1.1]{boedihardjo2016signatureroughpath} and Proposition~\ref{Prop:Markovian_lift}, since the process $Z^T$ introduced in Proposition~\ref{Prop:Markovian_lift} is a time reparametrization of $X^T$. If $X^T$ has a strictly monotone component, then it is determined uniquely up to a constant shift, and $X$ can be reconstructed via
    \[
        X_t = X_s + \int_s^t \D[u - T]\, dX_u^T,
    \]
    for all $s, t \leq T$.
\end{proof}

\begin{corollary}\label{cor:uniqueness_time_aug}
    Let $X = (X_t)_{-\infty < t \leq T}$ be a continuous $\R^d$-valued semimartingale satisfying \eqref{eq:lift_assumption}, let $\blambda \in \R^{d+1}$ with $\blambda \succ 0$, and let $\ssigX[T][\widehat{X}]$ denote the EFM signature of its time-augmented path $\widehat{X}_t = (t, X_t)$. Then, $\ssigX[T][\widehat{X}]$ determines $X = (X_t)_{t \leq T}$ uniquely up to a constant shift.
\end{corollary}

\begin{proof}
    The first coordinate of the time-augmented modified path $\widehat{X}^T$ corresponding to the time component is given by
    \[
        \widehat{X}_t^{T,0}
        =
        \int_{-\infty}^t e^{-\blambda^0(T-u)}\, du
        =
        \frac{1}{\blambda^0} e^{-\blambda^0(T-t)},
    \]
    which is a strictly increasing function of $t$. Applying Proposition~\ref{prop:uniqueness} completes the proof.
\end{proof}

The uniqueness result allows us to conclude that any functional of the path $X$ can be expressed as a function of the EFM signature of this path:
$$
F\big((X_s)_{s \leq t}\big) = \tilde{F}\big(\ssigX[t][\widehat{X}]\big),
$$
for some function $\tilde{F}$, provided that the EFM signature $\ssigX[t][\widehat{X}]$ captures all the information contained in the path $(X_s)_{s \leq t}$.

\subsection{Linearization property}

Another important operation on $\eTA$ is the shuffle product, which allows for the linearization of products of two linear functionals of the signature; see \cite*[Theorem 2.15]{lyons2007differential}. For basis vectors $\word{v}, \word{w} \in \eTA$ and letters $\word{i}, \word{j}$, the shuffle product is defined recursively by
\begin{align*}
    (\word{v} \word{i}) \shuprod (\word{w} \word{j}) &
    = (\word{v} \shuprod (\word{w} \word{j})) \word{i} + ((\word{v} \word{i}) \shuprod \word{w}) \word{j}, \\
    \word{w} \shuprod \emptyword &= \emptyword \shuprod \word{w} = \word{w},
\end{align*}
and extended to $\eTA$ by linearity. The following proposition establishes the shuffle property for EFM signatures.

\begin{proposition}\label{prop:shuffle}
    If $\bell_1, \bell_2 \in \TA$, then $\bell_1 \shuprod\bell_2 \in \TA$ and
    \begin{equation}
        \bracketssigX{\bell_1}\bracketssigX{\bell_2} = \bracketssigX{\bell_1\shuprod\bell_2}.
    \end{equation}
\end{proposition}
\begin{proof}
    The proof follows from Proposition~\ref{Prop:Markovian_lift} and the shuffle property of the standard signature.

\end{proof}

\subsection{Universal approximation theorem} 
The goal of this section is to show that a sufficiently large class of functions of path $F((X_t)_{t \leq T})$ can be approximated by a linear functional of the EFM signature $\bracketssigX[T][\widehat X]{\bell}$ of the time-augmented path $\widehat X_t = (t, X_t)$. In particular, if the approximation holds for $F$, the time-invariant output signal \eqref{eq:time_invariance} is approximated uniformly in $t\in\R$, i.e.,
\begin{equation}
    Y_t = F((X_s)_{s\leq t}) \approx \bracketssigX[t][\widehat X]{\bell}, \quad t \in\R.
\end{equation}

As for signatures, this result is based on the Stone--Weierstrass approximation theorem \cite[Theorem 7.32]{rudin1976principles}.

\begin{theorem}[Stone--Weierstrass]
    Let $K \subset \mathcal{X}$ be a compact set and let $\mathcal{A}$ denote the class of functions $f\colon\ \mathcal{X} \to \R$ satisfying:
    \begin{enumerate}
        \item Every $f \in \mathcal{A}$ is continuous;
        \item the constant function $1 \in \mathcal{A}$;
        \item if $f, g \in \mathcal{A}$, then $fg \in \mathcal{A}$;
        \item for $x, y \in \mathcal{X}$, such that $x \not= y$, there exists $f \in \mathcal{A}$, such that $f(x) \not= f(y)$.
    \end{enumerate}
    Then, for any continuous function $F\colon\ \mathcal{X} \to \R$ and for any $\epsilon > 0$, there exists $f \in \mathcal{A}$, such that $|F(x) - f(x)| < \epsilon$ for all $x \in K$.
\end{theorem}

Since for standard signatures on a finite interval $[0, T]$, the natural path space is the space of rough paths $\mathscr{C}^{\alpha}([0, T], \R^d)$, see, for example, \cite{cuchiero2022theocalib}, we will use the space of weighted weakly geometric rough paths  $\mathcal{X} = \mathscr{C}^{\alpha, \brho-\blambda}_{g}((-\infty, T], \R^d)$, with $0 \prec \brho \prec \blambda$, i.e., the rough paths in $\mathscr{C}^{\alpha, \brho-\blambda}((-\infty, T], \R^d)$ satisfying
\[
    \mathrm{Sym}(\mathbb{X}^{2}_{s,t}) =
    \frac{1}{2}\mathbb{X}^{1}_{s,t} \otimes \mathbb{X}^{1}_{s,t},
\]
in order to ensure the shuffle property of the signature, and we consider functionals
\begin{equation}\label{eq:sig_as_rp_functional}
    F\colon\ \mathscr{C}^{\alpha, \brho-\blambda}_{g}((-\infty, T], \R^d) \to \R, \quad \mathbf{X}^T \mapsto F(\mathbf{X}^T),
\end{equation}
continuous in the exponentially weighted rough path topology. Since {$\mathbf{X}^T$} contains all the increments of {$X^T$}, the set of functionals of the form \eqref{eq:sig_as_rp_functional} contains all dependencies of the form {$Y_t = F((X_t^T - X_s^T)_{s \leq t})$}.

\begin{theorem}\label{thm:UAT} 
    Consider a compact set $K \subset \mathscr{C}^{\alpha, \brho-\blambda}_{g}((-\infty, T], \R^d)$ and a continuous function 
    $$
    F \colon\ \mathscr{C}^{\alpha, \brho-\blambda}_{g}((-\infty, T], \R^d) \to \R.
    $$ 
    Then, for any $\epsilon > 0$, there exists $\bell \in \TA$ such that, for all ${\mathbf{X}^T} \in K$,
    \[
        \left| F({\mathbf{X}^T}) - \bracketssigX[T][\widehat X]{\bell} \right| < \epsilon.
    \]
\end{theorem}
\begin{proof}
    The proof consists in the application of the Stone--Weierstrass theorem.
    As the class $\mathcal{A}$, we choose the class of finite linear functionals of the EFM signature:
\begin{equation}
    \mathcal{A}:=\left\{\mathbf{X} \mapsto \bracketssigX[T][\widehat X]{\bell}\colon\ \bell\in\TA \right\}.
\end{equation}

Each functional in this class is continuous by Theorem~\ref{thm:rough_fm}.
The second condition of the Stone--Weierstrass theorem is trivially satisfied since $\bracketssigX[t][\hat X]{\emptyword} \equiv 1$. The third condition holds due to the shuffle property given by Proposition~\ref{prop:shuffle}. The point-separation condition 4 follows from the uniqueness result, Corollary~\ref{cor:uniqueness_time_aug}.
\end{proof}

This result can be seen as a combination of the classical universal approximation theorem for signatures and Theorem~1 in the paper by \cite{Boyd1985FadingMA} considering the approximation of operators with fading memory, which is essentially the continuity in the weighted distance.

\begin{sqremark}
    One can establish a similar result for regular processes by choosing the space of absolutely continuous paths with finite norm $\|X\|_{\mathrm{BV}, \blambda} < \infty$. The continuity of the signature map in this case follows from Proposition~\ref{prop:fm_bv_case}.
\end{sqremark}

\subsection{Characteristicness of the EFM signature}
We recall that a feature map $\Phi$ is called characteristic if the map
\[
    \mu \mapsto \E^{\mu}[\Phi(X)]
\]
is injective. In other words, the expectation of $\Phi(X)$ characterizes the law of the stochastic process $X$. It was shown by \citet{chevyrev2016CharacteristicFunctionsMeasures} that the standard signature $\mathbb{X}_T$ characterizes the law of the path $X = (X_t)_{t \in [0,T]}$ up to tree-like equivalence, provided that the expected signature $\E[\mathbb{X}_T]$ has infinite radius of convergence; that is, the series
\[
    z \mapsto \sum_{n \geq 0} \|\E[\mathbb{X}^n_T]\| z^n
\]
has infinite radius of convergence.

The representation of the EFM signature as a classical signature given by Proposition~\ref{Prop:Markovian_lift} allows us to derive the characteristic property of the EFM signature.

\begin{proposition}\label{prop:momentss}
    Let $\P$ and $\Q$ be probability measures on the path space {$\mathscr{C}_g^{\alpha, \brho-\blambda}((-\infty, T], \R^d)$}. If their expected time-augmented EFM signatures are equal and have infinite radii of convergence, i.e.,
    \begin{equation}\label{eq:moments_equality}
        \E^\P[\ssigX[T][\widehat{X}]] = \E^\Q[\ssigX[T][\widehat{X}]],
    \end{equation}
    then $\P = \Q$.
\end{proposition}

\begin{proof}
    It follows from \cite[Proposition~6.1]{chevyrev2016CharacteristicFunctionsMeasures} that $\E^\P[\ssigX[T][\widehat{X}]]$ characterizes the law of $\ssigX[T][\widehat{X}]$. The uniqueness result in Corollary~\ref{cor:uniqueness_time_aug} then implies that it also characterizes the law of the path $X$.
\end{proof}

\begin{sqremark}
    It is possible to remove the infinite-radius assumption by using the robust signature introduced by \citet{Chevyrev2022}. Indeed, if $\Psi$ denotes the tensor normalization given in \cite[Definition~12]{Chevyrev2022}, then the expected robust EFM signature $\E[\Psi(\ssigX[T][\widehat{X}])]$ characterizes the law of the process $X$ \citep[Theorem~26]{Chevyrev2022}.
\end{sqremark}

\section{EFM signature of the time-augmented Brownian motion}\label{sect:SSig_BM}

In this section, we focus on a particularly interesting case where the path is the time-augmented standard Brownian motion $\widehat{W}_t = (t, W_t^1, \ldots, W_t^d)$. {Recall that the EFM signature of $\widehat{W}$ is well-defined, see Remark~\ref{rem:BM_EFM_sig_well_def}.} While its standard signature $\sig[]$ inherits several key properties of Brownian motion---namely, $\sig[0] = \emptyword$, and $\sig[t, t+h]$ is independent of $\F_t$ with a law depending only on $h$---the EFM signature $\ssig[]$ can be viewed as an Ornstein--Uhlenbeck process on the Lie group. Indeed, beyond the mean-reverting dynamics 
\begin{equation}\label{eq:fm_sig_bm_sde}
    d\ssig[t] = -\G \ssig[t]\,dt + \ssig[t]\otimes\circ d\widehat{W}_t, \quad t\in\R,
\end{equation}
and the stationarity established in Theorem~\ref{thm:stationarity}, we will show that it is an ergodic Markov process. Moreover, we prove that the non-stationary version $\ssig[0, t]$ converges exponentially fast to the stationary distribution in the Wasserstein distance. Finally, we compute the expected EFM signature in both the stationary and non-stationary settings, in analogy with the celebrated result of \citet{fawcett}.

\subsection{Markov property for $\ssig[]$}

The EFM signature $\ssig[]$ can be considered as an adapted $\eTA[d+1]$-valued stochastic process. We equip the state space $\eTA[d+1]$ with the norm  
$$
\|\bell\|_1 := \sum_{n \geq 0} \sum_{|\word{v}| = n} |\bell^\word{v}|,
$$
and denote 
$$
{L}^1(\eTA[d+1]) = \left\{ \bell\in\eTA[d+1] \colon\ \|\bell\|_1 < +\infty\right\}.
$$
The space ${L}^1(\eTA[d+1])$ is a Banach space isomorphic to the standard $\ell^1$ space. Moreover, endowed with the tensor product $\otimes$, it forms a Banach algebra: for all $\bell, \bp \in {L}^1(\eTA[d+1])$,
\begin{equation}\label{eq:banach_alg_ppty}
   \|\bell \otimes \bp\|_1 \leq \|\bell\|_1 \|\bp\|_1. 
\end{equation}

We denote by $\mathcal{B}({L}^1(\eTA[d+1]))$ the Borel sigma-algebra generated by the open sets of ${L}^1(\eTA[d+1])$.

The following lemma ensures that $\ssig$ and $\ssig[0, t]$ are indeed elements of ${L}^1(\eTA[d+1])$ for any $t \in \R$.

\begin{lemma}\label{lemma:L1_bound}
    There exists a constant $M_{\blambda} > 0$, such that
    $$
    \E[\|\ssig[t]\|_1] \leq M_{\blambda}, \quad t \in\R \quad \text{ and } \quad \sup_{t\geq 0}\E[\|\ssig[0, t]\|_1] \leq M_{\blambda}.
    $$
\end{lemma}
\begin{proof}
    The proof is given in Appendix~\ref{app:lemma_l1_proof}
\end{proof}

Lemma~\ref{lemma:L1_bound} states that $\ssig[] = (\ssig)_{t \in \R}$ can be considered as a general Markov process with state space $({L}^1(\eTA[d+1]),\, \mathcal{B}({L}^1(\eTA[d+1])))$, as shown in the following proposition.

\begin{theorem}\label{thm:Markov}
    The process $\ssig[] = (\ssig)_{t \in\R}$ is a Markov process in the sense that for any measurable bounded function $f: L^1(\eTA[d+1])  \to \R$, the Markov property is verified:
    \begin{equation}
        \E\left[f(\ssig) \, \Big| \, \F_s\right] = \E\left[f(\ssig) \, \Big| \, \ssig[s]\right], \quad t \geq s.
    \end{equation}
\end{theorem}
\begin{proof}
    For a measurable bounded $f$ and $t \geq 0$, we have
    \begin{equation}
        \E\left[f(\ssig) \, \Big| \, \F_s\right] = \E\left[f(\D[t - s]\ssig[s]\otimes\ssig[s, t]) \, \Big| \, \F_s\right] = \E\left[f(\D[t - s]\bell\otimes\ssig[s, t]) \, \Big| \, \F_s\right] \bigg|_{\bell = \ssig[s]} = \E\left[f(\D[t - s]\bell\otimes\ssig[s, t])\right] \bigg|_{\bell = \ssig[s]}.
    \end{equation}
    Here we have used the Chen's identity, the measurability of $\ssig[s]$ with respect to $\F_s$ and finally, independence of $\ssig[s, t]$ and $\F_s$.
\end{proof}

\begin{sqremark}
    In the proof above, the submultiplicativity \eqref{eq:banach_alg_ppty} is implicitly used to ensure that the product \( \D[t - s]\bell \otimes \ssig[s, t] \in L^1(\eTA[d+1]) \)  for all \( \bell \in L^1(\eTA[d+1]) \), so that the expression \( f(\D[t - s]\bell \otimes \ssig[s, t]) \) is meaningful.
\end{sqremark}

\subsection{Ergodicity of $\ssig[]$}

As we have already shown in Theorem~\ref{thm:stationarity} and Theorem~\ref{thm:Markov}, $\ssig$ is a stationary Markov process. The process $(\ssig[0, t])_{t \geq 0}$ corresponds to its nonstationary version, starting from $\ssig[0, 0] = \emptyword$ (i.e., we assume that the underlying path is constant on the negative half-line: $\widehat{W}_t \equiv 0$ for $t < 0$). As a one-dimensional analogy, one can think of the stationary Ornstein--Uhlenbeck process $\int_{-\infty}^t e^{-\lambda(t-s)}\,dW_s$ and its nonstationary {counterpart} $\int_{0}^t e^{-\lambda(t-s)}\,dW_s$.

A natural question arises: How quickly does the law of $\ssig[0, t]$ converge to the stationary distribution $\mathcal{L}(\ssig)$?

In this section, we will show that, exactly as in the Ornstein--Uhlenbeck case, this convergence is geometric in the Wasserstein distance 
\begin{equation}\label{eq:Wasser_def}
    \mathcal{W}_1(\mu, \nu) = {\inf_{\pi\in \Pi(\mu, \nu)}\int\|\mathbbm{x} - \mathbbm{y}\|_1\,\pi(d\mathbbm{x}, d\mathbbm{y})}, \quad \mu, \nu \in \mathcal{P}(L^1(\eTA[d+1])),
\end{equation}
where $\Pi(\mu, \nu)$ stands for a set of all coupling of $\mu$ and $\nu$.

More precisely, we will prove that the dependence on the initial condition decays exponentially with rate $\min_{i}\blambda^i$. In particular, this will imply the uniqueness of the invariant distribution and the ergodicity of the EFM signature process.

\begin{theorem}\label{thm:ergodicity}
Let $W = (W_t)_{t \in \R}$ denote a Brownian motion and let $x = (x_t)_{t \leq 0}$ be a continuous semimartingale. Consider the path obtained by concatenation of $(x_t)_{t \leq 0}$ with $(W_t)_{t \geq 0}$ at $t = 0$:
$$
X_t = \begin{cases}
    x_t, \ &t \leq 0, \\
    x_0 + W_t, \ &t > 0,
\end{cases} \qquad
$$ Assume {that $x$ and $(W_t - W_0)_{t \geq 0}$ are independent} and that the EFM signature $\widehat{\mathbbm{x}}^{\blambda}_t$ of the time-augmented path $\widehat{x}_t = (t, x_t)$ is well-defined as a process in ${L}^1(\eTA[d+1])$ for $t \leq 0$ with $\E[\|\widehat{\mathbbm{x}}^{\blambda}_0\|_1] < \infty$. Let $\ssigX[t][\widehat{X}]$ denote the EFM signature of $\widehat{X}_t = (t, X_t), \, t\in\R$ and set  $\mu_t = \mathcal{L}(\ssigX[t][\widehat{X}]),\, t \geq 0$. Then,
there exists a unique ergodic measure $\mu_* = \mathcal{L}(\ssig[t])$ and a constant $M_{\blambda} > 0$ such that
\begin{equation}\label{eq:wasserstein_bound}
    \mathcal{W}_1(\mu_t, \mu_*) \leq M_{\blambda}\left(\E\left[\|\widehat{\mathbbm{x}}^{\blambda}_0\|_1\right] + M_{\blambda}\right)e^{-\blambda^{\mathrm{min}} t}, 
\end{equation}
where $\blambda^{\mathrm{min}} = \min\limits_i \blambda^i$. In particular, taking $x \equiv 0$, one obtains
\begin{equation}
    \mathcal{W}_1(\mathcal{L}(\ssig[0, t]),\, \mathcal{L}(\ssig[t])) \leq M_{\blambda}(1+M_{\blambda}) e^{-\blambda^{\min} t}, 
\end{equation}
\end{theorem}
\begin{proof}
    First, note that $\ssigX[t][\widehat{X}] = \widehat{\mathbbm{x}}^{\blambda}_t$ for $t \leq 0$.
    By Chen's identity \eqref{eq:stat_chen}, we have, for $t \geq 0$,
\begin{equation}
    \ssigX[t][\widehat{X}] = \D[t]\ssigX[0][\widehat{X}]\otimes\ssig[0, t], \quad \ssig[t] = \D[t]\ssig[0]\otimes\ssig[0, t],
\end{equation}
and 
\begin{equation}
    \ssigX[t][\widehat{X}] -  \ssig[t] = \D[t]\left(\ssigX[0][\widehat{X}] - \ssig[0]\right)\otimes\ssig[0, t].
\end{equation}
By the submultiplicative property \eqref{eq:banach_alg_ppty} of the norm $\|\cdot \|_1$, 
\begin{equation}\label{eq:exp_conv_wasser}
    \|\ssigX[t][\widehat{X}] -  \ssig[t]\|_1 \leq \|\D[t](\ssigX[0][\widehat{X}] - \ssig[0])\|_1 \|\ssig[0, t]\|_1.
\end{equation}

Conditioning with respect to $x$ and then taking the expectation, we obtain
\begin{equation}\label{eq:ergodicity_ineq}
\begin{aligned}
    \E\left[\|\ssigX[t][\widehat{X}] -  \ssig[t]\|_1\right] &\leq \E\left[\|\D[t](\ssigX[0][\widehat{X}] - \ssig[0])\|_1\right]\E\left[\|\ssig[0, t]\|_1\right] \\
    &\leq e^{-\blambda^{\mathrm{min}} t}\,\E\left[\|\ssigX[0][\widehat{X}] - \ssig[0]\|_1 \right]\E\left[\|\ssig[0, t]\|_1\right] \\
    &\leq M_{\blambda}\left(\E\left[\|\ssigX[0][\widehat{X}]\|_1\right] + M_{\blambda}\right)e^{-\blambda^{\mathrm{min}} t},
\end{aligned}
\end{equation}

where we applied Lemma~\ref{lemma:L1_bound} and the trivial observation that $\|\D\bell\|_1 \leq e^{-\blambda^{\mathrm{min}} h}\|\bell\|_1$ for all $\bell \in \eTA[d+1]$ such that $\bell^{\emptyword} = 0$.

By definition \eqref{eq:Wasser_def}, the same bound holds for $\mathcal{W}_1(\mu_t, \mu_*)$. In particular, \eqref{eq:ergodicity_ineq} implies the uniqueness of the invariant measure $\mu_*$. Hence, this invariant measure is ergodic \cite[Corollary 5.12]{hairer2006ergodic}. 
\end{proof}

\begin{corollary}
    The continuous-time version of the Birkhoff--Khinchin ergodic theorem \cite[Corollary 10.9]{kallenberg1997foundations} then implies that, for any measurable integrable function $f\colon \eTA[d+1] \to \R$,
\begin{equation}\label{eq:time-average}
    \lim_{T \to \infty} \frac{1}{T} \int_0^T f(\ssig)\,dt = \E\left[f(\ssig[0])\right].
\end{equation}

The convergence in \eqref{eq:time-average} is important for applications, as it allows for the estimation of relevant statistics (e.g., signature moments) and enables learning the law of the stationary process from a single trajectory observed over a sufficiently long time horizon.
\end{corollary}

\subsection{Expected EFM signature of the Brownian motion.}
The expected signature of Brownian motion is given by Fawcett's formula (see \cite{fawcett}). 

In this section, we aim to compute the expected EFM signatures:
\begin{enumerate}
    \item $\Essig[] := \ssigE[t]$, corresponding to the stationary distribution;
    \item $\Essig[t - s] := \ssigE[s, t]$, the expected EFM signature on $[s, t]$.
\end{enumerate}
Since $\mathcal{L}(\ssig[s, t])$ converges weakly to the stationary law $\mathcal{L}(\ssig[t])$ as $s \to -\infty$, and since $\sup\limits_{t \geq s} \E[\|\ssig[s, t]\|_1] \leq M_{\blambda}$ by Lemma~\ref{lemma:L1_bound}, the expected EFM signatures converge weakly in $L^1(\eTA[d+1])$:
\begin{equation}
    \Essig[h] \underset{h \to \infty}{\rightharpoonup} \Essig[].
\end{equation}

The importance of the Magnus expansion for the computation of expected signatures in a general semimartingale context was highlighted by \citet*{friz2022unified, friz2024expected}. We will show that for the time-augmented Brownian motion path, the expected EFM signature computation problem reduces to the computation of an EFM signature of a linear path, a problem we have already treated in Subsection~\ref{sect:linear_path} using the Magnus exponential. Indeed, let us denote $\word{\tilde 0} := \word{0}$ and $\word{\tilde i} := \word{ii}$ for $\word{i} \in \{\word{1}, \ldots, \word{d}\}$ and the alphabet $\tilde A_{d + 1} = \{\word{\tilde 0}, \word{\tilde 1}, \ldots, \word{\tilde{d}}\}$. Note that the words formed by the letters of $\tilde A_{d + 1}$ form a subset of words in the initial alphabet $ A_{d + 1} = \{\word{ 0}, \word{ 1}, \ldots, \word{{d}}\}$ that correspond to the non-zero elements of the expected EFM signature. 

Let $\mathbb{T}^{\tilde\blambda}$ denote the EFM signature of the linear path $T_t = t\cdot (1, \frac{1}{2}, \ldots, \frac{1}{2})$ corresponding to the new alphabet $\tilde A_{d + 1}$ and $\tilde\blambda = (\blambda^{\word{\tilde{0}}}, \blambda^{\word{\tilde{1}}}, \ldots, \blambda^{\word{\widetilde{d}}})$. We recall that this signature is given in closed form by Theorem~\ref{thm:linear_path_magnus}.

\begin{theorem}\label{thm:expected_sig}
    Let $\Essig[]$ and $\Essig[t-s]$ be defined as above. Then,
    \begin{equation}\label{eq:expected_fm_sig}
    \begin{aligned}
        \Essig[t-s] = \exp^\otimes\left(\sum_{k \geq 1}\bm{\Omega}_k^{\Essig[]}(s, t)\right),\quad
        \Essig[] = \exp^\otimes\left(\sum_{k \geq 1}\bm{\Omega}_k^{\Essig[]}(-\infty, t)\right) = \sum_{k=0}^\infty {\bf H}_{\left(\word{0} +  \frac12\sum\limits_{i=1}^{d}\word{ii}\right)}^k\emptyword,
    \end{aligned}
    \end{equation}   
    where ${\bm{\Omega}}_k^{\Essig[]}(s, t)$ is given by \eqref{eq:magnus_omega} with $\bm A(u) = \D[t - u]\left(\word{0} +  \dfrac12\sum\limits_{i=1}^{d}\word{ii}\right)$, and the operator ${\bf H}$ is defined in Theorem~\ref{thm:linear_path_magnus}. The non-zero coefficients are given explicitly by
    \begin{equation}
        (\Essig[t-s])^{\word{\tilde{i}_1 \ldots \tilde{i}_n}} = \mathbb{T}_{0, t-s}^{\tilde\blambda, \word{\tilde{i}_1 \ldots \tilde{i}_n}}, \quad (\Essig[])^{\word{\tilde{i}_1 \ldots \tilde{i}_n}} = \mathbb{T}_{0}^{\tilde\blambda, \word{\tilde{i}_1 \ldots \tilde{i}_n}}.
    \end{equation}
    Moreover, if $\blambda^0 = 2\blambda^1 = \ldots = 2\blambda^d =: \lambda$, then 
    \begin{equation}
        \Essig[t-s] = \exp^\otimes\left(\left(\frac{1 - e^{-\lambda(t - s)}}{\lambda}\right)  \left(\word{0} +  \frac12\sum\limits_{i=1}^{d}\word{ii}\right)\right), \quad \Essig[] = \exp^\otimes\left({\frac{1}{\lambda}\left(\word{0} +  \frac12\sum\limits_{i=1}^{d}\word{ii}\right)}\right).
    \end{equation}
\end{theorem}
\begin{proof}
    Recall that the signature dynamics reads:
    \begin{equation}
        d\ssig[s, t] = -\G\ssig[s, t] dt + \ssig[s, t]\otimes\word{0}\,dt + \sum_{i=1}^{d}\ssig[s, t]\otimes\word{i}\circ dW_t^{i}, \quad t \geq s.
    \end{equation}
    Converting the Stratonovich integrals to Itô integrals, we deduce
    \begin{equation}
        d\ssig[s, t] = -\G\ssig[s, t] dt + \ssig[s, t]\otimes\left(\word{0} +  \dfrac12\sum_{i=1}^{d}\word{ii}\right)dt + \sum_{i=1}^{d}\ssig[s, t]\otimes\word{i}\, dW_t^{i}, \quad t \geq s.
    \end{equation}
    Using the variation of constant and taking the expectation, one obtains
    \begin{equation}\label{eq:expected_sig_equation}
        \ssigE[s, t] = \emptyword +\int_s^{t}\D[t-u]\left(\ssigE[s, u]\otimes\left(\word{0} +  \dfrac12\sum_{i=1}^{d}\word{ii}\right)\right)du, \quad t \geq s.
    \end{equation}
    Since $\left(\word{0} +  \dfrac12\sum\limits_{i=1}^{d}\word{ii}\right) = \left(\word{\tilde{0}} +  \dfrac12\sum\limits_{i=1}^{d}\word{\tilde{i}}\right)$, the solution of \eqref{eq:expected_sig_equation} corresponds exactly to the stationary signature of a linear path $T_t = (t, \frac{t}{2}, \ldots, \frac{t}{2})$ with $\tilde\blambda = (\blambda^{\word{\tilde{0}}}, \blambda^{\word{\tilde{1}}}, \ldots, \blambda^{\word{\tilde{d}}})$ indexed by the words in the alphabet $\{\word{\tilde{1}}, \ldots, \word{\tilde{d}}\}$.
    We conclude by applying Theorem~\ref{thm:linear_path_magnus}.
\end{proof}

\begin{sqremark}
    One can equivalently use the recursive formulas
obtained by a straightforward computation: for all words $\word{v}$ and $\word{i}, \word{j} \in \{ \word{1}, \ldots, \word{d}\}$ such that $\word{i} \neq \word{j}$, we have 
    \begin{align}
        \Essig[][\lambda, \word{v0}] = \dfrac{1}{\blambda^{\word{v0}}}\Essig[][\lambda, \word{v}], \quad \Essig[][\lambda, \word{vii}] = \dfrac{1}{2}\dfrac{1}{\blambda^{\word{vii}}}\Essig[][\lambda, \word{v}], \quad \Essig[][\lambda, \word{v0i}] = 0,  \quad \Essig[][\lambda, \word{vij}] = 0,
    \end{align}
    and
    \begin{align}
        \Essig[t][\lambda, \word{v0}] = \int_0^te^{-\blambda^{\word{v0}}(t-s)}\Essig[s][\lambda, \word{v}]\,ds, \quad \Essig[t][\lambda, \word{vii}] = \frac12\int_0^te^{-\blambda^{\word{vii}}(t-s)}\Essig[s][\lambda, \word{v}]\,ds, \quad \Essig[t][\lambda, \word{v0i}] = 0, \quad \Essig[t][\lambda, \word{vij}] = 0.
    \end{align}
     The coefficients of the expected EFM signature are no longer symmetric in the sense that
    \[
    \Essig[][\lambda, \word{ii0}] = \dfrac{1}{{2}\blambda^{\word{ii0}}\blambda^{\word{ii}}} \not= \dfrac{1}{{2}\blambda^{\word{0ii}}\blambda^{\word{0}}} = \Essig[][\lambda, \word{0ii}], \quad \word{i} \in \{1, \ldots, d\},
    \]
    as they are in Fawcett's formula until one has $\blambda^0 = 2\blambda^1 = \ldots = 2\blambda^d$. The choice $\blambda^0 = 2\blambda^i$ comes naturally from the fact that the {variance} of the Ornstein--Uhlenbeck processes
    $$
    \widehat{\mathbb{W}}_t^{\blambda, \word{i}} = \int_{-\infty}^t e^{-\blambda^i(t-u)}dW_u^i, \quad i = 1, \ldots, d,
    $$
    appearing at the first level of the EFM signature, equals in this case
    $$
    \widehat{\mathbb{W}}_t^{\blambda, \word{0}} = \int_{-\infty}^t e^{-\blambda^0(t-u)}du = \int_{-\infty}^t e^{-2\blambda^i(t-u)}du.
    $$
\end{sqremark}

\begin{sqexample}   
As an illustration of the expected EFM signature, we consider a prediction problem. Suppose that the output signal $Y = (Y_t)_{t \in \R}$ is given as a linear combination of EFM signature elements:
\begin{equation}
    Y_t = \bracketssig{\bell}, \quad t \in\R.
\end{equation}

As follows immediately from Chen's identity \eqref{eq:stat_chen}, the conditional expectation and variance are given by
\begin{equation}\label{eq:prediction}
    \E[Y_{t+h} \mid \F_t] = \left\langle \bell, \D\ssig \otimes \Essig[h] \right\rangle, \quad 
    \mathbb{V}\mathrm{ar}[Y_{t+h} \mid \F_t] = \left\langle \bell\shupow{2}, \D\ssig \otimes \Essig[h] \right\rangle - \left\langle \bell, \D\ssig \otimes \Essig[h] \right\rangle^2.
\end{equation}
Figure~\ref{fig:prediction} illustrates the use of this prediction model for a stationary signal
$$
Y_t = \dfrac{1}{1.01 + \sin(U_t)}, \quad dU_t = -\kappa U_t\,dt + \nu\,dW_t,
$$
with $\kappa = 15$ and $\nu = 1.5$. First, we learn the coefficients $\bell$ via standard linear regression against the EFM signature $\ssig$, truncated at order $5$, over the interval $[0, 1]$. The signature parameters are chosen as $\blambda = (10, 10)$. We then apply the prediction formula~\eqref{eq:prediction} to estimate the conditional expectation of the signal given $\mathcal{F}_1$. The shaded region represents one standard deviation.

\begin{figure}[H]
    \begin{center}
    \includegraphics[width=0.6\linewidth]{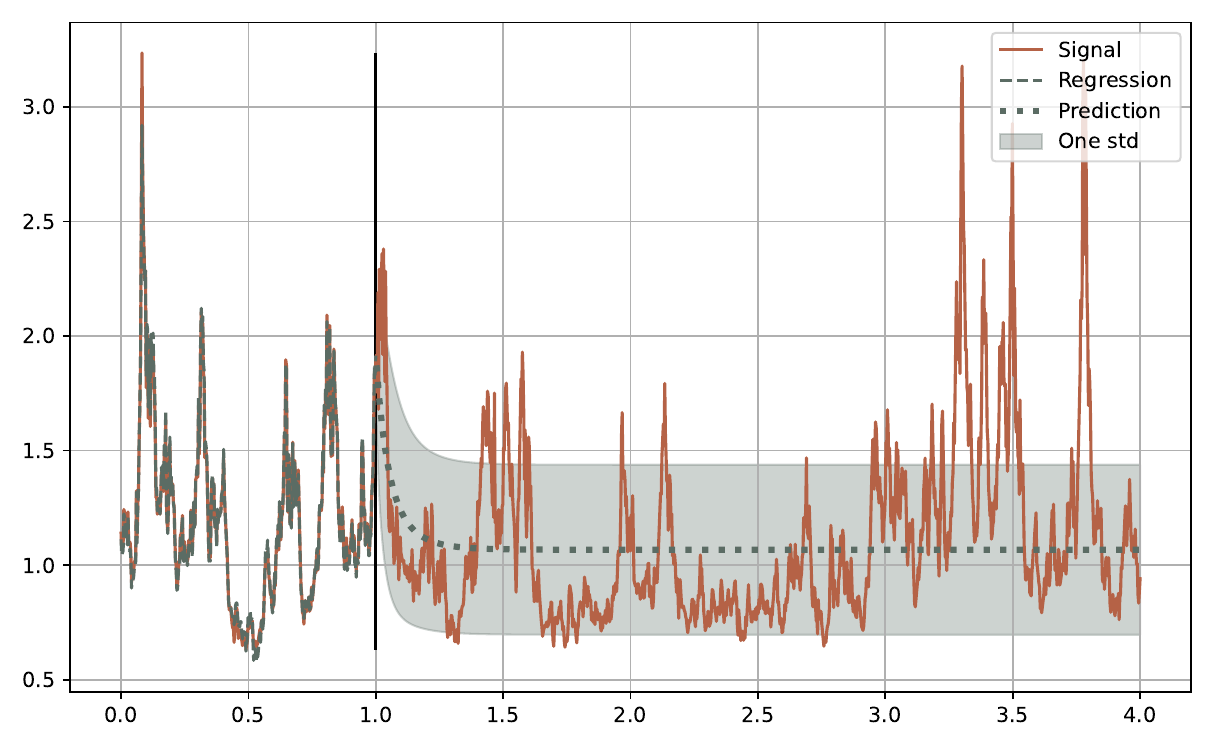}
    \caption{Trajectory of the signal $Y_t$ (brown), the regression process $\bracketssig{\bell}$ (dashed green), and the prediction process (dotted green) defined by~\eqref{eq:prediction}.}
    \label{fig:prediction}
    \end{center}
\end{figure}

\end{sqexample}

\section{Linear functionals of EFM signatures}\label{sect:linear_forms}

As we saw in previous sections, the linear functional of the EFM signature $\bracketssig[t]{\bell}$ of the time-augmented Brownian motion, with $\bell \in \TA$, is a natural candidate for time-invariant path-dependent modeling, since: 
\begin{itemize}
    \item $\ssig[t]$ is a stationary Markov process by Theorem~\ref{thm:Markov};
    \item $\ssig[t]$ uniquely determines the path $\widehat{W}$ up to time $t$ by Corollary~\ref{cor:uniqueness_time_aug};
    \item functionals of the form $\bracketssig[t]{\bell}$ are universal approximators of path functionals by Theorem~\ref{thm:UAT};
    \item the expected EFM signature $\E[\ssig[t]]$ can be computed explicitly by Proposition~\ref{thm:expected_sig}.
\end{itemize}

This section is dedicated to studying the properties of linear functionals of the EFM signature $\bracketssig{\bell}$. For simplicity and clarity of the proposed proofs, we establish Itô's decomposition only in the case where the coefficient $\bell \in \TA$ is finite. Analogous results for possibly infinite linear functionals of the signature of extended Brownian motion can be found in the works by \citet{abijaber2024signature} and \citet*{jaber2024pathdependentprocessessignatures}. We expect that the approach proposed in these works can be extended to the case of EFM signatures.

\subsection{Itô's decomposition}
We start this subsection by considering a general continuous semimartingale path $X$ in $\R^d$. The following definition is useful for obtaining the semimartingale decomposition of linear functionals of the signature.

\begin{definition}
    For a tensor sequence $\bell \in \eTA$ and a word $\word{u}$, the \emph{projection} $\bell\proj{\word{u}}$ is defined as
    \begin{equation}
        \bell\proj{\word{u}} = \sum_{n \geq 0}\sum_{\word{v} \in V_n} \bell^{\word{v}\word{u}}\, \word{v}.
    \end{equation}
\end{definition}
In other words, the projection retains only the coefficients of $\bell$ corresponding to words that end with $\word{u}$, and replaces each such word $\word{v}\word{u}$ by its prefix $\word{v}$. The connection between this projection and differentiation becomes evident in the following decomposition:
\begin{equation}
    \bracketsigX{\bell} = \bracketsigX[0]{\bell} + \sum_{i=1}^d \int_0^t \bracketsigX[s]{\bell\proj{\word{i}}} \circ dX_s^i,
\end{equation}
or, expressed in terms of tensor sequences,
\begin{equation}
    \bell = \bell^\emptyword + \sum_{i=1}^d \bell\proj{\word{i}} \, \word{i}.
\end{equation}
This decomposition leads to the Itô formula for (possibly infinite) linear functionals of the signature of semimartingales, as shown by \cite*[Theorem 3.3]{jaber2024pathdependentprocessessignatures}. In the case of EFM signatures, a similar result holds, with the only difference being the presence of exponential kernels in their definition, which introduces an additional mean-reverting term.

\begin{proposition}\label{prop:ito_general}
    Let $\bell_t \in \TA$ be continuously differentiable (element-wise). Then the linear functional $\bracketssigX{\bell_t}$ satisfies the stochastic differential equation:
    \begin{equation}
        d\bracketssigX{\bell_t} = \bracketssigX{\dot\bell_t - \G\bell_t}\,dt + \sum_{i=1}^d \bracketssigX{\bell_t\proj{\word{i}}} \circ dX_t^i.
    \end{equation}
\end{proposition}

\begin{proof}
    Recall that by \eqref{eq:sig_sde}, for any non-empty word $\word{v} = \word{i_1 \ldots i_n}$, the corresponding EFM signature component satisfies
    \begin{equation}
        d\sigX^{\lambda, \word{v}} = -\blambda^{\word{v}} \sigX^{\lambda, \word{v}}\,dt + \sigX^{\lambda, \word{v'}} \circ dX_t^{i_n},
    \end{equation}
    where $\word{v'} := \word{i_1 \ldots i_{n-1}}$. 
    The result follows by plugging this expression into the definition
    \[
        \bracketssigX{\bell_t} = \sum_{n \geq 0} \sum_{\word{v} \in V_n} \bell_t^{\word{v}}\, \sigX^{\lambda, \word{v}},
    \]
    and applying the product rule:
    \begin{align}
        d\bracketssigX{\bell_t} &= \sum_{n \geq 0}\sum_{\word{v} \in V_n}\dot\bell_t^{\word{v}}\sigX^{\blambda, \word{v}}dt +
        \sum_{n \geq 0}\sum_{\word{v} \in V_n}\bell_t^{\word{v}}d\sigX^{\blambda, \word{v}} \\
        &= \sum_{n \geq 0}\sum_{\word{v} \in V_n}\left(\dot\bell_t^{\word{v}}\sigX^{\blambda, \word{v}} - \lambda ^{\word{v}} \bell_t^{\word{v}}  \sigX^{\blambda, \word{v}} \right)dt +
        \sum_{n \geq 0}\sum_{\word{v} \in V_n}\bell_t^{\word{v}}\sigX[t]^{\blambda, \word{v'}}\circ dX_t^{i_n} \\
        &= \bracketssigX{\dot\bell_t - \G \bell_t}dt + \sum_{i=1}^d \bracketssigX{\bell_t\proj{\word{i}}}\circ dX_t^i.
    \end{align}
\end{proof}
This proposition allows us to derive the Itô formula for the EFM signature $\ssig$ of the extended standard Brownian motion $\widehat{W}_t = (t, W_t^1, \ldots, W_t^d)$.

\begin{proposition}\label{prop:Ito_BM}
    Let $\bell_t \in \TA[d+1]$ be continuously differentiable (element-wise). Then the process $\bracketssig{\bell_t}$ satisfies:
    \begin{equation}\label{eq:ito_decomposition}
        d\bracketssig{\bell_t} = \bracketssig{\dot\bell_t - \G\bell_t + \bell_t\proj{0} + \dfrac{1}{2}\sum_{i = 1}^{d}\bell_t\proj{ii}}\,dt + \sum_{i=1}^{d}\bracketssig{\bell_t\proj{i}}\,dW_t^{i}.
    \end{equation}
\end{proposition}

\begin{proof}
    Applying Proposition~\ref{prop:ito_general} to the EFM signature $\ssig$ yields:
    \begin{equation}\label{eq:Strato_decomp}
        d\bracketssig{\bell_t} = \bracketssig{\dot\bell_t - \G\bell_t + \bell_t\proj{0}}\,dt + \sum_{i=1}^{d}\bracketssig{\bell_t\proj{i}} \circ dW_t^{i}.
    \end{equation}
    To convert the Stratonovich integral to an Itô integral, we compute the quadratic covariation between $W^{i-1}$ and $\bracketssig{\bell_t\proj{i}}$. Applying again Proposition~\ref{prop:ito_general} to $\bracketssig{\bell_t\proj{i}}$, we find:
    \[
        d\left\langle W^{i}, \bracketssig[\cdot]{\bell_\cdot\proj{i}} \right\rangle_t = \bracketssig{\bell_t\proj{ii}}\,dt.
    \]
    Therefore, rewriting \eqref{eq:Strato_decomp} in Itô form yields the desired result \eqref{eq:ito_decomposition}.
\end{proof}

\begin{sqexample}[OU representation revisited]\label{ex:OU_repr_revisited}
     The Itô decomposition can be used to find an approximation of an Ornstein--Uhlenbeck process $Y_t = \int_{-\infty}^t e^{-\mu(t-s)}\,dW_s$ satisfying
    $$
    dY_t = -\mu Y_t\, dt + dW_t,
    $$
    in the form $\bracketssig{\bell^\mu}$ for fixed $\blambda$ and for arbitrary $\mu > 0$. Indeed, applying Itô's formula and equating the drift and the volatility of $Y_t$ and $\bracketssig{\bell^\mu}$, we obtain
\begin{equation}
\begin{cases}
     -\G \bell^\mu + \bell^\mu \proj{0} = -\mu \bell^\mu , \\
    \bell^\mu \proj{1} = \emptyword.
\end{cases}
\end{equation}
It is clear that $\bell^\mu$ satisfies this condition if it is a solution to the algebraic equation
\begin{equation}\label{eq:OU_algebraic}
    \bell^\mu = (\G(\bell^\mu) - \mu \bell^\mu)\word{0} + \word{1} = \mathcal{B}\bell^\mu + \word{1},
\end{equation}
where the operator $\mathcal{B} : \eTA \to \eTA$ is defined by 
$
\mathcal{B}\colon \bell \mapsto (\G(\bell) - \mu \bell)\word{0}.
$
This is a linear equation which can be solved with the resolvent:
\begin{equation}
    \bell^\mu = \sum_{k \geq 0} \mathcal{B}^k \word{1},
\end{equation}
which corresponds to the Picard iterations for the equation \eqref{eq:OU_algebraic}. Note that this can be rewritten explicitly as
\begin{equation}\label{eq:ou_representation}
    \bell^\mu = \sum_{k \geq 0} c_k \word{1} \word{0} \conpow{k},
\end{equation}
with
\begin{equation}
    \begin{cases}
        c_0 = 1, \\
        c_{k + 1} = (k\blambda^0 + \blambda^1 - \mu) c_k, \quad k \geq 0.
    \end{cases}
\end{equation}
This is the formula we used in the introductory Subsection~\ref{ex:OU_repr}. If we denote
\begin{equation}\label{eq:OU_approx}
    \bell^{\mu, N} = \sum_{k = 0}^N c_k \word{1} \word{0} \conpow{k}, \quad Y_t^N = \bracketssig{\bell^{\mu, N}},
\end{equation}
one can verify that 
$$
d(Y_t - Y_t^N) = -\mu (Y_t - Y_t^N) dt + \bracketssig{c_{N + 1} \word{10}^{\otimes N}},
$$
so that 
$$
\E[(Y_t - Y_t^N)^2] = \E\left[\left(\int_{-\infty}^t e^{-\mu(t-s)} \bracketssig[s]{c_{N + 1} \word{10}^{\otimes N}} ds\right)^2\right] \leq \dfrac{c_{N + 1}^2}{2\mu} \left\langle (\word{10}^{\otimes N}) \shupow{2}, \Essig[] \right\rangle.
$$
A cumbersome but straightforward computation shows that the right-hand side expression converges to $0$ as $N \to +\infty$, so that $Y_t^N \to Y_t$ in $L^2$. 
\end{sqexample}

\subsection{A remark on the characteristic function of linear functional of $\ssig$}\label{sec:cf_riccati}

In this section, we show how the approach introduced in \cite{abijaber2024signature} can be adapted to compute the characteristic functions of the EFM signature:
\begin{equation}
    \phi(\bell) := \E\left[e^{i\bracketssig{\bell}}\right] = \E\left[e^{i\bracketssig[0]{\bell}}\right], \quad t \in \R,
\end{equation}
and
\begin{equation}
    \phi_T(\bell) := \E\left[e^{i\bracketssig[0, T]{\bell}}\right], \quad T > 0,
\end{equation}
for a given $\bell \in \TA[d+1]$.

The following theorem provides a representation of these characteristic functions in terms of the solution to an infinite-dimensional system of mean-reverting Riccati equations. In this characterization, we allow for infinite linear combinations of EFM signature elements. To that end, and only in this subsection, we assume the necessary convergence of the infinite series involved.

\begin{theorem}\label{thm:char}
    Let $T > 0$. Suppose that $\bpsi_t \in \eTA[d+1],\, t \geq 0$, is a smooth solution to the mean-reverting Riccati equation
    \begin{equation}\label{eq:riccati_cf}
    \begin{cases}
        \dot\bpsi_t = -\G\bpsi_t + F(\bpsi_t), \quad t \geq 0,\\
        \bpsi_0 = i\bell,
    \end{cases}
    \end{equation}
    with $F(\bpsi) = \bpsi\proj{0} + \frac12 \sum\limits_{i = 1}^{d}\bpsi\proj{ii} + \frac{1}{2}\sum\limits_{i = 1}^{d}\left(\bpsi\proj{i}\right)\shupow{2}$, such that the processes $M = (M_t)_{t\leq 0}$ and $N^T = (N_t^T)_{t\in[0, T]},$ given by 
    \begin{equation}
        M_t := e^{\bracketssig{\bpsi_{-t}}}, \quad N_t^T = e^{\bracketssig[0, t]{\bpsi_{T-t}}}.
    \end{equation}
    are well-defined and the Itô's formula can be applied.
    Then $M$ and $N^T$ are local martingales. If, in addition, they are true martingales, then the following expression for the characteristic functions holds:
    \begin{align}\label{eq:cf_repr}
        M_t &=  \E\left[e^{i\bracketssig[0]{\bell}} \big| \F_t\right], \quad t \leq 0, \\
        N_t^T &= \E\left[e^{i\bracketssig[0, T]{\bell}} \big| \F_t\right], \quad t \in [0, T].
    \end{align}
    In particular, the characteristic functions are given by
    $$
    \phi(\bell) = M_{-\infty} = \lim_{T\to\infty}e^{\bpsi_{T}^\emptyword} \quad \text{and} \quad \phi_T(\bell) = N_0^T = e^{\bpsi_{T}^\emptyword}.
    $$
\end{theorem}

\begin{proof}
    Applying the Itô's formula to $\bracketssig{\bpsi_{-t}}$, we obtain
    \begin{equation}
        d\bracketssig{\bpsi_{-t}} = \bracketssig{-\dot\bpsi_{-t} - \G\bpsi_{-t} + \bpsi_{-t}\proj{0} + \dfrac12\sum_{i=1}^{d}\bpsi_{-t}\proj{ii}}dt + \sum_{i=1}^{d}\bracketssig{\bpsi_{-t}\proj{i}}dW_t^{i}.
    \end{equation}
    The application of Itô's formula to $M$ for $t \leq 0$ yields 
    \begin{align}
        dM_t &= M_t d\bracketssig{\bpsi_{-t}} + \dfrac{1}{2}M_t\sum_{i=1}^{d}\bracketssig{\bpsi_{-t}\proj{i}}^2dt \\
        &= M_t \bracketssig{-\dot\bpsi_{-t} - \G\bpsi_{-t} + \bpsi_{-t}\proj{0} + \dfrac12\sum_{i=1}^{d}\bpsi_{-t}\proj{ii} + \dfrac12\sum_{i=1}^{d}\left(\bpsi_{-t}\proj{i}\right)\shupow{2}}dt + M_t \sum_{i=1}^{d}\bracketssig{\bpsi_{-t}\proj{i}}dW_t^{i}.
    \end{align}
    Since $\bpsi_t$ satisfies \eqref{eq:riccati_cf}, the $dt$-term vanishes, so that $M$ is a local martingale. The same argument shows that $N^T$ is a local martingale as well.
    
    Now, note that $M_0 = e^{\bracketssig[0]{\bpsi_{0}}} = e^{i\bracketssig[0]{\bell}}$ and $N_T^T = e^{\bracketssig[0, T]{\bpsi_{0}}} = e^{i\bracketssig[0, T]{\bell}}$, so that \eqref{eq:cf_repr} holds for $t = 0$ and $t = T$, respectively. If $M$ and $N^T$ are true martingales, the equality \eqref{eq:cf_repr} will hold for all $t \leq 0$ and $t \in [0, T]$, correspondingly.
    
    By \eqref{eq:cf_repr}, $M$ is uniformly integrable, so that it is closable and we set $M_{-\infty} = \lim\limits_{t \to -\infty} M_t$. We note that
    \[
    \phi(\bell) = M_{-\infty}, \quad \phi_T(\bell) = N_0^T = e^{\bracketssig[0, 0]{\bpsi_{T}}} = e^{\bpsi_T^\emptyword}.
    \]
    Since the law of $\ssig[0, T]$ converges weakly to the law of $\ssig[0]$ as $T \to \infty$, so does the characteristic function:
    \begin{equation}
        e^{\bpsi_T^\emptyword} = \E\left[e^{i\bracketssig[0, T]{\bell}}\right] \underset{T \to\infty}{\longrightarrow} \E\left[e^{i\bracketssig[0]{\bell}}\right]. 
    \end{equation}
    This finishes the proof. 
\end{proof}

We note that the existence of the solution $\bpsi$ satisfying Itô's decomposition \eqref{eq:ito_decomposition} and the true martingality of $M$ and $N^T$ are the assumptions that are difficult to verify even in the one-dimensional case $d = 1$. Some partial results concerning $\eTA$-valued Riccati ODEs were obtained by \cite*{cuchiero2025signaturesdesaffinepolynomial} and \cite*{jaber2024fourierlaplace}.

{\begin{sqexample}
    The equation \eqref{eq:riccati_cf} can be solved numerically using Euler's scheme or the predictor-corrector semi-integrated scheme described in Section~\ref{sect:numerical_ode}. Note that the only difference from the Riccati equation corresponding to the standard signature is the linear term $-\G\bpsi_t$. Appearing thanks to the mean reversion of the EFM signature, it stabilizes the equation and forces all the coefficients except for $\psi_t^{\emptyword}$ to vanish as $t \to \infty$. A numerical illustration in Figure~\ref{fig:cf_plot} shows that the characteristic function $\phi_T(\bell)$, obtained by numerically solving the mean-reverting Riccati equation using the predictor--corrector scheme, converges to the stationary value $\phi(\bell)$ computed by the Monte Carlo method.
\begin{figure}[H]
    \begin{center}
    \includegraphics[width=1\linewidth]{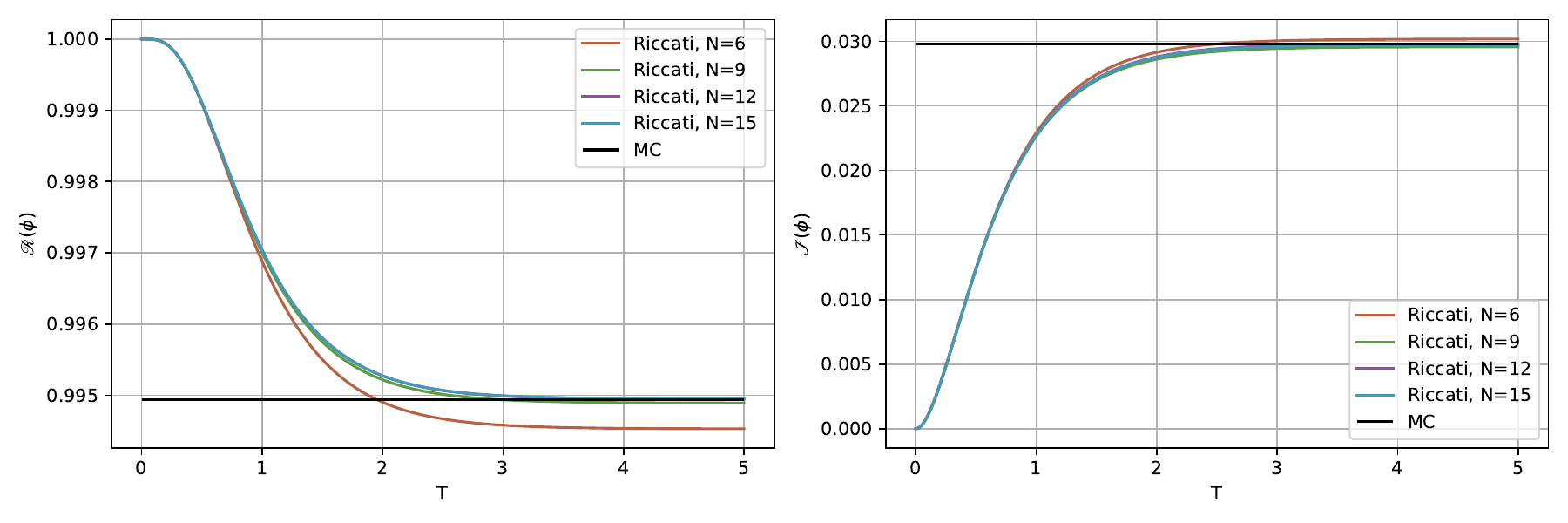}
    \caption{Real (on the left) and imaginary (on the right) parts of $\phi_T(\bell) = e^{\bpsi_T^\emptyword}$ for $\bell = \word{1}\conpow{4}$ and $\blambda = (1, 1)$ for different truncation orders $N$. The horizontal black line corresponds to the stationary value $\phi(\bell)$ estimated by the Monte Carlo method.}
    \label{fig:cf_plot}
    \end{center}
\end{figure}
\end{sqexample}
}

\section{Rough path view on the EFM signature}\label{sect:rough_paths}

This section is devoted to the construction of the EFM signature with rough path techniques. Namely, our goal is to define improper rough integrals over $(-\infty, T]$ in order to make sense of \eqref{eq:recursive_fm_sig_rough} when $s = -\infty$. We start by recalling the construction of the rough integral over finite intervals $[s, t]$ and the definition of the controlled paths required for it.

    \begin{definition}
    Let $\mathbf{X} = (\mathbb{X}^1, \mathbb{X}^2)\in \mathscr{C}^{\alpha}([s, t])$ for some $\alpha \in (\frac13, \frac12]$. We say that $(Y, Y') \in \mathcal{C}^{\alpha}([s, t])\times{\mathcal{C}^{\alpha}}([s, t])$ is controlled by $\mathbf{X}$ if the remainder term $R^Y$ defined by
    $$
    R^Y_{u, v} = \delta Y_{u, v} - Y'_u \otimes \mathbb{X}^1_{u, v},
    $$
    belongs to $\mathcal{C}^{2\alpha}([s, t])$. We write $(Y, Y')\in\mathscr{D}_{X}^{2\alpha}([s, t])$ and define the seminorm on $\mathscr{D}_{X}^{2\alpha}([s, t])$ by
    $$
    \|Y, Y'\|_{X, 2\alpha}= \|Y'\|_\alpha + \|R^Y\|_{2\alpha}. 
    $$
    \end{definition}

    The rough integral of $(Y, Y')$ with respect to $\mathbf{X}$ is then defined by
    \begin{equation}\label{eq:rough_int_def}
        Z_{s, t} = \int_s^t Y_u d\mathbf{X}_u := \lim_{|\mathcal{P}|\to 0}\sum_{[u, v]\in \mathcal{P}} (Y_u \otimes \mathbb{X}^1_{u, v} + Y'_u\otimes \mathbb{X}^{2}_{u, v}),
    \end{equation}
    where the limit is taken over all partitions $\mathcal{P}$ of the interval $[s, t]$.
    Rough integration theory (see, e.g., \cite[Theorem~4.10]{friz2020course}) provides the following bounds.

    \begin{theorem}\label{thm:rough_conv_finite_interval}
    Fix $s \leq t$ and $\mathbf{X} = (\mathbb{X}^1, \mathbb{X}^2)\in \mathscr{C}^{\alpha}([s, t])$ for some $\alpha \in (\frac13, \frac12]$. Let $(Y, Y') \in \mathscr{D}^{2\alpha}_{X}([s, t])$. Then the integral $Z$ given by \eqref{eq:rough_int_def} exists as an element of ${{\mathcal{C}}^{\alpha}}([s, t])$, and, for all $u, v \in [s, t]$ with \( 0 < v - u < 1 \),
        \begin{equation}\label{eq:conv_int_control}
            \left|\int_u^v Y_r\, d\mathbf{X}_r - (Y_u \otimes \mathbb{X}^1_{u, v} + Y'_u \otimes \mathbb{X}^{2}_{u, v}) \right| \leq C_{\alpha}(\|R^Y\|_{2\alpha}\|\mathbb{X}^{1}\|_{\alpha} + \|Y'\|_{\alpha}\|\mathbb{X}^{2}\|_{2\alpha})|v - u|^{3\alpha}.
        \end{equation}
        Moreover, $(Z, Y) \in \mathscr{D}^{2\alpha}_{X}([s, t])$, and one has the bound
        \begin{equation}
            \|Z, Y\|_{X, 2\alpha} \lesssim \|Y\|_\alpha + \|\mathbf{X}\|_\alpha(|Y_s'| + \|Y, Y'\|_{X, 2\alpha}).
        \end{equation}
    \end{theorem}

\subsection{Improper rough integration}
Let $\mathbf{X}^T = (\mathbb{X}^{T, 1}, \mathbb{X}^{T, 2})$ be the rough path lift of the process $X^T$ defined by \eqref{eq:volterra_lift}. The goal of this subsection is to define the improper rough integral against $\mathbf{X}^T$ over an infinite interval:
\begin{equation}\label{eq:improper_conv_int}
Z_t = \int_{-\infty}^tY_u\,d\mathbf{X}_u^T := \lim_{R \to -\infty}\int_R^t  Y_{u}\, d\mathbf{X}_u^{T}, \quad t \leq T.
\end{equation}
In view of Remark~\ref{rmk:lift-exp-decay}, it is natural to assume that $\mathbf{X}^T$ and the integrands belong to exponentially weighted spaces with negative exponents $-\bm\mu \prec 0$. We will also need the notion of exponentially weighted controlled paths.

\begin{definition}
Let $\mathbf{X}^T \in \mathscr{C}^{\alpha, -\bm\mu}((-\infty, T])$ for some $\alpha \in (\frac13, \frac12]$. We say that $(Y, Y') \in \mathcal{C}^{\alpha, -\bm\mu}_b((-\infty, T])\times\mathcal{C}^{\alpha, -\bm\mu}_b((-\infty, T])$, where $Y$ takes values in $(\R^d)^{\otimes n}$ and $Y'$ in $(\R^d)^{\otimes (n - 1)}$ (with $Y' \equiv 0$ if $n = 0$), is controlled by $\mathbf{X}^T$ if the remainder term $R^Y$ defined by
$$
R^Y_{u, v} = \delta Y_{u, v} - Y'_u \otimes \mathbb{X}^{T, 1}_{u, v},
$$
belongs to $\mathcal{C}^{2\alpha, -\bm\mu}((-\infty, T]^2)$. We write $(Y, Y')\in\mathscr{D}_{X^T}^{2\alpha, -\bm\mu}((-\infty, T])$ and define the seminorm on $\mathscr{D}_{X^T}^{2\alpha, -\bm\mu}((-\infty, T])$ by
$$
\|Y, Y'\|_{2\alpha, -\bm\mu} = \|Y'\|_{\alpha, -\bm\mu} + \|R^Y\|_{2\alpha, -\bm\mu}.
$$
\end{definition}

\begin{theorem}[Improper rough integration]\label{thm:improper_rough_int}
    Fix $\alpha \in (\frac13, \frac12]$ and $\bm\mu \in \R^d$ with $\bm\mu \succ 0$. Let $\mathbf{X}^T \in \mathscr{C}^{\alpha, -\bm\mu}((-\infty, T])$ and $(Y, Y') \in \mathscr{D}^{2\alpha, -\bm\mu}_{X^T}((-\infty, T])$, where $Y$ takes values in $(\R^d)^{\otimes n}$, and set $\theta := e^{-\bm\mu^{\min}} \in (0, 1)$. Then, for all $t \leq T$, the improper rough integral
    \eqref{eq:improper_conv_int}
    is well-defined, and, for all $R \leq t \leq T$,
    \begin{align}\label{eq:improper_convergence_control_int}
        \left|\D[T-t][-\bm\mu]\left(Z_t - \int_R^t Y_u\, d\mathbf{X}_u^T\right)\right| \lesssim \theta^{(n+1)(t - R)}\|\mathbf{X}^T\|_{\alpha, -\bm\mu}(\|Y\|_{\infty, -\bm\mu} + \|Y'\|_{\infty, -\bm\mu} + \|Y, Y'\|_{2\alpha, -\bm\mu}).
    \end{align}
    Moreover, $Z \in \mathcal{C}^{\alpha, -\bm\mu}_b((-\infty, T])$, the couple $(Z, Y)$ belongs to $\mathscr{D}^{2\alpha, -\bm\mu}_{X^T}((-\infty, T])$, and the following bounds hold:
    \begin{equation}\label{eq:improper_bound_1_int}
        \|Z\|_{\alpha, -\bm\mu} + \|Z\|_{\infty, -\bm\mu} \lesssim \|\mathbf{X}^T\|_{\alpha, -\bm\mu}(\|Y\|_{\infty, -\bm\mu} + \|Y'\|_{\infty, -\bm\mu} + \|Y, Y'\|_{2\alpha, -\bm\mu}),
    \end{equation}
    \begin{equation}\label{eq:improper_bound_2_int}
        \|Z, Y\|_{2\alpha, -\bm\mu} \lesssim \|\mathbf{X}^T\|_{\alpha, -\bm\mu}(\|Y'\|_{\infty, -\bm\mu} + \|Y, Y'\|_{2\alpha, -\bm\mu}) + \|Y, Y'\|_{2\alpha, -\bm\mu}.
    \end{equation}
    The implicit constants depend only on $\alpha$, $\bm\mu$ and $n$.
\end{theorem}
\begin{proof} The proof consists of three steps.
    \paragraph{Step 1. Existence of the limit \eqref{eq:improper_conv_int}.} 
    Let us denote
    $$
    Z^R_{t} = \int_{R}^t Y_u\, d\mathbf{X}_u^T, \quad R \leq t \leq T.
    $$
    Take $S < R \leq t$ such that $|R - S| < 1$. Since \( \D[T-t][-\bm\mu] \) is a homomorphism, we have
    $$
    \D[T-t][-\bm\mu](Z^S_{t} - Z^{R}_{t}) = \D[T-t][-\bm\mu]\int_{S}^{R}Y_u\, d\mathbf{X}_u^T = \int_{S}^{R}(\D[T-t][-\bm\mu]Y_u)\, d(\D[T-t][-\bm\mu]\mathbf{X}_{u}^T)
    $$
    By \eqref{eq:conv_int_control} and using \( |R - S| < 1 \), we have
    \begin{equation}\label{eq:cauchy_diff_int_0}
    \begin{aligned}
        |\D[T-t][-\bm\mu](Z^R_{t} - Z^{S}_{t})| &\leq |\D[T-t][-\bm\mu]Y_{S} \otimes \D[T-t][-\bm\mu]\mathbb{X}^{T, 1}_{S, R} + \D[T-t][-\bm\mu]Y_{S}' \otimes \D[T-t][-\bm\mu]\mathbb{X}^{T, 2}_{S, R}| 
        \\ & + C_{\alpha}\left(\|\D[T-t][-\bm\mu]R^Y\|_{2\alpha, [S, R]}\|\D[T-t][-\bm\mu]\mathbb{X}^{T, 1}\|_{\alpha, [S, R]} + \|\D[T-t][-\bm\mu]Y'\|_{\alpha, [S, R]}\|\D[T-t][-\bm\mu]\mathbb{X}^{T, 2}\|_{2\alpha, [S, R]}\right)
    \end{aligned}
    \end{equation}
    Notice that \( \D[T-t][-\bm\mu] = \D[t-S][\bm\mu]\D[T-S][-\bm\mu], \) so that the first group of terms in \eqref{eq:cauchy_diff_int_0} is bounded by 
    \begin{align*}
        &\|\D[t-S][\bm\mu]Y\|_{\infty, -\bm\mu}\|\D[t-S][\bm\mu]\mathbb{X}^{T, 1}\|_{\alpha, -\bm\mu} + \|\D[t-S][\bm\mu]Y'\|_{\infty, -\bm\mu}\|\D[t-S][\bm\mu]\mathbb{X}^{T, 2}\|_{2\alpha, -\bm\mu} \\
        &\quad\leq \theta ^{(n+1)(t-S)} \big(\|Y\|_{\infty, -\bm\mu}\|\mathbb{X}^{T, 1}\|_{\alpha, -\bm\mu} + \|Y'\|_{\infty, -\bm\mu}\|\mathbb{X}^{T, 2}\|_{2\alpha, -\bm\mu}\big)
    \end{align*}
    To bound the second group of terms in \eqref{eq:cauchy_diff_int_0}, we observe that for all \( (A_{u, v})_{u \leq v \leq T} \) taking values in \( (\mathbb{R}^d)^{\otimes m} \) and for all \( \beta > 0 \),
    \begin{align*}
        \|\D[T-t][-\bm\mu]A\|_{\beta, [S, R]} &\leq \theta^{m(t-S)}\|\D[T-S][-\bm\mu]A\|_{\beta, [S, R]} = \theta^{m(t-S)}\sup_{S \leq u < v \leq R}\dfrac{|\D[T-S][-\bm\mu]A_{u, v}|}{|v - u|^{\beta}} \\
        &\leq \theta^{m(t-S)}e^{m\bm\mu^{\max}}\sup_{S \leq u < v \leq R}\dfrac{|\D[T-u][-\bm\mu]A_{u, v}|}{|v - u|^{\beta}} \leq \theta^{m(t-S)}e^{m\bm\mu^{\max}} \|A\|_{\beta, -\bm\mu}
    \end{align*} 
    Applying the two inequalities above, we can further bound \eqref{eq:cauchy_diff_int_0} by
    \begin{equation}\label{eq:cauchy_diff_int}
    \begin{aligned}
        |\D[T-t][-\bm\mu](Z^R_{t} - Z^{S}_{t})| 
        &\lesssim_{\bm\mu} \theta^{(n+1)(t-S)}\big(\|Y\|_{\infty, -\bm\mu}\|\mathbb{X}^{T, 1}\|_{\alpha, -\bm\mu} + \|Y'\|_{\infty, -\bm\mu}\|\mathbb{X}^{T, 2}\|_{2\alpha, -\bm\mu} \\
        &\quad +C_{\alpha}e^{(n+1)\bm\mu^{\max}}(\|R^Y\|_{2\alpha, -\bm\mu}\|\mathbb{X}^{T, 1}\|_{\alpha, -\bm\mu} + \|Y'\|_{\alpha, -\bm\mu}\|\mathbb{X}^{T, 2}\|_{2\alpha, -\bm\mu})\big) \\
        &\lesssim_{\alpha, \bm\mu, n} \theta^{(n+1)(t-S)}\|\mathbf{X}^T\|_{\alpha, -\bm\mu}(\|Y\|_{\infty, -\bm\mu} + \|Y'\|_{\infty, -\bm\mu} + \|Y, Y'\|_{2\alpha, -\bm\mu})
    \end{aligned}
    \end{equation}

    For general $S < R \leq t$ with $|R - S| \geq 1$, there exist $\Delta \in [1/2, 1)$ and $m \in \N$ such that $R - S = m\Delta$. Hence, we can bound
    \begin{align}\label{eq:cauchy_diff_int_2}
    \begin{aligned}
            |\D[T-t][-\bm\mu](Z^R_{t} - Z^S_{t})| &\leq \sum_{k=1}^{m}|\D[T-t][-\bm\mu](Z^{R - (k-1)\Delta}_{t} - Z^{R - k\Delta}_{t})| \\ 
        &\lesssim_{\alpha, \bm\mu, n} \sum_{k = 1}^{\infty}\theta^{(n+1)(t-R + k\Delta)}\|\mathbf{X}^T\|_{\alpha, -\bm\mu}(\|Y\|_{\infty, -\bm\mu} + \|Y'\|_{\infty, -\bm\mu} + \|Y, Y'\|_{2\alpha, -\bm\mu}) \\
        &\leq \dfrac{\theta^{(n+1)(t-R+1/2)}}{1 - \theta^{(n+1)/2}}\|\mathbf{X}^T\|_{\alpha, -\bm\mu}(\|Y\|_{\infty, -\bm\mu} + \|Y'\|_{\infty, -\bm\mu} + \|Y, Y'\|_{2\alpha, -\bm\mu}),
    \end{aligned}
    \end{align}
    where we used \eqref{eq:cauchy_diff_int} in the second inequality, together with the bounds on \( \Delta \). Combining \eqref{eq:cauchy_diff_int} and \eqref{eq:cauchy_diff_int_2} yields the Cauchy property of $(Z^R_{t})_{R \leq t}$ and hence the existence of the limit in \eqref{eq:improper_conv_int}, which satisfies \eqref{eq:improper_convergence_control_int}.

    \paragraph{Step 2. Showing that $Z \in \mathcal{C}^{\alpha, -\bm\mu}_b((-\infty, T])$.}
    
    For all $s, t \in (-\infty, T]$ such that $0 < t - s < 1$, we apply \eqref{eq:conv_int_control}:
    \begin{align}
        |\D[T-s][-\bm\mu]\delta Z_{s, t}| &= \left|\D[T-s][-\bm\mu]\int_s^tY_u \, d\mathbf{X}_u^T\right| \\
        &\leq |\D[T-s][-\bm\mu](Y_s \otimes \mathbb{X}^{T, 1}_{s,t})| + |\D[T-s][-\bm\mu](Y'_s \otimes \mathbb{X}^{T, 2}_{s, t})| \\ &\quad+ C_{\alpha}(\|\D[T-s][-\bm\mu]R^Y\|_{2\alpha, [s, t]}\|\D[T-s][-\bm\mu]\mathbb{X}^{T, 1}\|_{\alpha, [s, t]} + \|\D[T-s][-\bm\mu]Y'\|_{\alpha, [s, t]}\|\D[T-s][-\bm\mu]\mathbb{X}^{T, 2}\|_{2\alpha, [s, t]})|t - s|^{3\alpha} \\
        &\leq \|Y\|_{\infty, -\bm\mu}\|\mathbb{X}^{T, 1}\|_{\alpha, -\bm\mu}|t-s|^\alpha + \|Y'\|_{\infty, -\bm\mu}\|\mathbb{X}^{T, 2}\|_{2\alpha, -\bm\mu}|t-s|^{2\alpha} \\
        &\quad+ C_{\alpha}e^{(n+1)\bm\mu^{\max}}(\|R^Y\|_{2\alpha, -\bm\mu}\|\mathbb{X}^{T, 1}\|_{\alpha, -\bm\mu} + \|Y'\|_{\alpha, -\bm\mu}\|\mathbb{X}^{T, 2}\|_{2\alpha, -\bm\mu})|t - s|^{3\alpha},
    \end{align}
    so that
    \begin{align}
        \|Z\|_{\alpha,-\bm\mu} \lesssim \|\mathbf{X}^T\|_{\alpha, -\bm\mu}(\|Y\|_{\infty, -\bm\mu} + \|Y'\|_{\infty, -\bm\mu} + \|Y, Y'\|_{2\alpha, -\bm\mu}).
    \end{align}
    The bound for $\|Z\|_{\infty, -\bm\mu}$ follows from \eqref{eq:improper_convergence_control_int} with \( R = t \). 

    \paragraph{Step 3. Showing that $(Z, Y) \in \mathscr{D}^{2\alpha, -\bm\mu}_{X^T}((-\infty, T])$.}
    For all $s < t \leq T$ such that $0 < t - s < 1$, we have
    \begin{align}
        |\D[T-s][-\bm\mu](\delta Y_{s,t})| &\leq |\D[T-s][-\bm\mu]Y'_s \otimes \D[T-s][-\bm\mu]\mathbb{X}^{T, 1}_{s, t}| + |\D[T-s][-\bm\mu]R^Y_{s, t}| \\
        &\leq \|Y'\|_{\infty, -\bm\mu}\|\mathbb{X}^{T, 1}\|_{\alpha, -\bm\mu}|t-s|^{\alpha} + \|R^Y\|_{2\alpha, -\bm\mu}|t-s|^{2\alpha},
    \end{align}
    so that
    \begin{equation}
        \|Y\|_{\alpha, -\bm\mu} \leq \|Y'\|_{\infty, -\bm\mu}\|\mathbb{X}^{T, 1}\|_{\alpha, -\bm\mu} + \|R^Y\|_{2\alpha, -\bm\mu} \leq \|\mathbf{X}^T\|_{\alpha, -\bm\mu}(\|Y'\|_{\infty, -\bm\mu} + \|Y, Y'\|_{2\alpha, -\bm\mu}) + \|Y, Y'\|_{2\alpha, -\bm\mu}.
    \end{equation}
    To bound $\|R^Z\|_{2\alpha, -\bm\mu}$, we again use \eqref{eq:conv_int_control} and proceed as in Step 2:
    \begin{align}
        |\D[T-s][-\bm\mu]R^Z_{s, t}| &= |\D[T-s][-\bm\mu]\delta Z_{s, t} - \D[T-s][-\bm\mu](Y_s \otimes \mathbb{X}^{T, 1}_{s,t})| = \left|\int_s^t\D[T-s][-\bm\mu]Y_u\, d(\D[T-s][-\bm\mu]\mathbf{X}_u^T) - \D[T-s][-\bm\mu]Y_s \otimes \D[T-s][-\bm\mu]\mathbb{X}^{T, 1}_{s,t}\right| \\
        &\leq |\D[T-s][-\bm\mu](Y'_s \otimes \mathbb{X}^{T, 2}_{s, t})| + C_{\alpha}e^{(n+1)\bm\mu^{\max}}(\|R^Y\|_{2\alpha, -\bm\mu}\|\mathbb{X}^{T, 1}\|_{\alpha, -\bm\mu} + \|Y'\|_{\alpha, -\bm\mu}\|\mathbb{X}^{T, 2}\|_{2\alpha, -\bm\mu})|t - s|^{3\alpha} \\
        &\lesssim \|Y'\|_{\infty, -\bm\mu}\|\mathbb{X}^{T, 2}\|_{2\alpha, -\bm\mu}|t-s|^{2\alpha} + (\|R^Y\|_{2\alpha, -\bm\mu}\|\mathbb{X}^{T, 1}\|_{\alpha, -\bm\mu} + \|Y'\|_{\alpha, -\bm\mu}\|\mathbb{X}^{T, 2}\|_{2\alpha, -\bm\mu})|t - s|^{3\alpha},
    \end{align}
    and we obtain
    \begin{align}
        \|R^Z\|_{2\alpha, -\bm\mu} \lesssim \|\mathbf{X}^T\|_{\alpha, -\bm\mu}(\|Y'\|_{\infty, -\bm\mu} + \|Y, Y'\|_{2\alpha, -\bm\mu}).
    \end{align}
    This finishes the proof.
\end{proof}

\subsection{Proof of Theorem~\ref{thm:fm_sig_existence}}\label{corollary:fm_sig_s_inf}
\begin{proof}
    Consider first the finite case \( -\infty < s \leq T \).
    Set \( \bm\mu := \blambda - \brho \).
    We construct the signature \( \mathbb{X}^T_{s, \cdot} \) recursively 
    \[
        \mathbb{X}^{T, n+1}_{s, t} = \int_{s}^t \mathbb{X}^{T, n}_{s, u} \, d\mathbf{X}_{u}^T,
    \]
    as the rough integral of \( (\mathbb{X}^{T, n}_{s, \cdot}, \mathbb{X}^{T, n-1}_{s, \cdot}) \) given by Theorem~\ref{thm:rough_conv_finite_interval}, with the initialization \( (\mathbb{X}^{T, 0}_{s, \cdot}, 0) \), where \( \mathbb{X}^{T, 0}_{s, \cdot} \equiv 1 \). The EFM signature is then defined by
    \begin{equation}\label{eq:efm_lift_link}
        \mathbb{X}^{\blambda}_{s, t} := \D[T-t][-\blambda]\mathbb{X}^{T}_{s, t}.
    \end{equation}
    The finiteness of the Hölder norm in the case \( s > -\infty \) follows from the fact that for all \( u, v \in [s, T] \) such that \( 0 < v - u < 1 \),
    \begin{equation}\label{eq:holder-efm-tmp}
        \mathbb{X}^{\blambda, n}_{s, v} - \mathbb{X}^{\blambda, n}_{s, u} = \D[T-v][-\blambda]\mathbb{X}^{T, n}_{s, v} - \D[T-u][-\blambda]\mathbb{X}^{T, n}_{s, u} = \D[T-u][-\blambda](\mathbb{X}^{T, n}_{s, v} - \mathbb{X}^{T, n}_{s, u}) + (\mathrm{Id} - \D[v-u][-\blambda])\D[T-v][-\blambda]\mathbb{X}^{T, n}_{s, v},
    \end{equation}
    where the first term is controlled by \( \|\mathbb{X}^{T, n}_{s, \cdot}\|_{\alpha, [s, T]} \) and the second one by \( \|\mathbb{X}^{T, n}_{s, \cdot}\|_{\infty, [s, T]} \) together with the derivative of \( h \mapsto \D[h][-\blambda] \).

    For \( s = -\infty \), we define 
    \[ 
    \mathbb{X}^{T, n+1}_{-\infty, t} := \int_{-\infty}^t \mathbb{X}^{T, n}_{-\infty, u}\, d\mathbf{X}^T_u, 
    \] 
    recursively, by applying Theorem~\ref{thm:improper_rough_int} to \( (\mathbb{X}^{T, n}_{-\infty, \cdot}, \mathbb{X}^{T, n-1}_{-\infty, \cdot}) \), which belongs to \( \mathscr{D}^{2\alpha, -\bm\mu}_{X^T}((-\infty, T]) \) by the induction hypothesis; Theorem~\ref{thm:improper_rough_int} itself then ensures this membership at the next level. The initialization is again given by \( (\mathbb{X}^{T, 0}_{-\infty, \cdot}, 0) \) with \( \mathbb{X}^{T, 0}_{-\infty, \cdot} \equiv 1 \), which trivially belongs to \( \mathscr{D}^{2\alpha, -\bm\mu}_{X^T}((-\infty, T]) \). The EFM signature \( \mathbb{X}^{\blambda}_{t} := \mathbb{X}^{\blambda}_{-\infty, t} \) is then defined by \eqref{eq:efm_lift_link} with \( s = -\infty \).
    Note that
    \[
        \|\mathbb{X}^{\blambda, n}_{\cdot}\|_{\infty, \brho} = \sup_{t \leq T}|\D[T-t][\brho]\mathbb{X}^{\blambda, n}_{t}| = \sup_{t \leq T}|\D[T-t][-\bm\mu]\mathbb{X}^{T, n}_{-\infty, t}| = \|\mathbb{X}^{T, n}_{-\infty, \cdot}\|_{\infty, -\bm\mu},
    \]
    and that, applying $\D[T-u][\brho]$ to \eqref{eq:holder-efm-tmp} with \( s = -\infty \), we have, for all \( u, v \in (-\infty, T]\) such that \(0 < v - u < 1 \),
    \[
        \D[T - u][\brho](\mathbb{X}^{\blambda, n}_{-\infty, v} - \mathbb{X}^{\blambda, n}_{-\infty, u}) = \D[T - u][-\bm\mu](\mathbb{X}^{T, n}_{-\infty, v} - \mathbb{X}^{T, n}_{-\infty, u}) + \D[v-u][\brho](\mathrm{Id} - \D[v-u][-\blambda])\D[T - v][-\bm\mu]\mathbb{X}^{T, n}_{-\infty, v},
    \]
    so that
    \begin{equation}\label{eq:norms-lift-efm-equiv}
         \|\mathbb{X}^{\blambda, n}_{-\infty, \cdot}\|_{\alpha, \brho} + \|\mathbb{X}^{\blambda, n}_{-\infty, \cdot}\|_{\infty, \brho} \lesssim_{\blambda, n} \|\mathbb{X}^{T, n}_{-\infty, \cdot}\|_{\alpha, -\bm\mu} + \|\mathbb{X}^{T, n}_{-\infty, \cdot}\|_{\infty,-\bm\mu} < \infty,
    \end{equation}
   where the finiteness of the right-hand side follows from \eqref{eq:improper_bound_1_int}.
\end{proof}

\subsection{Fading memory property: proof of Theorem~\ref{thm:rough_fm}}\label{sect:rough_fm_proof}
    To prove the continuity at a fixed ${\mathbf{X}^T}$, we can suppose that there exists a constant $M$ such that
    \begin{equation}\label{eq:norm_ball}
        \|\mathbf{X}^T\|_{\alpha, -\bm\mu} \leq M,\quad \|\mathbf{Y}^T\|_{\alpha, -\bm\mu} \leq M,
    \end{equation}
    since $\|\mathbf{X}^T\|_{\alpha, -\bm\mu} = \varrho_{\alpha, -\bm\mu}(\mathbf{0}, \mathbf{X}^T)$. 

    Note that for $s > -\infty$, the result follows from the stability result for rough integration (see \cite[Theorem 4.17]{friz2020course}). Thus, it is sufficient to consider only the case $s = -\infty$.
    
    We proceed by induction in $n$. For $n = 0$, $\mathbb{X}_{t}^{\blambda, n} \equiv \mathbb{Y}_{t}^{\blambda, n} \equiv 1$, and the statement holds trivially. Suppose that ${\mathbf{X}^T} \mapsto \mathbb{X}_{t}^{\blambda, k}$ is continuous for all $k \leq n$ and $t \leq T$. 
    
    Fix $t \leq T$ and $\varepsilon > 0$. We will show that one can choose $r > 0$ such that $|\mathbb{X}_{t}^{\blambda, n+1} - \mathbb{Y}_{t}^{\blambda, n+1}| < \varepsilon$ for all $\mathbf{X}^T, \mathbf{Y}^T$ satisfying $\varrho_{\alpha, -\bm\mu}(\mathbf{X}^T, \mathbf{Y}^T) < r$.
    Recall that
    \begin{equation}\label{eq:recursive_int_fm_sig}
        \mathbb{X}_{t}^{\blambda, n+1} = \D[T-t][-\blambda]\int_{-\infty}^t\mathbb{X}_{u}^{T, n}d\mathbf{X}_u^T = \lim_{R \to -\infty} \D[T-t][-\blambda]\int_{R}^t\mathbb{X}_{u}^{T, n}d\mathbf{X}_u^T,
    \end{equation}
    where $\mathbb{X}^{T, k}_{u} := \mathbb{X}^{T, k}_{-\infty, u}$, $k \geq 0$, denote the levels of the signature of $\mathbf{X}^T$ constructed in Subsection~\ref{corollary:fm_sig_s_inf}, and similarly for $\mathbf{Y}^T$.

    The bound \eqref{eq:improper_convergence_control_int} applied to $(\mathbb{X}^{T, n}_{\cdot}, \mathbb{X}^{T, n-1}_{\cdot}) \in \mathscr{D}^{2\alpha, -\bm\mu}_{X^T}((-\infty, T])$, together with the identity $\D[T-t][-\blambda] = \D[T-t][-\brho]\D[T-t][-\bm\mu]$, yields
    \begin{equation}
        \left|\mathbb{X}_{t}^{\blambda, n+1} - \D[T-t][-\blambda]\int_R^t\mathbb{X}_{u}^{T, n}\,d\mathbf{X}_u^T\right|
        \lesssim \theta^{(n+1)(t-R)}\|\mathbf{X}^T\|_{\alpha, -\bm\mu}(\|\mathbb{X}_{\cdot}^{T, n}\|_{\infty, -\bm\mu} + \|\mathbb{X}_{\cdot}^{T, n - 1}\|_{\infty, -\bm\mu} + \|\mathbb{X}_{\cdot}^{T, n}, \mathbb{X}_{\cdot}^{T, n-1}\|_{2\alpha, -\bm\mu}),
    \end{equation}
    where $\bm\mu := \blambda - \brho$ and $\theta := e^{-\bm\mu^{\min}}$.
    Applying recursively \eqref{eq:improper_bound_1_int}, \eqref{eq:improper_bound_2_int}, and using \eqref{eq:norm_ball}, we can show that
    \begin{equation}
        \left|\mathbb{X}_{t}^{\blambda, n+1} - \D[T-t][-\blambda]\int_R^t\mathbb{X}_{u}^{T, n}\,d\mathbf{X}_u^T\right| \lesssim_M \theta^{(n+1)(t-R)}.
    \end{equation}
    This shows that the convergence \eqref{eq:recursive_int_fm_sig} is uniform in ${\mathbf{X}^T}$ satisfying \eqref{eq:norm_ball}, so that we can choose $R < {t}$ such that
    \begin{equation}
        \left|\mathbb{X}_{t}^{\blambda, n+1} - \D[T-t][-\blambda]\int_R^t\mathbb{X}_{u}^{T, n}\,d\mathbf{X}_u^T\right| < \frac{\varepsilon}{4},
    \end{equation}
    for all ${\mathbf{X}^T}$ satisfying \eqref{eq:norm_ball}. It remains to control the difference
    \begin{equation}\label{eq:diff_to_control}
        \left|\D[T-t][-\blambda]\left(\int_R^t\mathbb{X}_{u}^{T, n}\,d\mathbf{X}_u^T - \int_R^t\mathbb{Y}_{u}^{T, n}\,d\mathbf{Y}_u^T\right)\right|.
    \end{equation}
    We note that $\mathbb{X}_{u}^{T, n}$ satisfies Chen's identity
    \begin{equation}
        \mathbb{X}_{u}^{T, n} = \sum_{k=0}^n\mathbb{X}_{R}^{T, k} \otimes \mathbb{X}_{R, u}^{T, n - k}, \quad R \leq u \leq T,
    \end{equation}
    which follows by induction on $n$ from \eqref{eq:recursive_fm_sig_rough} and the additivity of the (improper) rough integral,
    so that
    \begin{align}
        \int_R^t\mathbb{X}_{u}^{T, n}\,d\mathbf{X}_u^T &= \sum_{k=0}^n \int_R^t\mathbb{X}_{R}^{T, k} \otimes \mathbb{X}_{R, u}^{T, n - k}\,d\mathbf{X}_u^T = \sum_{k=0}^n \mathbb{X}_{R}^{T, k} \otimes \int_R^t\mathbb{X}_{R, u}^{T, n - k}\,d\mathbf{X}_u^T = \sum_{k=0}^n \mathbb{X}_{R}^{T, k} \otimes \mathbb{X}_{R, t}^{T, n - k + 1}.
    \end{align}
    Hence, since $\D[T-t][-\blambda]$ is a bounded linear operator on $(\R^d)^{\otimes (n+1)}$, it is sufficient to control
    \begin{align}\label{eq:fm_sig_diff}
        \mathbb{X}_{R}^{T, k} \otimes \mathbb{X}_{R, t}^{T, n - k + 1} - \mathbb{Y}_{R}^{T, k} \otimes \mathbb{Y}_{R, t}^{T, n - k + 1} =\mathbb{X}_{R}^{T, k} \otimes (\mathbb{X}_{R, t}^{T, n - k + 1} - \mathbb{Y}_{R, t}^{T, n - k + 1}) + (\mathbb{X}_{R}^{T, k} - \mathbb{Y}_{R}^{T, k}) \otimes \mathbb{Y}_{R, t}^{T, n - k + 1},
    \end{align}
    for all $k = 0, \ldots, n$.
    In the first term, we have
    \begin{align}
        |\mathbb{X}_{R}^{T, k}| = |\D[T - R][\bm\mu]\D[T - R][-\bm\mu]\mathbb{X}_{R}^{T, k}| \leq \|\mathbb{X}_{\cdot}^{T, k}\|_{\infty, -\bm\mu} \lesssim_M 1,
    \end{align}
     where we recursively applied \eqref{eq:improper_bound_1_int} and \eqref{eq:improper_bound_2_int}. The second part of the first term is bounded by
    \begin{equation}
        |\mathbb{X}_{R, t}^{T, n - k + 1} - \mathbb{Y}_{R, t}^{T, n - k + 1}| \lesssim_{R, t} \|\mathbb{X}_{R, \cdot}^{T, n - k + 1} - \mathbb{Y}_{R, \cdot}^{T, n - k + 1}\|_{\alpha, [R, t]}.
    \end{equation}
    By the stability result for rough integration \cite[Theorem 4.17]{friz2020course}, applied recursively, this norm tends to $0$ as $\varrho_{\alpha, -\bm\mu}(\mathbf{X}^T, \mathbf{Y}^T) \to 0$, since $\|A\|_{\beta, [R, t]} \leq \|A\|_{\beta, -\bm\mu}$ for all $A$ and $\beta > 0$, and hence the rough path distance between $\mathbf{X}^T$ and $\mathbf{Y}^T$ on $[R, t]$ is bounded by $\varrho_{\alpha, -\bm\mu}(\mathbf{X}^T, \mathbf{Y}^T)$.

    In the second term, the second multiplier is bounded by the standard scaling argument
    \begin{equation}
        |\mathbb{Y}_{R, t}^{T, n - k + 1}| \lesssim_{R, t}  \left(\|\mathbb{Y}^{T, 1}\|_{\alpha, [R, t]} + \sqrt{\|\mathbb{Y}^{T, 2}\|_{2\alpha, [R, t]}}\right)^{n - k + 1} \lesssim_{R,T, M} 1,
    \end{equation}
    while the first one
    $$
    |\mathbb{X}_{R}^{T, k} - \mathbb{Y}_{R}^{T, k}| \to 0,
    $$
    as $\varrho_{\alpha, -\bm\mu}(\mathbf{X}^T, \mathbf{Y}^T) \to 0$ by the induction hypothesis, since $\mathbb{X}_{R}^{T, k} = \D[T - R]\mathbb{X}_{R}^{\blambda, k}$ by \eqref{eq:efm_lift_link}. Thus, one can choose $r > 0$ such that
    \begin{equation}
        \left|\D[T-t][-\blambda]\left(\mathbb{X}_{R}^{T, k} \otimes \mathbb{X}_{R, t}^{T, n - k + 1} - \mathbb{Y}_{R}^{T, k} \otimes \mathbb{Y}_{R, t}^{T, n - k + 1}\right)\right| < \dfrac{\varepsilon}{2(n + 1)},
    \end{equation}
    for $\mathbf{X}^T, \mathbf{Y}^T$ such that $\varrho_{\alpha, -\bm\mu}(\mathbf{X}^T, \mathbf{Y}^T) < r$. This implies that \eqref{eq:diff_to_control} is less than $\dfrac{\varepsilon}{2}$ and concludes the proof.

\appendix

\section{Proof of Lemma~\ref{lem:linear_path}}\label{sect:linear_path_proof}
The computation reduces to evaluating the iterated integral \eqref{eq:interated_exp_int}. After a change of variables 
$$u_j = t - \sum\limits_{k = 1}^{n - j + 1}\tau_k,\quad j = 1, \ldots, n,$$ 
this integral becomes
    \begin{equation}
        \int_{\tau_j \geq 0,\ \sum\limits_{j=1}^n \tau_j \leq t - s} e^{-\sum\limits_{j=1}^n\blambda^{i_j}\sum\limits_{k = 1}^{n - j + 1}\tau_k}\,d\tau_1\ldots d\tau_n = \int_{\tau_j \geq 0,\ \sum\limits_{j=1}^n \tau_j \leq t - s} e^{-\sum\limits_{k=1}^n\left(\sum\limits_{j = 1}^{n - k + 1}\blambda^{i_j}\right)\tau_k}\,d\tau_1\ldots d\tau_n.
    \end{equation}
    Setting $\mu_k = \sum\limits_{j = 1}^{n - k + 1}\blambda^{i_j}$ for $k = 1, \ldots, n$ and introducing $\mu_{n + 1} = 0$, as well as $\tau_{n + 1} = (t - s) - \sum\limits_{k=1}^n \tau_k$, the integral becomes one over the simplex $C := \left\{\tau = (\tau_1, \ldots, \tau_{n+1}) \in \R^{n+1}\colon\ \sum\limits_{k=1}^{n+1}\tau_k = t - s\right\}$:
    \begin{equation}
        \int_{C} e^{-\sum\limits_{k=1}^{n+1}\mu_k\tau_k}\,d\tau.
    \end{equation}
    By induction, this integral can be written as the convolution of exponential functions $f_k(t) = e^{-\mu_k t}$:
    \begin{equation}
        \int_{C} e^{-\sum\limits_{k=1}^{n+1}\mu_k\tau_k}\,d\tau = (f_1 \star \ldots \star f_{n+1})(t - s).
    \end{equation}
    The Laplace transform of each $f_k$ is $L[f_k](w) = \dfrac{1}{w + \mu_k}$, so
    \begin{equation}
        L\left[\int_{C} e^{-\sum\limits_{k=1}^{n+1}\mu_k\tau_k}\,d\tau\right](w) = \prod\limits_{k=1}^{n+1}\dfrac{1}{w + \mu_k} = \sum\limits_{k = 1}^{n+1}\dfrac{c_k}{w + \mu_k},
    \end{equation}
    where the coefficients $c_k = \prod\limits_{l \neq k}\dfrac{1}{\mu_l - \mu_k}$ come from the partial fraction decomposition, valid when all $\mu_k$ are distinct. Applying the inverse Laplace transform yields
    \begin{equation}
         \int_{C} e^{-\sum\limits_{k=1}^{n+1}\mu_k\tau_k}\,d\tau = \sum_{k=1}^{n+1}c_ke^{-\mu_k (t-s)}.
    \end{equation}
    A straightforward reindexing $k \mapsto n + 1 - k$ concludes the proof of the first formula. The second one is obtained by taking the limit $(t - s)\to +\infty$.

\section{Proof of Lemma~\ref{lemma:L1_bound}}\label{app:lemma_l1_proof}
\begin{proof}
We prove the lemma only in the case of one-dimensional Brownian motion $d = 1$. The multidimensional generalization is straightforward. We start by observing that  
    \begin{equation}\label{eq:l2_norm_series}
        \E[\|\ssig[0, t]\|_1] = \sum_{n \geq 0} \sum_{|\word{v}| = n} \E[|\sig[0, t]^{\blambda, \word{v}}|] \leq \sum_{n \geq 0} \sum_{|\word{v}| = n} \sqrt{\E[|\sig[0, t]^{\blambda, \word{v}}|^2]}.
    \end{equation}  
    We first recursively bound $\E[|\sig[0, t]^{\blambda, \word{v}}|^2]$ for each $\word{v}$ and then prove the convergence of the series in \eqref{eq:l2_norm_series}.
    
    For a word $\word{v}$ such that $|\word{v}| = n > 0$, we consider three cases. In the first case, where $\word{v} = \word{v'0}$, apply Hölder's inequality:
    \begin{equation}
        \E\left[\left(\int_{0}^t e^{-\blambda^{\word{v}}(t-s)} \sig[0, s]^{\blambda, \word{v'}}\,ds\right)^2\right]
        \leq \int_{0}^t e^{-\blambda^{\word{v}}(t-s)} \E\left[|\sig[0, s]^{\blambda, \word{v'}}|^2\right]\,ds
        \int_{0}^t e^{-\blambda^{\word{v}}(t-s)}\,ds
        \leq \dfrac{1}{(\blambda^{\word{v}})^2} \sup_{s \geq 0} \E\left[|\sig[0, s]^{\blambda, \word{v'}}|^2\right].
    \end{equation}

    For $\word{v} = \word{v'1}$, with $\word{v'}$ not ending in $\word{1}$, the Stratonovich integral coincides with the Itô integral, and we have
    \begin{equation}
    \E\left[\left(\int_{0}^te^{-\blambda^{\word{v}}(t-s)}\sig[0, s]^{\blambda, \word{v'}}\circ dW_s\right)^2\right] = \int_{0}^te^{-2\blambda^{\word{v}}(t-s)}\E\left[|\sig[0, s]^{\blambda, \word{v'}}|^2\right]\,ds \leq \dfrac{1}{2\blambda^{\word{v}}} \sup_{s \geq 0} \E\left[|\sig[0, s]^{\blambda, \word{v'}}|^2\right]. 
    \end{equation}
    Finally, if $\word{v} = \word{v''11}$, then
    \begin{align}
    \E\left[\left(\int_{0}^te^{-\blambda^{\word{v}}(t-s)}\sig[0, s]^{\blambda, \word{v'}}\circ dW_s\right)^2\right] &\leq {2}\E\left[\left(\int_{0}^te^{-\blambda^{\word{v}}(t-s)}\sig[0, s]^{\blambda, \word{v'}} dW_s\right)^2\right] + {\dfrac{1}{2}}\E\left[\left(\int_{0}^te^{-\blambda^{\word{v}}(t-s)}\sig[0, s]^{\blambda, \word{v''}} ds\right)^2\right] \\ 
    &\leq \dfrac{1}{\blambda^{\word{v}}} \sup_{s \geq 0} \E\left[|\sig[0, s]^{\blambda, \word{v'}}|^2\right] + \dfrac{1}{{2}(\blambda^{\word{v}})^2} \sup_{s \geq 0} \E\left[|\sig[0, s]^{\blambda, \word{v''}}|^2\right], 
    \end{align}
    where the integrals are bounded as in the first two cases.
    
    Combining these three bounds and noting that the bound does not depend on $t$, we obtain
    \begin{align}\label{eq:recursive_bound}
        \sup_{t \geq 0} \E\left[|\sig[0, t]^{\blambda, \word{v}}|^2\right]
        &\leq \left(\dfrac{1}{(\blambda^{\word{v}})^2} \lor \dfrac{1}{\blambda^{\word{v}}} \right) \sup_{t \geq 0} \E\left[|\sig[0, t]^{\blambda, \word{v'}}|^2\right] + \dfrac{1}{{2}(\blambda^{\word{v}})^2} \sup_{s \geq 0} \E\left[|\sig[0, s]^{\blambda, \word{v''}}|^2\right]
        \\
        &\leq \dfrac{C_{\blambda}}{n} \sup_{t \geq 0} \E\left[|\sig[0, t]^{\blambda, \word{v'}}|^2\right] + \dfrac{C_{\blambda}}{n^2} \sup_{t \geq 0} \E\left[|\sig[0, t]^{\blambda, \word{v''}}|^2\right],
    \end{align}
    since $\blambda^{\word{v}} \geq n \min\limits_{i \in \{0, 1\}} \blambda^i$, with a constant $C_{\blambda} \geq 1$ depending only on $\blambda$.

Applying \eqref{eq:recursive_bound} recursively, we obtain
\begin{equation}
    \sup_{t \geq 0} \E\left[|\sig[0, t]^{\blambda, \word{v}}|^2\right] \leq \dfrac{(2C_{\blambda})^n}{n!}.
\end{equation}

Finally, noting that there are $2^n$ terms in \eqref{eq:l2_norm_series} corresponding to words of length $n$, and taking the supremum, we obtain
\begin{equation}
    \sup_{t \geq 0} \E[\|\ssig[0, t]\|_1] \leq \sum_{n \geq 0} 2^n \sqrt{\dfrac{(2C_{\blambda})^n}{n!}} < +\infty,
\end{equation}
which completes the proof of the second bound. The first one can be shown analogously.
\end{proof}

\section{Numerical scheme for infinite-dimensional ODEs}\label{sect:numerical_ode}
In this section, we describe a heuristic numerical algorithm for the solution of an infinite-dimensional ODE for $\bpsi_t \in \eTA$:
\begin{equation}\label{eq:inf_dim_ode}
\begin{cases}
    \dot\bpsi_t = -\G\bpsi_t + F(\bpsi_t), \quad t \geq 0,\\
    \bpsi_0 = \bell,
\end{cases}
\end{equation}
which was used to solve the Riccati equation in Section~\ref{sec:cf_riccati}.

By Proposition \ref{prop:variation_of_c}, the equation \eqref{eq:inf_dim_ode} can be rewritten in the integral form
\begin{equation}
    \bpsi_t = \D[{t}](\bell) + \int_0^t\D[{t - s}]F(\bpsi_s)\,ds.
\end{equation}
We denote by $\hat\bpsi$ an approximated solution truncated at order $N$ and corresponding to the discretization step $h$. We also denote by $\hat F(\psi)$ the value of $F(\psi)$ truncated at the same order $N$.
This allows us to construct a semi-integrated scheme for numerical resolution of \eqref{eq:inf_dim_ode}:
\begin{align}
    \bpsi_{t + h} = \D[{h}](\bpsi_t) + \int_t^{t + h}\D[{t + h - s}]F(\bpsi_s)\,ds \approx  \D[{h}](\bpsi_t) + \int_t^{t+h}\D[{t + h - s}]\,ds F(\bpsi_t),
\end{align}
so that
\begin{equation}
    \hat\bpsi_{t + h} = \D\hat\bpsi_t + \mathcal{C}_h^\lambda \hat F(\hat\bpsi_t),
\end{equation}
where the operator $\mathcal{C}_h\colon \eTA \to\eTA$ is defined by
\begin{equation}
    \mathcal{C}_h\colon \bell = \sum_{n\geq 0}\sum_{|\word{v
    }| = n}\bell^{\word{v}} \mapsto \mathcal{C}_h\bell = \sum_{n\geq 0}\sum_{|\word{v
    }| = n}\left(\frac{1 - e^{-\blambda^{\word{v}} h}}{\blambda^{\word{v}}}\right)\bell^{\word{v}}
\end{equation}
with convention $\left(\frac{1 - e^{-\blambda^{\word{v}} h}}{\blambda^{\word{v}}}\right)\bigg|_{\blambda^{\word{v}} = 0} = h$.

To improve the numerical stability, we use the predictor-corrector scheme:
\begin{align}
    \hat\bpsi_{t + h}^P &= \D\hat\bpsi_t + \mathcal{C}_h^\lambda F(\hat\bpsi_t), \\
    \hat\bpsi_{t + h} &= \D\hat\bpsi_t + \mathcal{C}_h^\lambda \left(\frac{F(\hat\bpsi_t) + F(\hat\bpsi_{t + h}^P)}{2}\right). \\
\end{align}


\end{document}